\documentclass[12pt, leqno]{amsart}

\usepackage{graphicx}
\usepackage{amssymb,amsmath,amsthm,amscd}
\usepackage{mathrsfs}
\usepackage{lscape}
\usepackage{enumerate}
\usepackage[usenames,dvipsnames]{color}
\usepackage[colorlinks=true, pdfstartview=FitV,
 linkcolor=blue,citecolor=blue,urlcolor=blue]{hyperref}
\usepackage{tikz}

\usepackage[all]{xy}
\allowdisplaybreaks[3]

\numberwithin{equation}{section}

\newenvironment{red}{\relax\color{red}}{\relax}
\newenvironment{blue}{\relax\color{blue}}{\hspace*{.5ex}\relax}

\newcommand{\ber}{\begin{red}}
\newcommand{\er}{\end{red}}
\newcommand{\beb}{\begin{blue}}
\newcommand{\eb}{\end{blue}}

\theoremstyle{plain}
\newtheorem{lemma}{Lemma}[section]
\newtheorem{proposition}[lemma]{Proposition}
\newtheorem{theorem}[lemma]{Theorem}
\newtheorem{algorithm}[lemma]{Algorithm}

\newtheorem{corollary}[lemma]{Corollary}

\theoremstyle{definition}
\newtheorem{remark}[lemma]{Remark}
\newtheorem{example}[lemma]{Example}

\newtheorem{definition}[lemma]{Definition}
\newtheorem{question}[lemma]{Question}
\newtheorem{convention}[lemma]{Convention}

\newcommand{\soc}{{\operatorname{soc}}}

\newcommand{\srt}[2]{{}_{\langle #1,#2 \rangle}}
\newcommand{\g}{\mathfrak{g}}

\newcommand{\C}{\mathbb{C}}
\newcommand{\Z}{\mathbb{Z}}
\newcommand{\N}{\mathsf{N}}
\newcommand{\n}{\mathfrak{n}}

\newcommand{\W}{\mathsf{W}}
\newcommand{\cmA}{\mathsf{A}}
\newcommand{\seteq}{\mathbin{:=}}

\newcommand{\al}{\alpha}

\newcommand{\bi}{\mathbf{i}}

\newcommand{\lan}{\langle}
\newcommand{\ran}{\rangle}
\newcommand{\ko}{ \textbf{k}}
\newcommand{\rev}{{\rm rev}}

\newcommand{\ii}{ \textbf{\textit{i}}}
\newcommand{\jj}{ \textbf{\textit{j}}}
\newcommand{\kk}{ \textbf{\textit{k}}}

\newcommand{\be}{\beta}
\newcommand{\ga}{\gamma}
\newcommand{\ve}{\varepsilon}

\newcommand{\PR}{\Phi^+}

\newcommand{\redex}{{\widetilde{w}}}
\newcommand{\redez}{{\widetilde{w}_0}}
\newcommand{\um}{\underline{m}}

\newcommand{\us}{\underline{s}}

\newcommand{\up}{\underline{p}}

\newcommand{\tb}{\mathtt{b}}
\newcommand{\rl}{\mathsf{Q}}
\newcommand{\wt}{{\rm wt}}

\newcommand{\rds}{{\rm rds}}
\newcommand{\gdist}{{\rm gdist}}

\newcommand{\lf}{[\hspace{-0.3ex}[}
\newcommand{\rf}{]\hspace{-0.3ex}]}

\newcommand{\Qg}{{Q^{\gets}}}
\newcommand{\Qn}{{Q^{\gets n}}}
\newcommand{\Qnn}{{Q^{\gets n+1}}}

\newcommand{\wUp}{ \widehat{\Upsilon}}

\newcommand*\ov[1]{\overline{#1}}

\newlength{\mylength}
\title[Twisted Coxeter elements and folded AR-quivers: II]{Twisted Coxeter elements and folded AR-quivers via Dynkin diagram automorphisms: II}

\author[S.-j. Oh, U. Suh]{Se-jin Oh$^\dagger$,  Uhi Rinn  Suh$^\ddagger$}

\address{Department of Mathematics, Ewha Woman's University, Seoul 120-750, Korea}
\email{sejin092@gmail.com}
\address{Department of Mathematical Sciences and Research Institute of Mathematics, Seoul National University,  Seoul 151-747, Korea}
\email{uhrisu1@snu.ac.kr}

\thanks{$^\dagger$ This work was supported by  NRF Grant \# 2016R1C1B2013135.}
\thanks{$^\ddagger$ This work was supported by BK21 PLUS SNU University Mathematical Sciences Division and NRF Grant \# 2016R1C1B1010721.}


\date{\today}

\begin{document}

\begin{abstract}
As a continuation of the paper \cite{OS16}, we find a combinatorial interpretation of Dorey's rule for type $C_n$ via twisted Auslander-Reiten quivers (AR-quivers) of type $D_{n+1}$, which are combinatorial
AR-quivers related to certain Dynkin diagram automorphisms.
 Combinatorial properties of
twisted AR-quivers are useful to understand not only Dorey's rule but also other notions in the representation
theory of the quantum affine algebra $U_q'(C_n^{(1)})$ such as
denominator formulas. In addition, unlike twisted adapted classes of type $A_{2n-1}$ in \cite{OS16}, we show twisted AR-quivers of type $D_{n+1}$ consist of the  cluster point called twisted adapted cluster point. Hence, by introducing  new combinatorial objects called
twisted Dynkin quivers of type $D_{n+1}$, we give one to one correspondences between
twisted Coxeter elements, twisted adapted classes and twisted AR-quivers.
\end{abstract}

\maketitle

%Furthermore, we can check that Conjecture \ref{conj: Qvee form} and Conjecture \ref{conj: Qvee form iff} hold for $A_5$-case.

\section{Introduction}

 In \cite{Dor91}, Dorey  described relations between
three-point couplings in the simply-laced  affine Toda field theories  (ATFTs) and Lie theories. More precisely, simply-laced ATFTs on untwisted affine Lie algebras  are related to the same type of Lie algebras  and simply-laced ATFTs on twisted affine Lie algebras are related to non-simply laced Lie algebras obtained by corresponding Dynkin diagram automorphisms.

Generally, in ATFTs,  quantum affine algebras appear
as quantum symmetry groups and the fundamental representations of quantum affine algebras
correspond to the quantum particles in the theories \cite{Ber91}. In particular,  the Dorey's rule was  interpreted by  Chari-Pressley \cite{CP96}  in the language of representations of quantum affine algebras. For type $A_n$ and  $D_n$, they used representations associated to Coxeter elements and, for type $B_n$ and $C_n$, they used representations associated to twisted Coxeter elements (\cite{Spr74}) of type $A_{2n-1}$ and $D_{n+1}$. Recently, the first author \cite{Oh14A, Oh14D} found combinatorial interpretation of Dorey's rule for type $A_{n}$ and $D_{n}$ using Auslader-Reiten quivers (\cite{ARS}) and, in \cite{OS16}, we generalized his work for type $B_n$ using twisted  Auslander-Reiten quivers (AR-quivers). In this paper, we introduce and  investigate twisted AR-quivers of type $D_{n+1}$ to find a combinatorial statement of Dorey's rule for type $C_n.$

\vskip 2mm

%In \cite{OS16}, we described combinatorial properties of
%combinatorial Auslander-Reiten quiver (AR-quiver) of twisted adapted
%classes of type  $A_{2n-1}$  and their applications in representation
%theories. A main application was Dorey's rule for the quantum affine algebra $U_q'(B^{(1)}_n)$ in the affine Today field theory. (See \cite{CP96} and \cite{Dor91} for Dorey's rule in the theories of quantum affine algebras and affine Today field theory.)
%As a continuation of \cite{OS16}, in this paper, we study twisted adapted
%classes of type $D_{n+1}$ and introduce  their twisted and folded AR-quivers  to
%investigate the relationship with representation theories of the
%quantum affine algebra $U_q'(C^{(1)}_n)$. Using folded AR-quivers of type $D_{n+1}$, Dorey's rule for $U_q'(C^{(1)}_n)$ is interpreted in a combinatorial way.

A twisted AR-quiver is in the center of this paper and it is a generalization of an AR-quiver which gives an answer to Question \ref{Que:1}.  For any given Dynkin quiver $Q$ of finite type $ADE_n$, the AR-quiver $\Gamma_Q$
can be understood as a realization of the convex partial order $\prec_Q$ on the set of positive roots $\Phi^+$ (\cite{B99}): For $\alpha,\beta\in \Phi^+$ ,
\[   \alpha \prec_Q \beta \Leftrightarrow \text{ there is a path from $\beta$ to $\alpha$ in $\Gamma_Q.$}\]

The set of all AR-quivers $\Gamma_Q$ has one to one correspondences with
 the set of convex partial orders $\prec_Q$,  Dynkin quivers $Q$, Coxeter elements $\phi_Q$ and classes $[Q]$ of reduced expressions of the
longest element $w_0\in\W$ adapted to $Q$ whose cardinalities are $2^{n-1}$. Moreover, the set of $[Q]$'s can be grouped into the {\it adapted cluster point}
$\lf Q \rf$; that is
\begin{align}\label{Eqn:0_1-1}
 \scalebox{0.8}{\raisebox{3.6em}{\xymatrix@C=8ex@R=4ex{
& \{ \prec_Q \} \ar@{<->}[dl]_{1-1} \ar@{<->}[d]_{1-1} \ar@{<->}[dr]^{1-1}  \\
\{ [Q] \} \ar@{<->}[r]^{1-1} & \{ \phi_Q \} \ar@{<->}[r]^{1-1} & \{ Q \} \\
& \{ \Gamma_Q \} \ar@{<->}[ul]^{1-1}\ar@{<->}[ur]_{1-1}
\ar@{<->}[u]^{1-1}}}} \qquad \text{for $2^{n-1}$-many Dynkin quivers $Q$}.
\end{align}

For a Weyl group $\W$ of finite type, a combinatorial AR-quiver $\Upsilon_{[\redex]}$ associated to a class of reduced expressions $[\redex]$ of
$w\in \W$ was introduced in \cite{OS15} as a generalization of AR-quiver in the sense that $\Upsilon_{[\redex]}$ is a realization of the partial order
$\prec_{[\redex]}$ on $\Phi(w)= \{  \alpha\in \Phi^+| w^{-1} \alpha\in -\Phi^+\}$
(see Remark \ref{Rem:2.17_0419} for the definition of $\prec_{[\redex]}$).

Since Dorey's rules for type $C_n$  are described via {\it twisted Coxeter elements} of  $D_{n+1}$,  we reach to the following natural question.

%However, there are no analogous objects to Coxeter elements or Dynkin quivers letting generalize \eqref{Eqn:0_1-1} for non-adapted cluster points.

\begin{question}\label{Que:1}
Is there any non-adapted cluster point which is associated to twisted Coxeter elements? For the cluster point, can we generalize the correspondences in \eqref{Eqn:0_1-1}?
\end{question}

Twisted Coxeter elements of type  $D_{n+1}$  depend on the automorphisms \eqref{eq: C_n} and \eqref{eq: G_2} of Dynkin diagrams. In Section \ref{Sec:twistedDynkin} and Appendix \ref{Appen_A}, we give an affirmative answer to Question \ref{Que:1} for the both types of twisted Coxeter elements associated to \eqref{eq: C_n} and \eqref{eq: G_2}.
%In order to distinguish automorphisms in \eqref{eq: C_n} and \eqref{eq: G_2},
%we call the twisted Coxeter elements associated to \eqref{eq: G_2} by {\it triply} twisted Coxeter elements.
\begin{align}
 C_n \ (n \ge 3)    \longleftrightarrow &
\left( D_{n+1}: \raisebox{1em}{\xymatrix@R=0.5ex@C=4ex{
& & &  *{\circ}<3pt>\ar@{-}[dl]^<{ \  n} \\
*{\circ}<3pt> \ar@{-}[r]_<{1 \ \ }  &*{\circ}<3pt>
\ar@{.}[r]_<{2 \ \ } & *{\circ}<3pt> \ar@{.}[l]^<{ \ \ n-1}  \\
& & &   *{\circ}<3pt>\ar@{-}[ul]^<{\quad \ \  n+1} \\
}}, \ i^{\vee(n+1,2)} = \begin{cases} i & \text{ if } i \le n-1, \\ n+1 & \text{ if } i = n, \\ n & \text{ if } i = n+1. \end{cases} \right) \label{eq: C_n} \\
 G_2   \longleftrightarrow &
\left( D_{4}: \raisebox{1em}{\xymatrix@R=0.5ex@C=4ex{
& &   *{\circ}<3pt>\ar@{-}[dl]^<{ \ 2} \\
*{\circ}<3pt> \ar@{-}[r]_<{1 \ \ }  &*{\circ}<3pt>
\ar@{-}[l]^<{4 \ \ }   \\
& &    *{\circ}<3pt>\ar@{-}[ul]^<{\quad \ \  3} \\
}}, \ \begin{cases} 1^{\vee(4,3)}=2, \ 2^{\vee(4,3)}=3, \ 3^{\vee(4,3)}=1, \\ 4^{\vee(4,3)}=4. \end{cases} \right) \label{eq: G_2}
\end{align}

In Section \ref{Sec:Twisted_adapted_CP}, we construct a commutation class of reduced expressions $[\redez]$ of $w_0$ for each twisted Coxeter element %Also, we prove that the classes obtained from twisted Coxeter elements of type $D_{n+1}$ are reflective equivalent.
and show such classes form a cluster point $\lf \Qg \rf$, called {\it the twisted cluster point}.
Moreover,  in Section \ref{Sec:twistedDynkin}, we introduce new combinatorial model $\Qg$, called a {\it twisted Dynkin quiver of type $D_{n+1}$} and prove that
$$\text{a reduced expression $\redez$ is in $[\Qg] \in \lf \Qg \rf$ if and only if $\redez$ is {\it adapted} to $\Qg$}$$
(see Definition \ref{def: adapted to twisted DQ}). Note that twisted Dynkin quivers of type $D_{n+1}$
can be understood as {\it oriented Dynkin diagrams of type $C_n$} (see Remark \ref{rem: C_n Dynkin}). As a consequence, the twisted analogue of \eqref{Eqn:0_1-1} holds:
\begin{align}\label{Eqn:0_1-1 t}
 \scalebox{0.8}{\raisebox{3.6em}{\xymatrix@C=8ex@R=4ex{
& \{ \prec_{[\Qg]} \} \ar@{<->}[dl]_{1-1} \ar@{<->}[d]_{1-1} \ar@{<->}[dr]^{1-1}  \\
\{ [\Qg] \} \ar@{<->}[r]^{1-1} & \{ \phi_{\Qg} \} \ar@{<->}[r]^{1-1} & \{ \Qg \} \\
& \{ \Upsilon_{[\Qg]} \} \ar@{<->}[ul]^{1-1}\ar@{<->}[ur]_{1-1}
\ar@{<->}[u]^{1-1}}}} \qquad \text{for $2^{n}$-many twisted Dynkin quivers $\Qg$'s}.
\end{align}
Here, $\Upsilon_{[\Qg]}$ is a {\it twisted AR-quiver} equipped with a coordinate system. %In addition, the analogous result for the triply twisted cluster point is written in Appendix \ref{Appen_A}.
Remark that
\begin{equation*}
\text{ for the twisted adapted classes of $A_{2n-1}$ type, \eqref{Eqn:0_1-1 t} is not true.}
\end{equation*}
% In \cite{OS16}, we proved the following:
%\begin{itemize}
%\item There are twisted adapted classes of type $A_{2n-1}$ not associated to twisted Coxeter elements.
 %\end{itemize}

\vskip 2mm
Another goal of this paper is answering to the following question.

\begin{question}
Is there a combinatorial way to find labels of twisted AR-quiver $\Upsilon_{[\Qg]}$?
\end{question}

 In \cite{Oh14A, Oh14D},  the first named author  showed a purely combinatorial way to find labels of $\Gamma_Q$ in $\Phi^+$.
Analogous to the adapted cases, in Section \ref{Section:twisted_AR_D_label} and \ref{Sec:twisted_AR_D_label_shape},
we describe purely combinatorial ways to find labels of $\Upsilon_{[\Qg]}.$
In order to do this, we introduce new quivers with coordinate systems, called {\it folded AR-quivers} $\widehat{\Upsilon}_{[\Qg]}$.
%Folded AR-quivers are quite useful to generalize the results in \cite{Oh14A, Oh14D} to $\widehat{\Upsilon}_{[\Qg]}$.

An interesting result in Section \ref{Section:twisted_AR_D_label} is that each twisted adapted class  $[\Qg]$ of $D_{n+1}$
%associated to a twisted Coxeter element $s_{i_1} \cdots s_{i_n}\vee$
is closely related to the adapted class $[Q]$ where $Q$ is the associated Dynkin quiver of type $A_n$. Moreover,
%associated with the Coxeter element $s_{\overline{i}_1}s_{\overline{i}_2} \cdots s_{\overline{i}_n}$ ($\{ \overline{i}_1,\overline{i}_2,\ldots,\overline{i}_n \} =\{ 1,2,\ldots,n \}$):
a folded AR-quiver $\widehat{\Upsilon}_{[\Qg]}$ of $D_{n+1}$ can be obtained by {\it gluing} two copies of  AR quiver $\Gamma_{[Q]}$ of $A_n$.

On the other hand, in Section \ref{Sec:twisted_AR_D_label_shape}, we introduce another method to find labels of
twisted AR-quiver $\widehat{\Upsilon}_{[\Qg]}$ using {\it swings} (see Definition \ref{def: rename} for swings.) This result implies % we can find labels of $\widehat{\Upsilon}_{[\Qg]}$ just by its shape,
 a folded AR-quiver $\widehat{\Upsilon}_{[\Qg]}$ of type $D_{n+1}$ has similar properties with a full subquiver of $\Gamma_{[Q]}$ of {\it type $D_{n+2}$}
consisting of vertices with residues $1$ to $n$ (see \cite[Theorem 1.20]{Oh14D}  and \eqref{eq: CDCD} also).

\vskip 2mm
Finally, using the combinatorial properties of folded AR-quivers, we give answers to the following question.

\begin{question}\label{Qes3:application}
Are there applications of twisted and folded AR-quivers to the  representation theory of $U'_q(C_n^{(1)})$? Can we derive Dorey's rule for type $C_n$ from $\widehat{\Upsilon}_{[\Qg]}$?
\end{question}

%In Section \ref{Section:additive}, we prove {\it the twisted additive property} for $\Upsilon_{[\Qg]}$ of type $D_{n+1}$,
%which were introduced in \cite{OS16} for the twisted adapted point $\lf Q^{\diamondsuit} \rf$ of type $A_{2n-1}$.
%In the formula of twisted additive property, the orders of orbits associated to the automorphism $\vee(n+1, 2)$ in \eqref{eq: C_n} play a crucial role.
%If we consider the identity map on the Dynkin diagram %instead of $\vee(n+1, 2)$ in \eqref{eq: C_n}
%then, from the twisted additive formula, we get the original additive property of an AR-quiver $\Gamma_Q$.
%So the twisted additive property of $\Upsilon_{[\Qg]}$ is considered as a generalization of the additive property of $\Gamma_{Q}$.

In \cite{Oh15E}, using the properties of $\Gamma_Q$ discovered in \cite{Oh14A, Oh14D},
 the first named author  introduced {\it generalized distances} $\gdist_{[\redez]}$ of a sequence of positive root lattice elements,
{\it radiuses} $\rds_{[\redez]}$ of roots and {\it socles} $\soc_{[\redez]}$ of pairs in $\Phi^+$ associated to any $[\redez]$. He proved that
{\rm (i)} the Dorey's rule for $U_q'(A_n^{(1)})$ (resp. $U_q'(D_n^{(1)})$) can be interpreted as the coordinates of $(\al,\be,\ga) \in (\PR)^3$ in $\Gamma_Q$
where $\al+\be = \ga\in \PR$, {\rm (ii)} denominator formulas $d_{k,l}(z)$ for $U_q'(A_n^{(1)})$ (resp. $U_q'(D_n^{(1)})$)
can be read from {\it any} $\Gamma_Q$, {\rm (iii)} $\rds_{[Q]}(\gamma)= \mathsf{m}(\gamma)$, $\gdist_{[Q]}(\alpha,\beta)\leq \text{max} \{\mathsf{m}(\alpha),\mathsf{m}(\beta)\}$
and the notion $\soc_{[Q]}$ is well-defined. Here,
$\mathsf{m}(\alpha)$ denotes the multiplicity of $\alpha$ (see Definition \ref{def: multiplicity} for the multiplicity).

As an answer to Question \ref{Qes3:application}, in Section \ref{Sec:dist_rds} and Section \ref{Sec:distancePoly},
we generalize {\rm (i)}$\sim${\rm (iii)} by introducing
{\it folded multiplicities} and applying the combinatorial properties of $\widehat{\Upsilon}_{[\Qg]}$:
{\rm (i$'$)} the Dorey's rule for $U_q'(C_n^{(1)})$ can be interpreted as the coordinates of $(\al,\be,\ga) \in (\PR_{D_{n+1}})^3$ in $\widehat{\Upsilon}_{[\Qg]}$
where $(\al,\be)$ is a {\it $[\Qg]$-minimal pair} for $\ga$, {\rm (ii$'$)} denominator formulas $d_{k,l}(z)$ for $U_q'(C_n^{(1)})$
can be read from any $\widehat{\Upsilon}_{[\Qg]}$, {\rm (iii$'$)} $\rds_{[\Qg]}(\gamma)= \overline{\mathsf{m}}(\gamma)$,
$\gdist_{[\Qg]}(\alpha,\beta)\leq \text{max} \{\overline{\mathsf{m}}(\alpha),\overline{\mathsf{m}}(\beta)\}$
and the notion $\soc_{[\Qg]}$ is well-defined. Here $\overline{\mathsf{m}}(\alpha)$  denotes the {\it folded multiplicity} of $\alpha$
(see Definition \ref{def: folded multiplicity}).

\vskip 2mm

Note that generalized Cartan matrices of $D_{n+1}^{(2)}$ and $C_n^{(1)}$ are transpose to each other. As we mentioned in \cite{OS16}, such relation was investigated in the theory of representations of quantum groups via $q$-characters in \cite{FH11,H10}. Moreover, in the forthcoming paper by Kashiwara-Oh  \cite{KO16}, the categorical relations between representations of $U_q(D_{n+1}^{(2)})$ and $U_q(C_{n}^{(1)})$ will be described using the results in this paper.

\vskip 5mm

\section{Backgrounds}

\subsection{Foldable $r$-cluster point and related reduced expressions}
In this subsection, we recall the definition of  foldable $r$-cluster points of type ADE and related  notions and properties. For more detail, see \cite{OS16}.

Let $I_n=\{1, 2, \cdots, n\}$ be the index set of Dynkin diagram $\Delta_n^X$ of finite type $X_n$.
Denote by $W^X_n$  and ${}_n w_0^X$ the Weyl group and the longest element.
The Weyl group $W^X_n$ is generated by simple reflections $\{ \, s_i \, |\, i\in I_n\, \}.$
We sometimes omit the type and the rank if there is no danger of confusion.

For $w\in W$, let  $i_1, i_2, \cdots, i_l$ be indices satisfying
$w=s_{i_1} s_{i_2}\cdots s_{i_l}$. If there is no sequence $j_1,
j_1, \cdots, j_{l'}$ of indices such that $w=s_{j_1} s_{j_2}\cdots
s_{j_{l'}}$ and $l'<l$ then we say $l$ is the {\it length } of $w$
and denote by $\ell(w)$. The sequences $\ii:= i_1\, i_2\, \cdots
i_l$ (resp. $s_{i_1}\, s_{i_2}\, \cdots s_{i_l}$) is a {\it reduced
word} (resp. {\it reduced expression}) of $w$. We often abuse the
notation $\ii$, instead of the reduced expression $s_{i_1}\, s_{i_2}\, \cdots s_{i_l}$.
Also, note that we denote by $\ii_0$ a reduced word of the longest
element $w_0\in W$ and by $\ell$ the length of $w_0$.

\begin{definition}
Let $\ii$ and $\jj$ be reduced expressions of $w\in W$. If $\jj$ can be obtained by applying commutation relations, $s_is_j=s_js_i$, to
$\ii$ then we say $\ii$ and $\jj$ are {\it commutation equivalent} and denote $\ii \sim \jj.$ The {\it commutation class} of $\ii$ is denoted by $[\ii].$
\end{definition}

It can be easily shown that if $\jj\sim \ii$ and $\ii$ is a reduced expression of $w$ then $\jj$ is also a reduced word of $w\in W$.
Now, if $w=w_0$, there is another way to find reduced expressions from a given reduced expression $\ii_0$ by Proposition \ref{Prop: reflection}.
(See \cite{OS15}, for example.)

\begin{proposition} \label{Prop: reflection}
Let $*: I \to I$, $i\mapsto i^*$, be the involution
induced by $w_0$ on $I$ $($See \cite{Bour}$)$.
For a reduced word  $\ii_0= i_1\, i_2\, \cdots\, i_\ell$ of $w_0$, two other words $\ii'_0= i_2\, i_3\, \cdots\, i_\ell i_1^*$
and $\ii''_0=i_\ell^*\, i_1\, i_2\, \cdots\, i_{\ell-1}$ are also reduced words of $w_0.$ Moreover, we have $[\ii_0] \neq [\ii'_0],\, [\ii''_0].$
\end{proposition}

\begin{definition} \cite{OS15} Let $\ii_0$ be a reduced expression of $w_0$.
\begin{enumerate}
\item  If there is a reduced expression $\ii'_0\in [\ii_0]$ (resp. $\ii''_0\in [\ii_0]$) such that $\ii'_0= i \, i_2\, i_3\, \cdots\, i_\ell$ (resp. $\ii''_0= i_1 \, i_2\, i_3\, \cdots\, i_{\ell-1}\, i$) then the index $i$ is called a {\it sink } (resp. {\it source}) of $[\ii_0]$.
\item The right and left {\it reflection functors} $r_i$ for $i\in I$ are defined as follows:
\[ \, [\ii_0] \, r_i = \left\{
\begin{array}{ll}\,  [\, i_2\, i_3\, \cdots\, i_\ell\, i_1^*\,] & \text{if $i=i_1$ is a sink of $[\ii_0]$ and $i_1\, i_2\, \cdots \, i_\ell \in [\ii_0]$,}\\ \, [\ii_0] & \text{otherwise; }
 \end{array}
\right. \]
 \[ \, r_i\,  [\ii_0] = \left\{
\begin{array}{ll}\,  [\ i_\ell^* \,i_1\, i_2\, i_3\, \cdots\, i_{\ell-1}\,] & \text{if $i=i_l$ is a source of $[\ii_0]$ and $i_1\, i_2\, \cdots \, i_\ell \in [\ii_0]$,}\\ \, [\ii_0] & \text{otherwise. }
 \end{array}
\right.
\]
For a sequence of indices $\bi:= j_1\, j_2\, \cdots\, j_t$, we denote by $[\ii_0] \, r_\bi:= [\ii_0]\, r_{j_1} r_{j_2}\cdots r_{j_t}$ and $r_\bi \, [\ii_0]:= r_{j_t} r_{j_{t-1}}\cdots r_{j_1} \, [\ii_0].$
\item If $[\ii'_0]$ can be obtained by applying reflection functors to $[\ii_0]$ then we say $[\ii_0]$ and $[\ii'_0]$ are
{\it reflection equivalent} and denote $[\ii_0] \overset{r}{\sim} [\ii'_0]$. The family of commutation classes
\[ \lf \ii_0 \rf := \{ \, [\ii'_0] \, | \, [\ii_0]\overset{r}{\sim}[\ii'_0]\, \} \]
is called an {\it r-cluster point}.
\end{enumerate}
\end{definition}

By the definition of reflection functors, we can see that if $[\ii_0] \overset{r}{\sim} [\ii'_0]$ for $\ii_0=i_1\, i_2 \cdots \, i_\ell$ and $\ii'_0= i'_1\, i'_2\, \cdots\, i'_\ell$ then
\[ \# \{ \, i_s\, | \, s=1, \cdots, \ell\, \text{ such that } i_s=i \text{ or } i^* \,  \} =\# \{ \, i'_s\, |  s=1, \cdots, \ell\, \text{ such that }  \, i'_s=i \text{ or } i^* \, \}  \]
for any $i\in I.$ Hence we consider the following notion introduced in \cite{OS16}.

\begin{definition} \label{Def: Coxeter Composition}\cite{OS16}
\begin{enumerate}
\item Let $\vee:I \to I$ be an automorphism such that $i \mapsto i^\vee$. We say $\vee$ is {\it compatible} with the involution $*$ if $i$ and $i^*$ are in the same orbit class $\bar{i}$ determined by $\vee$.
\item Let $\vee$ be an automorphism on $I$ compatible with $*$ and $\ii_0= i_1\, i_2\, \cdots \, i_\ell$ be a reduced expression of $w_0$. Denote by $\overline{I}=\{ \, \bar{i} \, | \, i\in I\, \}$ the set of orbits determined by $\vee$. Then the {\it $\vee$-Coxeter composition } is
\[ \mathsf{C}_{[\ii_0]}:= ( \mathsf{C}_{[\ii_0]}^\vee(\bar{j_1}),  \mathsf{C}_{[\ii_0]}^\vee(\bar{j_2}), \cdots,  \mathsf{C}_{[\ii_0]}^\vee(\bar{j_p})), \quad  p=|\overline{I}|,  \]
where the smallest representative of $\bar{j_r}$ is less than that of $\bar{j_s}$ if and only if $r<s$ and
\[ \mathsf{C}_{[\ii_0]}^\vee(\bar{k}): =|\, \{\, i_s\, |\, i_s \in \bar{k},\, 1\leq s\leq \ell\, \} \, |.\]
\end{enumerate}
\end{definition}

Due to the compatibility of $\vee$ in Definition \ref{Def: Coxeter Composition}, if $[\ii_0]$ and $[\ii'_0]$ are in the same r-cluster point then they have the same $\vee$-Coxeter composition. Hence we can denote the $\vee$-Coxeter composition associated to $[\ii_0]$ by $\mathsf{C}_{\lf \ii_0\rf}^\vee.$

\begin{definition} \cite[Definition 1.11]{OS16}
For a given automorphism $\vee$, the r-cluster point $\lf \ii_0\rf$ is said to be {\it $\vee$-foldable} if
\[  \mathsf{C}_{\lf\ii_0\rf}^\vee(\bar{j_r})=  \mathsf{C}_{\lf\ii_0\rf}^\vee(\bar{j_s}) \]
for any $\bar{j_r}, \bar{j_s} \in \overline{I}.$
\end{definition}

\begin{convention}
If there is no danger of confusion, we identify  the index $\bar{j}_s$ in $\bar{I}$  with $s\in \Z.$ Hence the sum (resp. subtraction) $\bar{j}_s + \bar{j}_r$
(resp. $\bar{j}_s-\bar{j}_r)$ is considered as that in $\Z$ so that it is $s+r$ (resp. $s-r$) and, similarly, $\bar{j}_s \pm t:= s\pm t$ for any $t\in \Z.$
\end{convention}

In this paper, we deals with $\vee$-foldable cluster points for $\vee$ in \eqref{eq: C_n}. The following remark shows the $\vee$-Coxeter composition of such cluster points.

\begin{remark} Let $\vee$ be defined in \eqref{eq: C_n}. If $\lf \ii_0\rf$ is $\vee$-foldable, then
\begin{align} \label{eq: Foldable Coxeter D}
\mathsf{C}^{\vee}_{\lf \ii_0 \rf}=
( \underbrace{n+1, \ldots ,n+1}_{ n\text{-times}} ).
\end{align}
Let $\vee$ be defined in \eqref{eq: G_2}. If $\lf \ii_0\rf$ is $\vee$-foldable or $\vee^2$-foldable, then
\begin{align} \label{eq: Foldable Coxeter D4}
\mathsf{C}^{\vee}_{\lf \ii_0 \rf}=
(6,6).
\end{align}
\end{remark}

\vskip 3mm

Now let us recall another important notion called a {\it twisted Coxeter element} (resp.
a {\it triply twisted Coxeter element}).
Let $\sigma \in GL(\C\Phi)$ be a linear transformation of finite order
which preserves the set of simple roots $\Pi \subset \Phi$ and consider  $W\sigma\subset GL(\C\Phi).$ Note that the Weyl group $W$ acts on $W\sigma$ by conjugations.

\begin{definition} \label{def: twisted Coxeter} \hfill
\begin{enumerate}
\item Let $\Pi:=\bigsqcup_{t=1}^k \Pi_{i_t}$, where $\Pi_{i_t}$ are orbits by the automorphism $\sigma$. Take $s_{i_t}\in \Pi_{i_t}$ for $t=1, \cdots, k$ and consider the element $w=s_{i_{\tau(1)}}s_{i_{\tau(2)}}\cdots s_{i_{\tau(k)}}\in W$ for a permutation $\tau \in \mathfrak{S}_k.$ Then the element $w\sigma \in W\sigma$ is called a {\it $\sigma$-Coxeter element.}
\item If $\sigma$ in (1) is $\vee$ in \eqref{eq: C_n} then  {\it
$\sigma$-Coxeter element} is also called a {\it twisted Coxeter
element} of type $D_{n+1}$.
\item If $\sigma$ in (1) is $\vee$ in \eqref{eq: G_2} then  {\it
$\sigma$-Coxeter element} is also called a {\it triply twisted Coxeter
element} of type $D_4$.
\item If $\sigma$ is the identity map, a $\sigma$-Coxeter element is called a {\it Coxeter element}.
\end{enumerate}
\end{definition}

We remark some properties which are useful in the following sections.

\begin{remark} \label{rem: basic prop twist}
Note that $w$ in ${\rm (1)}$ of Definition \ref{def: twisted Coxeter} is {\it fully commutative}, i.e. reduced expressions of $w$ consist
of the unique commutation class.
\end{remark}

\begin{proposition}\label{prop: number tCox elts}  \cite[Proposition 2.6]{OS16}
 The number of twisted Coxeter elements associated to \eqref{eq: C_n} is $2^n.$
\end{proposition}

\subsection{ (Combinatorial) Auslander-Reiten quivers}
 For type $ADE$, there are one-to-one correspondences between the set of  commutation classes of  adapted reduced expressions of $w_0$, Dynkin quivers,
 Coxeter elements and  Auslander-Reiten quivers. In this section, we recall the correspondences and introduce the notion of combinatorial Auslander-Reiten quivers,
 which is a generalization of  Auslander-Reiten quivers in a combinatorial aspect.
Details can be found in \cite{ARS,ASS,Gab80,Oh14A,Oh15E,OS15}.

Let $\Delta$ be a Dynkin diagram of type $ADE_n$ and $Q$ be a Dynkin quiver,
which is obtained by assigning direction to every edge in $\Delta.$ The set of positive roots is denoted by $\Phi^+$ and
$\Pi=\{ \alpha_i \ | \ i \in I \}$ is the set of simple roots. Gabriel theorem (Theorem \ref{Gabriel theorem}) and the construction of Auslander-Reiten quiver
(Definition \ref{AR-quiver}) shows the relation between quiver representations and positive roots.

\begin{theorem}[Gabriel's theorem]\cite{Gab80} \label{Gabriel theorem}
There is the bijection between the set ${\rm Ind}Q$ of isomorphism classes of finite dimensional indecomposable modules over the path algebra
$\C Q$ and $\Phi^+$ which takes a class  to the corresponding dimension vector.
\end{theorem}

\begin{definition}
\begin{enumerate}
\item An index $i\in I$ is called a {\it sink} (resp. {\it source}) of $Q$ if there is no arrow exiting from (resp. entering to) the vertex $i$.
\item Denote by $s_i Q$ or $i\, Q$ the quiver obtained by reversing every arrow connecting $i$ and another vertex if $i$ is a sink or a source. Otherwise, we let $s_i Q=  i\, Q= Q.$
\item Let $\ii=i_1\, i_2\, \cdots \, i_{\ell(w)}$ be a reduced expression of $w$. We say $\ii$ is {\it adapted to} the Dynkin quiver $Q$ if
$ i_k \text{ is a sink of } i_{k-1} i_{k-2}\cdots i_2 i_1 Q$ for $k=1, \cdots, \ell(w).$
\end{enumerate}
\end{definition}

For adapted reduced expressions of $w_0$, there are well-known facts which is stated below.

\begin{theorem} \hfill
\begin{enumerate}
\item[{\rm (1)}] There is the natural one-to-one correspondence between the set of commutation classes of adapted reduced expressions of $w_0$  and the set of Dynkin quivers. Hence we denote by $[Q]$ the class of reduced expressions of $w_0$ which are adapted to the quiver $Q$.
\item[{\rm (2)}] There exists a unique Coxeter element $\phi_Q$ which is adapted to $Q$. Conversely, any Coxeter element is adapted to some Dynkin quiver $Q$.
\end{enumerate}
\end{theorem}

Hence there are canonical one-to-one correspondences between rank $2^{n-1}$ sets:
\begin{equation} \label{1-1_adapted}
\xymatrix@C=8ex@R=4ex{
\{ [Q] \} \ar@{<->}[r]^{1-1} & \{ \phi_Q \} \ar@{<->}[r]^{1-1} & \{ Q \}}
\end{equation}
Note that there exists a unique cluster point consisting of all adapted classes $[Q]$. we call it the {\it adapted cluster point } and denote by $\lf Q \rf.$

\begin{definition}\label{AR-quiver}
For a Dynkin quiver $Q$ of finite type $ADE$, let us take a reduced word $\ii_0$ in $[Q]$. The quiver
$\Gamma_Q=(\Gamma_Q^0,\Gamma_Q^1)$ is called the {\it Auslander-Reiten quiver} (AR-quiver) if
\begin{enumerate}
\item each vertex in $\Gamma_Q^0$ corresponds to an isomorphism class $[M]$ in ${\rm Ind}Q$,
\item an arrow $[M] \to [M']$ in $\Gamma_Q^1$ represents an irreducible morphism $M \to M'$.
\end{enumerate}
\end{definition}

By Gabriel's theorem, every vertex of an AR-quiver can be  labeled by a positive root and there is a well-known combinatorial construction of an AR-quiver. In order to introduce its simpler construction and properties,  let us consider the subset
\[ \Phi(w):= \{ \, \beta \in \Phi^+\, | \, w^{-1}(\beta)\in -\Phi^+\, \} =\{ \, \beta_k^\ii\, |\,\beta_k^\ii= s_{i_1}s_{i_2}\cdots s_{i_{k-1}}(\alpha_{i_k}),\ k=1, \cdots, \ell(w)\, \}.\]
Note that $|\Phi(w)|=\ell(w)$ and  $\Phi(w_0)=\Phi^+.$

\begin{definition} \cite[\S 5.3]{Kac}
For a positive root $\alpha=\sum_{i\in I} m_i \alpha_i \in \PR$, the {\it support of $\alpha$} is denoted by $\text{supp}(\alpha)$ and is defined by
\[ \text{supp}(\alpha):=\{  \, \alpha_k \, |\,  m_k \neq 0,\, k\in I\,  \}.\]
Also, if $\alpha_k\in \text{supp}(\alpha)$ then we say $\alpha_k$ is a support of $\alpha.$
\end{definition}

\begin{definition} \label{def: multiplicity}
For any $\gamma \in \PR$,
{\it the multiplicity} of $\gamma$, denoted by $\mathsf{m}(\gamma)$ is the positive integer defined as follows:
$$\mathsf{m}(\gamma) = \max \{\, m_i\, | \, \sum_{i\in I} m_i\alpha_i=\ga\, \}.$$
If $\mathsf{m}(\gamma)=1$, we say that $\gamma$ is {\it multiplicity free}.
\end{definition}

\begin{remark} \label{Rem:2.17_0419}
Let $\ii$ be a reduced word of $w$.  The {\it total order $<_{\ii}$} on $\Phi(w) $ is defined by
\[ \beta_k^{\ii} <_{\ii} \beta_l^{\ii} \text{ if and only if } k<l.\]
Then $<_{\ii}$ is  {\it convex}, i.e.,  if   $\alpha, \beta, \alpha+\beta\in \Phi^+$  and $\alpha<_{\ii}\beta$ then $\alpha<_{\ii} \alpha+\beta <_{\ii}\beta$.
Using the convex total orders $<_{\ii}$, the {\it convex partial order $\prec_{[\ii]}$} can be defined as follows:
\[  \alpha\prec_{[\ii]} \beta \quad  \Leftrightarrow \quad \alpha <_{\ii'} \beta \text{ for any } \ii'\in [\ii].\]
 If $\ii_0$ is adapted to $Q$ then we denote $\prec_{[\ii_0]}$ by $\prec_Q$ and $\prec_Q$ is closely related to the AR-quiver.
\end{remark}

\begin{definition} \label{def: Gamma_Q}
A combinatorial way to construct the AR-quiver $\Gamma_Q$ with a Coxeter element $\phi_Q$ is given as follows (\cite[\S 2.2]{HL11}):
\begin{enumerate}
\item Let us denote $\phi_Q=s_{i_1}s_{i_2}\cdots s_{i_n}$ and consider a {\it height function} $\xi: I \to \Z$ satisfying $\xi(i)= \xi(j)+1$ if there is an arrow from $j$ to $i$ in $Q$.  Take the injection $\widetilde{\phi}_Q: \Phi(\phi_Q) \to I \times \Z$ such that $\beta_k^{\phi_Q}:= s_{i_1}s_{i_2}
\cdots s_{i_{k-1}}(\alpha_{i_k})\mapsto (i_k, \xi(i_k))$.
\item Inductively, we extend $\widetilde{\phi}_Q$ to the map on $\Phi^+$ satisfying that if $\beta,\phi_Q(\beta)\in \Phi^+$ and $\widetilde{\phi}_Q(\beta)= (i, p)$ then $\widetilde{\phi}_Q(\phi_Q(\beta))= (i, p-2).$
\item If  $(i,p)$ and $(j,q)$ are in {\rm Im}($\widetilde{\phi}_Q$), two indices $i$ and $j$ are connected in $\Delta$ and $q=p+1$  then there is an arrow $(i,p)\to (j,q).$
\end{enumerate}
\end{definition}

Note that if $\widetilde{\phi}_Q(\alpha) = (i,p)$ for $\alpha\in \Phi^+,$ then we say that $i$ is the {\it residue of $\al$ with respect
to $[Q]$} or $\al$ is in the $i$-th residue in $\Gamma_Q$. Also, $\Gamma_Q$ has the {\it additive property}, i.e. if $\widetilde{\phi}_Q(\be)=(i,p)$ then
\begin{align}\label{eq: addtive property}
\be+\phi_Q(\be) = \sum_{\widetilde{\phi}_Q(\ga)=(j,p-1)} \ga,
\end{align}
where $i$ and $j$ are adjacent in $Q$.
Also, the followings are useful properties of $\Gamma_Q.$

\begin{proposition}  \cite{B99, Gab80, R80} \label{Prop:AR}
Let us denote by $\mathsf{h}^\vee$ the  dual Coxeter number associated to $Q$. In $\Gamma_Q$, there are $r^Q_i=(\mathsf{h}^\vee+a^Q_i-b^Q_i)/2$ vertices in the $i$-th residue, where $a^Q_i$ $($resp. $b^Q_i)$ is the number of arrows in $Q$ between $i$ and $i^*$ directed toward $i$ $($resp. $i^*)$.
\end{proposition}

Thus, for all $i \in I$, we have
\begin{align} \label{eq: r_i A_n}
 r^Q_i +r^Q_{i^*}  = n+1 \qquad \text{ if $Q$ is of type $A_n$}.
\end{align}

\begin{proposition} \label{Prop:adapted}\cite{B99,OS15,R96}
\begin{enumerate}
\item For distinct Dynkin quivers $Q$ and $Q'$, we have $\Gamma_Q \not\simeq \Gamma_{Q'}$ as quivers.
\item For $\alpha, \beta\in \Phi^+$, we have  $\alpha \prec_Q \beta$ if and only if  there is a path from $\beta$ to $\alpha$ in $\Gamma_Q.$
\end{enumerate}
\end{proposition}

Hence we get the following diagram \cite{OS16}:
\begin{align}\label{eq: 2n-1many 2}
\raisebox{3.6em}{\xymatrix@C=8ex@R=4ex{
& \{ \prec_Q \} \ar@{<->}[dl]_{1-1} \ar@{<->}[d]_{1-1} \ar@{<->}[dr]^{1-1}  \\
\{ [Q] \} \ar@{<->}[r]^{1-1} & \{ \phi_Q \} \ar@{<->}[r]^{1-1} & \{ Q \} \\
& \{ \Gamma_Q \} \ar@{<->}[ul]^{1-1}\ar@{<->}[ur]_{1-1} \ar@{<->}[u]^{1-1}}} \text{for $[Q] \in \lf Q \rf$}.
\end{align}

We close this subsection introducing the algorithm to get $\Gamma_{i\, Q}$ from  $\Gamma_Q$ for a sink $i$ of $Q$.
\begin{algorithm} \label{alg: Ref Q}
Let $\mathsf{h}^\vee$ be the dual  Coxeter number  associated to $Q$.
\begin{enumerate}
\item[{\rm (A1)}] Remove the vertex $(i,p)$ such that $\widetilde{\phi}_Q(\al_i)=(i,p)$ and the arrows entering into $(i,p)$.
\item[{\rm (A2)}] Add a vertex $(i^*,p-\mathsf{h}^\vee)$ and arrows to $(j,p-\mathsf{h}^\vee+1)$ for all $j$ adjacent to $i^*$ in $\Delta$.
\item[{\rm (A3)}] Label the vertex $(i^*,p-\mathsf{h}^\vee)$ with $\al_i$ and substitute the labels $\be$ with $s_i(\be)$ for all $\be \in \Phi^+ \setminus \{\al_i\}$.
\end{enumerate}
\end{algorithm}

\vskip 3mm

In \cite{OS15}, the authors constructed the combinatorial AR-quivers $\Upsilon_{[\ii_0]}$
for any commutation class $[\ii_0]$ of any finite type, which can be understood as a generalization of AR-quivers.

\begin{algorithm} \label{Alg_AbsAR}
Let $\ii_0=(i_1 i_2 i_3 \cdots i_{\N})$ be
a reduced expression of an element $w_0\in W$. The quiver $\Upsilon_{\ii_0}=(\Upsilon^0_{\ii_0},
\Upsilon^1_{\ii_0})$ associated to $\ii_0$ is constructed in the
following algorithm:
\begin{enumerate}
\item[{\rm (Q1)}] $\Upsilon_{\ii_0}^0$ consists of $\N$ vertices labeled by $\beta^{\ii_0}_1, \cdots, \beta^{\ii_0}_{\N}$.
\item[{\rm (Q2)}] The quiver $\Upsilon_{\ii_0}$ consists of $|I|$ residues and each vertex $\beta^{\ii_0}_k\in \Upsilon^0_{\ii_0}$ lies in the $i_k$-th residue.
\item[{\rm (Q3)}] There is an arrow from $\beta^{\ii_0}_k$ to $\beta^{\ii_0}_j$ if the followings hold:
\begin{enumerate}
\item[{\rm (Ar1)}] two vertices $i_k$ and $i_j$ are connected in the Dynkin diagram,
\item[{\rm (Ar2)}] $ j= \max \{ j'\, |\, j' <k, \, i_{j'}=i_j \} $,
\item[{\rm (Ar3)}] $ k= \min \{ k'\, |\, k' >j, \, i_{k'}=i_k \} $.
\end{enumerate}
\item[{\rm (Q4)}] Assign the color $m_{jk}=-(\alpha_{i_j}, \alpha_{i_k})$ to each arrow $\beta^{\ii_0}_k\to \beta^{\ii_0}_j$ in {\rm (Q3)}; that is,
$\beta^\redex_k \xrightarrow{m_{jk}} \beta^\redex_j$.  Replace
$\xrightarrow{1}$ by $\rightarrow$,  $\xrightarrow{2}$ by
$\Rightarrow$ and  $\xrightarrow{3}$ by  $\Rrightarrow$.
\end{enumerate}
\end{algorithm}

By the following theorem,  a combinatorial AR-quiver can be understood as a generalization of an AR-quiver  which proposes the generalization of  Proposition \ref{Prop:adapted}. (See \cite{OS15}.) Also, we have a simple algorithm to get $\Upsilon_{[\ii_0]\, r_i}$ from $\Upsilon_{[\ii_0]}$ for a sink $i$  of $[\ii_0].$

\begin{algorithm} \label{alg: Ref Q re}
\begin{enumerate}
\item[{\rm ($A'$1)}] Remove the vertex $v_0$  of  $\Upsilon_{[\ii_0]}$ in the $i$-th residue which does not have arrows exiting from $v_0$ and remove every arrow
entering  into  $v_0$.
\item[{\rm ($A'$2)}] Add a vertex $v_1$ in the $i^*$-th residue and add arrows from $v_1$ to a vertex $v$ such that
{\rm (i)} $v$ is in the $j$-th residue for an adjacent vertex  $j$ to $i^*$ in $\Delta$,
{\rm (ii)} $v'\prec_{[\ii_0]} v$ for any other vertex $v'$ in the $i^*$-th or $j$-th residue.
\item[{\rm ($A'$3)}] The label of  $v_1$ is $\alpha_{i}$. For the other vertices, substitute the label $\alpha$ by $s_i(\alpha)$.
\end{enumerate}
\end{algorithm}

\begin{theorem} \cite{OS15} \label{thm: OS14}
Let us choose any commutation class $[\ii_0]$ of $w_0$ and a reduced word $\ii_0$ in $[\ii_0]$.
\begin{enumerate}
\item[{\rm (1)}] The construction of $\Upsilon_{\ii_0}$ does depend only on its commutation class $[\ii_0]$ and hence $\Upsilon_{[\ii_0]}$
is well-defined.
\item[{\rm (2)}] $\al \prec_{[\ii_0]} \be$ if and only if there exists a path from $\be$ to $\al$ in $\Upsilon_{[\ii_0]}$.
\item[{\rm (3)}] By defining the notion, standard tableaux of shape $\Upsilon_{[\ii_0]}$, every reduced word $\ii'_0 \in [\ii_0]$
corresponds to a standard tableau of shape $\Upsilon_{[\ii_0]}$ and can be obtained by reading residues in a way compatible with the tableaux.
\item[{\rm (4)}] When $[\ii_0]=[Q]$, $\Upsilon_{[Q]}$ is isomorphic to $\Gamma_Q$ as quivers.
\end{enumerate}
\end{theorem}

\section{Twisted adapted cluster point of type $D_{n+1}$} \label{Sec:Twisted_adapted_CP}
\subsection{Twisted adapted cluster point and twisted Coxeter element}
In this section, we denote by $\Delta$,  $W$ and $\Phi^+$  the Dynkin diagram, the Weyl group and the set of positive roots of type $D_{n+1}$.
Also, we let $\vee$ be the automorphism in \eqref{eq: C_n}. By identifying $\al_i=\ve_i-\ve_{i+1}$ $(1 \le i \le n)$ and $\al_{n+1}=\ve_n+\ve_{n+1}$, positive roots in $\Phi^+$
can be denoted as follows:
 \begin{equation*}
 \Phi^+_{D_{n+1}}  =\{ \lan a_1, -a_2 \ran, \ \lan b_1, b_2 \ran, \ \lan c, n+1 \ran \, | 1\leq a_1 < a_2 \leq n+1 , \ 1\leq b_1 < b_2 \leq n, \ 1\leq c \leq n \}
 \end{equation*}
 where
 \begin{equation}
 \left\{
 \begin{aligned}
& \textstyle  \lan a_1, -a_2 \ran  = \sum_{i=a_1}^{a_2-1} \alpha_i =\ve_{a_1}-\ve_{a_2}, \\
 & \textstyle \lan b_1, b_2 \ran   =\sum_{i=b_1}^{b_2-1} \alpha_i+ 2\sum_{j=b_2}^{n-1}\alpha_j+\alpha_{n}+\alpha_{n+1}=\ve_{b_1}+\ve_{b_2}, \\
  &\textstyle \lan c, n+1\ran= \sum_{i=c}^{n+1} \alpha_i -\alpha_n=\ve_{c}+\ve_{n+1}.
 \end{aligned}
 \right.
  \end{equation}

Recall that
$$\Phi^+_{A_{n}}  = \{ [a,b] \seteq \sum_{k=a}^{b} \al_k=\ve_a-\ve_{b+1} \ | \ 1 \le a \le b \}.$$

In order to distinguish positive roots of type $D_m$ and $A_m$, we use the following definition.

 \begin{definition}\
 \begin{enumerate}
 \item
For $\left< a,b \right>\in \Phi_{D_m}^+,$ we denote $a$ and $b$ by the {\it first and second summands} of $\alpha= \left< a,b \right>$, respectively.
\item
For $\beta=[a,b]\in \Phi_{A_m}^+$, we denote $a$ and $b$ by the {\it first and second components} of $\beta$, respectively. If $\beta=\alpha_a$ then we write $\beta$ as $[a]$.
\end{enumerate}
\end{definition}

Consider the twisted Coxeter element $(s_1 s_2 \cdots s_n) \vee$ of $D_{n+1}$. In this subsection, we denote by the reduced expression of $w_0$
(see Remark \ref{Rem:i_0} below)
$$ \ii_0=\prod_{k=0}^{n} (1\ 2\ \cdots \ n)^{k\vee}$$
where
\begin{align} \label{eq: vee def}
& (j_1 \cdots j_n)^\vee \seteq j^\vee_1 \cdots j^\vee_n \text{ and }
(j_1 \cdots j_n)^{k \vee} \seteq  ( \cdots ((j_1 \cdots j_n \underbrace{ )^\vee )^\vee \cdots )^\vee}_{ \text{ $k$-times} }.
\end{align}

Note that
\begin{enumerate}
\item $(1\ 2\ \cdots \ n)^{k\vee}= \left\{ \begin{array}{ll} 1\ 2\ \cdots n-1\ n & \text{ if } k \text{ is even, } \\ 1\ 2\ \cdots n-1\ n+1 & \text{ if } k \text{ is odd, } \end{array} \right.$
\item $\N \seteq n(n+1) \in 2 \Z_{\ge 1}$ is the cardinality of $|\PR_{D_{n+1}}|$ and coincides with the 2 times of $\mathfrak{n} \seteq |\PR_{A_{n}}|=n(n+1)/2$.
\item $\lf \ii_0 \rf$ is the $r$-cluster point of $\ii_0$.
\end{enumerate}
We can check that $\ii_0$ is a reduced expression of $w_0$ by direct computations.

\begin{remark} \label{Rem:i_0}
By the definition of $\vee$, if we denote
\[ \beta_{p,q}^{\ii_0}= \prod_{k=0}^{p-2}(s_1\ s_2\ \cdots \ s_n)^{k\vee} (s_1 s_2 \cdots s_{q-1})^{(p-1)\vee} (\alpha_{q^{(p-1)\vee}}) \text{ for } p\in \{ 1, \cdots, n+1\},\ q=\{1, \cdots, n\}\]
then $\beta_{1, q}^{\ii_0}= \lan 1, -q-1\ran$, $ \beta_{n+1, q}^{\ii_0}= \lan q, n+1 \ran$
and for  $2\leq p \leq n$
\[ \beta_{p,q}^{\ii_0}=\left\{ \begin{array}{ll}  \lan p, -q-p \ran & \text{ if }p+q \leq n+1, \\
\lan p+q-n-1, p \ran & \text{ if } p+q>n+1.
\end{array}\right. \]
Since $\{\beta_{p,q}^{\ii_0}\}= \Phi^+$, the word $\ii_0$ is a reduced expression of $w_0$. Note that $\lf \ii_0 \rf$ is $\vee$-foldable and
is not adapted to any Dynkin quiver $Q$ of type $D_{n+1}$. We also note that
\begin{align} \label{eq: red 2}
\prod_{k=0}^{n} (1\ 2\ \cdots \ n-1 \ n+1)^{k\vee}
\end{align}
is a reduced expression of $w_0$.
\end{remark}

\begin{example}
The following is the combinatorial AR-quiver of $[\ii_0]$ of type $D_5.$
\begin{equation*}
 \scalebox{0.8}{\xymatrix@C=0.9ex@R=1ex{
&&& \srt{1}{5}\ar@{->}[dr]  && \srt{4}{-5}\ar@{->}[dr]  && \srt{3}{-4}\ar@{->}[dr]  && \srt{2}{-3}\ar@{->}[dr] && \srt{1}{-2} \\
&& \srt{2}{5}\ar@{->}[dr]\ar@{->}[ur]  && \srt{1}{4}\ar@{->}[dr]\ar@{->}[ur] && \srt{3}{-5}\ar@{->}[dr]\ar@{->}[ur] && \srt{2}{-4}\ar@{->}[dr]\ar@{->}[ur] && \srt{1}{-3}\ar@{->}[ur] \\
& \srt{3}{5}\ar@{->}[ddr]\ar@{->}[ur] && \srt{2}{4}\ar@{->}[dr]\ar@{->}[ur] && \srt{1}{3}\ar@{->}[ddr]\ar@{->}[ur] && \srt{2}{-5} \ar@{->}[dr]\ar@{->}[ur] && \srt{1}{-4}\ar@{->}[ur] \\
\srt{4}{5}\ar@{->}[ur] &&&& \srt{2}{3}\ar@{->}[ur] &&&& \srt{1}{-5}\ar@{->}[ur] \\
&& \srt{3}{4}\ar@{->}[uur] &&&& \srt{1}{2} \ar@{->}[uur]
}}
\end{equation*}
\end{example}

\begin{definition} \hfill
\begin{enumerate}
\item The cluster point $\lf \ii_0 \rf$ is called {\it the twisted adapted cluster point} of type $D_{n+1}.$
\item A class $[\ii'_0]\in \lf \ii_0 \rf$ is called a {\it twisted adapted class} of type $D_{n+1}.$
\end{enumerate}
\end{definition}

Consider the map
\[ \mathfrak{p}^{D_{n+1}}_{A_n}: \{ \text{ twisted Coxeter elements of $D_{n+1}$ }\} \to \{ \text{ Coxeter elements of $A_n$ }\} \]
such that $ [i_1 \ i_2 \ \cdots\ i_n]\vee \mapsto  \left\{ \begin{array}{ll}  i_1 \ i_2 \ \cdots\ i_n & \text{ if } i_t=n, \\  (i_1 \ i_2 \ \cdots\ i_n)^\vee & \text{ if } i_t=n+1, \end{array} \right.$
for $t$ such that $i_t\in \{n, n+1\}.$

\begin{proposition} \label{prop: 2 to 1 Dt to A}
The map $ \mathfrak{p}^{D_{n+1}}_{A_n}$ is a two-to-one and onto map.
\end{proposition}

\begin{proof}
Suppose $i_1 \ i_2 \ \cdots\ i_n $ is a Coxeter element of $A_n$. Then both $[i_1 \ i_2 \ \cdots\ i_n]\vee $ and  $[(i_1 \ i_2 \ \cdots\ i_n)^\vee]\vee$ are twisted Coxeter elements of $D_{n+1}.$ Hence we proved the proposition.
\end{proof}

\begin{proposition} \label{prop: one direction for fullfiling}
For a twisted Coxeter element $[s_{i_1}s_{i_2}  \cdots s_{i_n}]\vee$ of $D_{n+1}$, the word  \[ \prod_{i=0}^{n} (i_1\ i_2\ \cdots \ i_n)^{k\vee} \]
 is a reduced expression of $w_0$.
\end{proposition}

\begin{proof}
Assume that $\ii'_0=\prod_{i=0}^{n} (i_1\ i_2\ \cdots \ i_n)^{k\vee}$ is reduced. Then it is easy to check that
\begin{align}\label{ex: reflection red}
\text{ $\big[\prod_{i=0}^{n} ( i_2\ \cdots \ i_n\ i_1^\vee)^{k\vee}\big]$ is also reduced.}
\end{align}
By Proposition \ref{prop: 2 to 1 Dt to A}, our assertion follows from the fact that any Dynkin quiver $Q$ of type $A_n$ can be written as follows:
\begin{itemize}
\item Let us denote by $\overset{\gets}{Q} : \ \xymatrix@R=3ex{ *{ \circ }<3pt> \ar@{-}[r]_<{1}  &*{\circ}<3pt>
\ar@{->}[l]^<{2} &\cdots\ar@{->}[l] &*{\circ}<3pt>
\ar@{->}[l]^<{n-1} &*{\circ}<3pt>
\ar@{->}[l]^<{\ \ n}}$.
\item $Q=i_k \cdots i_1 \overset{\gets}{Q}$ for some $k \in \Z_{\ge 0}$ such that $i_k$ is a sink of the quiver $i_{s-1} \cdots i_{1} \overset{\gets}{Q}$ $(1 \le s \le k)$.
\end{itemize}

\end{proof}

\begin{lemma} \label{Lem: reflection sink i for comclass}
Let $\ii'_0=\prod_{i=0}^{n} (i_1\ i_2\ \cdots \ i_n)^{k\vee}$ where $[s_{i_1}s_{i_2}  \cdots s_{i_n}]\vee$ is a twisted Coxeter element of $D_{n+1}$.
If $i$ is a sink of $[\ii'_0]$ then there is  a reduced expression $\jj=j_1\,  j_2\, \cdots\,  j_n$ such that
\begin{itemize}
\item $[j_1 \ j_2\ \cdots\  j_n]=[i_1\ i_2\ \cdots \ i_n],$
\item $j_1=i,$
\item $ [\ii'_0]= \big[\prod_{i=0}^{n} (j_1 \ j_2\ \cdots \ j_n)^{k\vee}\big]$
\end{itemize}
\end{lemma}

\begin{proof}
The assertion follows from the fact that $s_{i_1}s_{i_2}  \cdots s_{i_n}$ is fully commutative.
\end{proof}

\begin{proposition} \label{prop: reverse direction for fullfiling}
Let $[\ii'_0]$ be a twisted adapted class of type $D_{n+1}$. Then there is a twisted Coxeter element $[i_1\, i_2\,  \cdots \, i_n] \vee$ such that
\[ [\ii'_0]=\big[ \prod_{i=0}^{n} (i_1\ i_2\ \cdots \ i_n)^{k\vee}\big]. \]
\end{proposition}

\begin{proof}
The proof is an immediate consequence of previous lemma.
\end{proof}

\begin{remark}
By Proposition \ref{prop: one direction for fullfiling} and  Proposition \ref{prop: reverse direction for fullfiling},
we can consider $\mathfrak{p}_{A_n}^{D_{n+1}}$ as a two-to-one map between twisted adapted classes of type
$D_{n+1}$ and adapted classes of type $A_n$, i.e., $\lf \ii_0\rf \twoheadrightarrow \lf Q\rf$. Thus, from now on, we abuse the notation
$\mathfrak{p}_{A_n}^{D_{n+1}}$ for the twisted adapted classes of $D_{n+1}$.
\end{remark}

\begin{theorem} \label{Thm:D_twisted adapted class} \hfill
\begin{enumerate}
\item[{\rm (1)}] There is the natural one-to-one correspondence between twisted Coxeter elements and twisted adapted classes of type $D_{n+1}$.
\item[{\rm (2)}] Since the number of adapted classes of type A$_n$ is $2^{n-1}$, the number of classes in the twisted adapted cluster point of type $D_{n+1}$ is $2^{n}.$
\end{enumerate}
\end{theorem}

\begin{remark}
In the twisted adapted cluster point of type $A_{2n+1}$ \cite{OS16}, there are some classes which are not associated to twisted Coxeter elements.
The number of twisted adapted classes of type $A_{2n+1}$ is $2^{2n}$ and the number of twisted Coxeter element of type $A_{2n+1}$ is $3^{n-1}\cdot 4.$
Hence Theorem \ref{Thm:D_twisted adapted class} is a special property for the type $D_{n+1}$ case.
\end{remark}

\section{Twisted and folded AR-quivers and their labeling via AR-quivers of type $A_{n}$} \label{Section:twisted_AR_D_label}

\subsection{ AR-quivers of type $A_{n}$ } In this subsection, we briefly review the combinatorial properties of AR-quiver for a Dynkin quiver $Q$ of type
$A_n$, which were studied in \cite{Oh14A}.

\begin{definition} \cite[Definition 1.6]{Oh14A} Fix any class $[\jj_0]$ of $w_0$ of any finite type.
\begin{enumerate}
\item[{\rm (a)}] A path in $\Upsilon_{[\jj_0]}$ is {\it $N$-sectional} (resp. {\it $S$-sectional})
if it is a concatenation of upward arrows (resp. downward arrows).
\item[{\rm (b)}] An $N$-sectional (resp. $S$-sectional) path $\rho$ is {\it maximal} if there is no longer
$N$-sectional (resp. $S$-sectional) path containing $\rho$.
\item[{\rm (c)}] For a sectional path $\rho$, the {\it length} of $\rho$ is the number of all arrows in $\rho$.
\end{enumerate}
\end{definition}

\begin{proposition} \label{pro: section shares} \cite[Proposition 4.5]{OS15}
Fix any class $[\jj_0]$ of $w_0$ of type $A_n$.
Let $\rho$ be an $N$-sectional $($resp. $S$-sectional$)$ path in $\Upsilon_{[\jj_0]}$. Then every positive roots contained in
$\rho$ has the same first $($resp. second$)$ component.
\end{proposition}

\begin{theorem} \label{thm: labeling GammaQ}
\cite[Corollary 1.12]{Oh14A} Fix any Dynkin quiver $Q$ of type $A_n$.
For $1 \le i \le n$, the AR quiver $\Gamma_Q$ contains a maximal $N$-sectional path of length $n-i$ once and exactly once whose vertices share $i$ as the first component.
At the same time, $\Gamma_Q$ contains a maximal $S$-sectional path of length $i-1$ once and exactly once whose vertices share $i$ as the second component.
\end{theorem}

Thus, for any AR-quiver $\Gamma_Q$ of type $A_n$, we say that an $N$-sectional path $\rho$ is the {\it maximal
$N$-sectional path of length $k$} $(0 \le k \le n-1)$ if all positive roots contained in $\rho$ have $n-k$ as first components.
Similarly, we can define a notion of the {\it maximal $S$-sectional path of length $k$}  $(0 \le k \le n-1)$ whose positive roots
contain $k+1$ as second components.
Note that, in this paper, we only need maximal sectional paths (see Section \ref{Sec:twisted_AR_D_label_shape}).

\begin{lemma} \cite[Lemma 3.2.3]{KKK13b} \label{lem: Leclerc}
For any Dynkin quiver $Q$ of type $AD_m$, the positions of simple roots inside of $\Gamma_Q$ are on the boundary of $\Gamma_Q$;
that is, either
\begin{itemize}
\item[{\rm (i)}] $\al_i$ is a sink or a source of $\Gamma_Q$, or
\item[{\rm (ii)}] residue of $\al_i$ is
$\begin{cases}
\text{$1$ or $n$} & \text{ if $Q$ is of type $A_n$,} \\
\text{$1$, $n$ or $n+1$} & \text{ if $Q$ is of type $D_{n+1}$.}
\end{cases}$
\end{itemize}
Furthermore,
\begin{enumerate}
\item[{\rm (a)}] all sinks and sources of $\Gamma_Q$ have their labels as simple roots,
\item[{\rm (b)}] $i$ is a sink $($resp. source$)$ of $Q$ if and only if $\al_i$ is a sink $($resp. source$)$ of $\Gamma_Q$
\end{enumerate}
\end{lemma}

\subsection{ Twisted and folded AR-quivers of type $D_{n+1}$ }
In this subsection, we describe a coordinate to a combinatorial AR-quiver of a twisted adapted class.
For the purpose, we denote by $\ii_0=\Pi_{k=0}^n (i_1\, i_2\, \cdots\, i_n)^{k\vee}$ a reduced expression in a
twisted adapted class.

\begin{definition}
If $[\ii_0]$ is a twisted adapted class then the combinatorial AR-quiver $\Upsilon_{[\ii_0]}$ is called a {\it twisted AR-quiver}.
\end{definition}

Let us denote the twisted Coxeter element of $[\ii_0]$ by $ \phi_{[\ii_0]}\vee=(s_{i_1}\cdots s_{i_n})\vee$ and
\begin{align} \label{eq: MB}
\text{ $\mathfrak{B}_{[\ii_0]} \seteq\Phi(\phi_{[\ii_0]})= \{ \beta^{\phi_{[\ii_0]}}_k=s_{i_1}s_{i_2}\cdots s_{i_{k-1}}(\alpha_{i_k})$ for $k=1,2, \cdots, n \}$}
\end{align}
Consider the Dynkin quiver $Q$ of type $A_n$ such that $\mathfrak{p}^{D_{n+1}}_{A_n}([\ii_0])=[Q].$ Now we suggest an algorithm for
the twisted Auslander-Reiten quiver $\Upsilon_{[\ii_0]}$ with the coordinate system as follows:

\begin{algorithm}  \label{Alg_twAR} \hfill
\begin{enumerate}
\item[{\rm (a)}] Let us define a {\it height function} $\xi: \overline{I}_{n+1}^D\to \Z$ where $I^D_{n+1}$ is the index set of $D_{n+1}$ and
$\overline{I}^D_{n+1}=\{\bar{1}, \bar{2}, \cdots, \bar{n}\}$ as follows : $\xi(\bar{i})=\xi(\bar{j})+1$ if there is an arrow from $j$ to $i$ in $Q$.
Note that $\xi$ is unique up to constant.
\item[{\rm (b)}] Take the injection $\widetilde{\phi}_{[\ii_0]}:\mathfrak{B}_{[\ii_0]}\to I\times \Z$ such that $\beta_k^{\phi_{[\ii_0]}}\mapsto ( i_k , \xi(\bar{i_k})).$
\item[{\rm (c)}] Let us denote $\psi_{[\ii_0]}= s_{i_1}s_{i_2}\cdots s_{i_n}  s_{i_1^\vee}s_{i_2^\vee}\cdots s_{i_n^\vee}$ and
$\beta^{\psi_{[\ii_0]}}_{n+k}=\phi_{[\ii_0]} s_{i_1^\vee}s_{i_2^\vee}\cdots s_{i_{k-1}^\vee}(\alpha_{i_k^\vee})$.
We extend the map $\widetilde{\phi}_{[\ii_0]}$ to the map on $\Phi(\Psi_{[\ii_0]})$ satisfying $\beta^{\psi_{[\ii_0]}}_{n+k}\mapsto (i_k^\vee, \xi(\bar{i_k})-2)$.
\item[{\rm (d)}] Extend the map $\widetilde{\phi}_{[\ii_0]}$ to the map on $\Phi^+$ satisfying if  both $\beta$ and $ \psi_{[\ii_0]}(\beta)$
are positive roots and $\widetilde{\phi}_{[\ii_0]}(\beta)=(i,p)$ then $\widetilde{\phi}_{[\ii_0]}(\psi_{[\ii_0]}(\beta))=(i, p-4).$
\item[{\rm (e)}] If {\rm (i)} $(i,p),(j,q)\in {\rm Im}(\widetilde{\phi}_{[\ii_0]})$,
{\rm (ii)} $|j-i|=1$ or $(i,j)\in \{(n-1,n+1), (n+1, n-1)\}$, {\rm (iii)} $q=p+1$ then there is an arrow $(i,p)\to (j,q).$
\end{enumerate}
\end{algorithm}

\begin{remark} The above algorithm is well-defined by \eqref{ex: reflection red} and
\begin{eqnarray*} &&
\parbox{70ex}{ for each $\be \in \PR$ and $[\ii_0]$ of $w_0$, {\it the residue} $i$ of $\be$ with respect to $[\ii_0]$ is well-assigned.}
\end{eqnarray*}
Also, the algorithm can be understood as a generalization of \cite[\S 2.2]{HL11} given in Definition \ref{def: Gamma_Q}.
\end{remark}

\begin{example} \label{Exam: D5 easy tAr}
The twisted AR-quiver $\Upsilon_{[\ii_0]}$ for $\ii_0=\prod_{k=0}^4 ( 4\, 3\, 2\, 1)^{k\vee}$ with coordinates  can be depicted as follows:
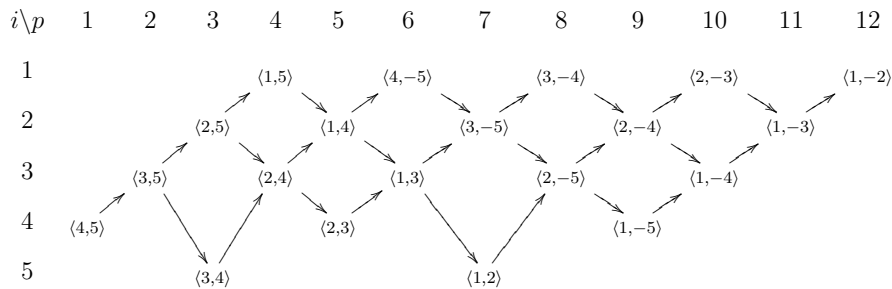
\begin{equation*}
 \scalebox{0.8}{\xymatrix@C=0.9ex@R=1ex{
 i \backslash p & 1& 2& 3& 4& 5& 6& 7& 8 & 9 & 10 & 11 & 12\\
1&&&& \srt{1}{5}\ar@{->}[dr]  && \srt{4}{-5}\ar@{->}[dr]  && \srt{3}{-4}\ar@{->}[dr]  && \srt{2}{-3}\ar@{->}[dr] && \srt{1}{-2} \\
2&&& \srt{2}{5}\ar@{->}[dr]\ar@{->}[ur]  && \srt{1}{4}\ar@{->}[dr]\ar@{->}[ur] && \srt{3}{-5}\ar@{->}[dr]\ar@{->}[ur] && \srt{2}{-4}\ar@{->}[dr]\ar@{->}[ur] && \srt{1}{-3}\ar@{->}[ur] \\
3 && \srt{3}{5}\ar@{->}[ddr]\ar@{->}[ur] && \srt{2}{4}\ar@{->}[dr]\ar@{->}[ur] && \srt{1}{3}\ar@{->}[ddr]\ar@{->}[ur] && \srt{2}{-5} \ar@{->}[dr]\ar@{->}[ur] && \srt{1}{-4}\ar@{->}[ur] \\
4&\srt{4}{5}\ar@{->}[ur] &&&& \srt{2}{3}\ar@{->}[ur] &&&& \srt{1}{-5}\ar@{->}[ur] \\
5&&& \srt{3}{4}\ar@{->}[uur] &&&& \srt{1}{2} \ar@{->}[uur]
}}
\end{equation*}
Here $\xi$ is defined by $\xi(\bar{1})= 12$, $\xi(\bar{2})=11$ , $\xi(\bar{3})= 10$ and $\xi(\bar{4})= 9,$ and $\mathfrak{p}^{D_{n+1}}_{A_n}([\ii_0])=[Q]$ where
\begin{equation*}
 Q= \xymatrix@R=4ex{
*{\circ}<3pt> \ar@{<-}[r]_<{1 \ \ }  &*{\circ}<3pt>
\ar@{<-}[r]_<{2 \ \ } & *{\circ}<3pt>
\ar@{<-}[r]_<{3 \ \ }  &*{\circ}<3pt> \ar@{-}[l]^<{\  \ 4  }}.
\end{equation*}
The labels on $(k, \xi(\bar{k}))$ and $(k^\vee, \xi(\bar{k})-2)$ for $k=1,2,3,4$ are determined by (b) and (c), respectively, in Algorithm \ref{Alg_twAR} and the rest of labels  are determined by (d).
\end{example}

In the sense of \cite[Section 7]{OS16}, we introduce {\it the folded Auslander-Reiten quiver} $\widehat{\Upsilon}_{[\ii_0]}$ of a twisted adapted class
$[\ii_0]$:

\begin{enumerate}
\item In the folded AR-quiver,  we assign {\it folded coordinate} $(\bar{i}, p, (-1)^{\delta_{i, n+1}})\in \overline{I} \times \Z \times \{\pm 1\}$ for the positive root $\beta$ such that $\widetilde{\phi}_{[\ii_0]}(\beta)=(i, p).$
\item There is an arrow from $(\bar{i}, p, (-1)^{\delta_{i, n+1}})\to(\bar{j}, q, (-1)^{\delta_{j, n+1}})$ if and only if there is an arrow $(i,p)\to (j,q)$ in $\Upsilon_{[\ii_0]}.$
\end{enumerate}

In $\widehat{\Upsilon}_{[\ii_0]}$, we often omit the third coordinate and use only the first and second coordinates, $(\bar{i}, p).$ It is not hard to see that the set of coordinates  which are assigned to positive roots is
\begin{align} \label{eq: r_i for twisted D_n+1}
\mathcal{\bar{I}}:=\{\, (\, \bar{i}, \xi(\bar{i})-2t\, )\, | \, t=0, 1, \cdots, n\, \}.
\end{align}
Also, there is an arrow $(\bar{i}, p) \to (\bar{j}, q)$ for $(\bar{i}, p), (\bar{j}, q)\in \mathcal{\bar{I}}$ if and only if  there are $i\in \bar{i}$ and $j\in \bar{j}$ such that $|\bar{i}-\bar{j}|=1$ and
$q=p+1$. We call $\overline{i}$ {\it the folded residue of $\be$ with respect to $[\ii_0]$}
when $\be$ is located at $(\overline{i},p)$ in $\widehat{\Upsilon}_{[\ii_0]}$.

\begin{example} \label{ex: folded}
Let $[\ii'_0]$ be the one in Example \ref{Exam: D5 easy tAr}. Then the folded AR-quiver $\widehat{\Upsilon}_{[\ii'_0]}$ with coordinates is
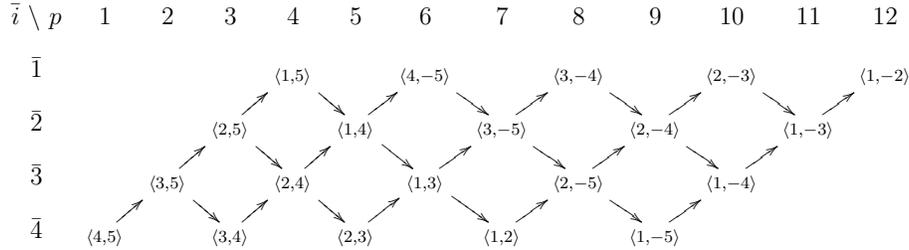
\begin{equation*}
 \scalebox{0.8}{\xymatrix@C=0.9ex@R=1ex{
 \bar{i} \ \backslash\ p & 1& 2& 3& 4& 5& 6& 7& 8 & 9 & 10 & 11 & 12\\
\bar{1}&&&& \srt{1}{5}\ar@{->}[dr]  && \srt{4}{-5}\ar@{->}[dr]  && \srt{3}{-4}\ar@{->}[dr]  && \srt{2}{-3}\ar@{->}[dr] && \srt{1}{-2} \\
\bar{2}&&& \srt{2}{5}\ar@{->}[dr]\ar@{->}[ur]  && \srt{1}{4}\ar@{->}[dr]\ar@{->}[ur] && \srt{3}{-5}\ar@{->}[dr]\ar@{->}[ur] && \srt{2}{-4}\ar@{->}[dr]\ar@{->}[ur] && \srt{1}{-3}\ar@{->}[ur] \\
\bar{3} && \srt{3}{5}\ar@{->}[dr]\ar@{->}[ur] && \srt{2}{4}\ar@{->}[dr]\ar@{->}[ur] && \srt{1}{3}\ar@{->}[dr]\ar@{->}[ur] && \srt{2}{-5} \ar@{->}[dr]\ar@{->}[ur] && \srt{1}{-4}\ar@{->}[ur] \\
\bar{4}&\srt{4}{5}\ar@{->}[ur] && \srt{3}{4}\ar@{->}[ur]&& \srt{2}{3}\ar@{->}[ur] &&\srt{1}{2} \ar@{->}[ur] && \srt{1}{-5}\ar@{->}[ur]
}}
\end{equation*}
\end{example}

\begin{remark} \label{rem: decompose into 2-comp}
Let $[\ii_0]$ be a twisted adapted class with $\mathfrak{p}^{D_{n+1}}_{A_n}([\ii_0])=[Q]$. Then one can easily check that
\begin{eqnarray*} &&
\parbox{80ex}{ {\rm (i)} the subquiver consisting of $\mathfrak{B}_{[\ii_0]}$ inside $\widehat{\Upsilon}_{[\ii_0]}$
is isomorphic to $Q$ as a Dynkin quiver of type $A_n$.}
\end{eqnarray*}
By \eqref{eq: r_i A_n} and \eqref{eq: r_i for twisted D_n+1}, we can take
\begin{eqnarray*} &&
\parbox{80ex}{ {\rm (ii)} the full subquiver ${}_1\Gamma_Q$ inside $\Upsilon_{[\ii_0]}$ such that
it contains $\mathfrak{B}_{[\ii_0]}$ and is isomorphic to $\Gamma_Q$ as quivers.
}
\end{eqnarray*}
Since the subquiver consisting of
leftmost vertices inside of $\Gamma_Q$ is isomorphic to $Q^{*}$,
\begin{eqnarray*} &&
\parbox{80ex}{ {\rm (iii)} the full subquiver ${}_2\Gamma_{Q^*} \seteq \Upsilon_{[\ii_0]} \setminus {}_1\Gamma_Q$
is isomorphic to $\Gamma_{Q^{*}}$,}
\end{eqnarray*}
where $Q^{*}$ is a Dynkin quiver obtained by swapping vertices $i \leftrightarrow n+1-i$ and $ \Upsilon_{[\ii_0]} = {}_1\Gamma_Q \sqcup {}_2\Gamma_{Q^*}$
as sets of vertices.
\end{remark}

\subsection{Labeling of twisted and folded AR-quivers of $D_{n+1}$ via AR-quivers of $A_{n}$}

\begin{proposition} \label{Prop:label come from 1GaQ}
Let $[\ii_0]$ be a twisted adapted class. Suppose $\jj_0= j_1 j_2 \cdots j_\N\in [\ii_0]$ and let $t_1$ and $t$ be the smallest and second smallest integers such that
$j_{t_1}, j_t\in \bar{n}$. If $j_{t_1}=n$ $($resp. $j_{t_1}=n+1)$  and
$ \mathfrak{p}_{A_n}^{D_{n+1}}([\ii_0])= [Q] $
then there is a reduced expression $\kk_0\in [Q]$ of $A_n$ starting with $j_1\, j_2\, \cdots\, j_{t-1}$ $($resp. $(j_1\, j_2\, \cdots\, j_{t-1})^\vee$$)$
\end{proposition}

\begin{proof}
Note that if $j_{t_1}=n$ (resp. $j_{t_1}=n+1$), then $j_t=n+1$ (resp. $j_t=n$).
By Theorem \ref{thm: OS14} and Remark \ref{rem: decompose into 2-comp},
the assertion comes from the compatible readings of (i) $\jj_0$ inside of ${}_1\Gamma_Q$ first, and then (ii) ${}_1\Gamma_Q \setminus \jj_0$ inside of ${}_1\Gamma_Q$.
Here, all vertices in the $n$-th and the $n+1$-th-resides have to be read as $n$.
\end{proof}

\begin{remark} \label{rem: jj_0}
In the above proposition, we can assume that
$\jj_0$ is a reduced expression in $[\ii_0]$ with the largest $t$ without any loss of generality.
By Theorem \ref{thm: OS14} and Theorem \ref{thm: labeling GammaQ}, such $\jj_0$ can be obtained by the following way:
\begin{eqnarray} &&
\parbox{80ex}{
(i) Read the residues of all vertices in the right part of the length $n-1$ $S$-sectional path $\rho$  inside ${}_1\Gamma_Q$. (ii) Read residues on the path $\rho$. (iii) Read the remained.
}\label{eq:jj1 reading}
\end{eqnarray}
(See the labeled vertices in \eqref{eq:jj_1}.)
We denote by $\Upsilon^R_{[\ii_0]}$ the full subquiver consisting of the vertices  in (i) and (ii) in  \eqref{eq:jj1 reading}
\end{remark}

\begin{corollary}  \label{Cor:label come from 1GaQ}
Let $\ii_0$, $\jj_0$, $Q$ and $t$ be in {\rm Proposition \ref{Prop:label come from 1GaQ}} and $j_{t_1}=n$ $($resp. $j_{t_1}=n+1$$)$.
The subquiver $\Upsilon^R_{[\ii_0]}$ of twisted AR-quiver $\Upsilon_{[\ii_0]}$
is isomorphic to a subquiver $\Gamma^R_{Q}$ of the AR-quiver $\Gamma_Q$. Moreover, labels in
$\Upsilon^R_{[\ii_0]}$ direct follow from those in $\Gamma^R_{Q}$. Precisely, correspondences between the labels of $\Gamma^R_{Q}$ and $\Upsilon^R_{[\ii_0]}$ are as follows:
\[ \, [i,j] \mapsto \left\{ \begin{array}{ll}\left<i,-j-1\right> & \text{ if $i,j\neq n$ } \\
\left<i, -n-1\right> \ ( \text{resp.}\left<i, n+1\right>)  & \text{ if $j=n$.}  \end{array}\right.\]
\end{corollary}

\begin{proof}
If $j_{t_1}=n$, our assertion directly follows from Proposition \ref{Prop:label come from 1GaQ}.
If $j_{t_1}=n+1$, our assertion can be obtained by swapping the roles of $\al_n$ and $\al_{n+1}$.
\end{proof}

\begin{example}\label{Ex:label come from 1GaQ}
Let $\ii_0= \prod_{k=0}^4 (2\ 1\ 3\ 5)^{k\vee}$. Then we can find labels using Corollary \ref{Cor:label come from 1GaQ} as follows:
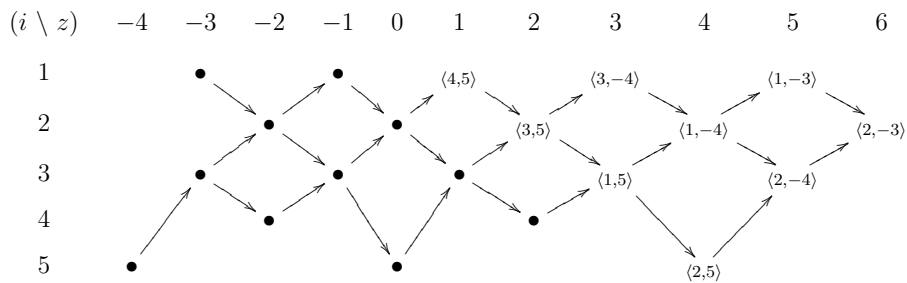
\[  \scalebox{0.8}{\xymatrix@C=2ex@R=1ex{
(i \  \backslash \ z)&-4 & -3 & -2 & -1 & 0 & 1 & 2& 3& 4& 5& 6 \\
1&&\bullet\ar@{->}[dr]  && \bullet\ar@{->}[dr]  &&  \srt{4}{5}\ar@{->}[dr]  && \srt{3}{-4}\ar@{->}[dr] && \srt{1}{-3} \ar@{->}[dr]\\
2&&& \bullet\ar@{->}[dr]\ar@{->}[ur]  &&\bullet\ar@{->}[dr]\ar@{->}[ur]  && \srt{3}{5}\ar@{->}[dr]\ar@{->}[ur]  && \srt{1}{-4}\ar@{->}[dr]\ar@{->}[ur] && \srt{2}{-3} \\
3& & \bullet\ar@{->}[dr]\ar@{->}[ur]  &&\bullet\ar@{->}[ddr]\ar@{->}[ur]  &&\bullet\ar@{->}[dr]\ar@{->}[ur]  && \srt{1}{5}\ar@{->}[ddr]\ar@{->}[ur] && \srt{2}{-4}\ar@{->}[ur] \\
4&&& \bullet\ar@{->}[ur]  &&&& \bullet\ar@{->}[ur] \\
5&\bullet\ar@{->}[uur]  &&&& \bullet\ar@{->}[uur]  &&&& \srt{2}{5}\ar@{->}[uur]
}}\]
The labels of ${}_1\Gamma_Q$ can be obtained from $\Gamma_Q$ where
\begin{align} \label{eq:jj_1}
Q = {\xymatrix@R=3ex{ \circ
\ar@{->}[r]_<{  1} &  \circ
\ar@{<-}[r]_<{  2}  &  \circ
\ar@{<-}[r]_<{  3}
& \circ \ar@{-}[l]^<{\ \ \ \ 4} }}  \text{ and }
 \Gamma_Q=    \scalebox{0.8}{\raisebox{3.6em}{\xymatrix@C=2ex@R=1ex{
(i \  \backslash \ z)  & 1 & 2& 3& 4& 5& 6 \\
1 &  [4]\ar@{->}[dr]  && [3]\ar@{->}[dr] && [1,2] \ar@{->}[dr]\\
2 && [3,4]\ar@{->}[dr]\ar@{->}[ur]  && [1,3]\ar@{->}[dr]\ar@{->}[ur] &&[2] \\
3 &  && [1,4]\ar@{->}[dr]\ar@{->}[ur] && [2,3]\ar@{->}[ur] \\
4 && \bullet \ar@{->}[ur] && [2,4]\ar@{->}[ur]
}}}.
\end{align}
(Here the path consisting of $\{ [4],[3,4],[1,4],[2,4]\}$ is the $\rho$ in \eqref{eq:jj1 reading}.) Hence
\[ Q^{*} = {\xymatrix@R=3ex{ \circ
\ar@{->}[r]_<{ 1} &  \circ
\ar@{->}[r]_<{  2}  &  \circ
\ar@{-}[l]^<{ \ \ \  3}
& \circ \ar@{->}[l]^<{\ \ \  4} }} \text{ and }
 \Gamma_{Q^*} \simeq {}_2\Gamma_{Q^*} \simeq   \scalebox{0.8}{\raisebox{2.5em}{  \xymatrix@C=2ex@R=1ex{
&\bullet\ar@{->}[dr]  && \bullet\ar@{->}[dr]  \\
&& \bullet\ar@{->}[dr]\ar@{->}[ur]  &&\bullet\ar@{->}[dr] \\
 & \bullet\ar@{->}[dr]\ar@{->}[ur]  &&\bullet\ar@{->}[dr]\ar@{->}[ur]&&\bullet\\
\bullet\ar@{->}[ur]&& \bullet\ar@{->}[ur] && \bullet\ar@{->}[ur]
}}}\]
as quivers.
\end{example}

\begin{remark} \label{rem: reflection w.r.t x-axis}
By taking reflection with respect to $x$-axis to ${}_2\Gamma_{Q^*}$, one can easily check that
${}_2\Gamma_{Q^*}$ is isomorphic to $\Gamma_Q$ as quivers. Then the $S$-sectional path (resp. $N$-sectional path) in ${}_2\Gamma_{Q^*}$ of length $k$
corresponds to the $N$-sectional path (resp. $S$-sectional path) in ${}_1\Gamma_Q$ of length $k$.
\end{remark}

\begin{eqnarray} &&
\parbox{80ex}{
Recall that there is an involution $*$ on the index set $I^{D}_{n+1}$ such that  $ w_0(\alpha_i)=-\alpha_{i^*}.$
If $n+1$ is odd then $* : i\mapsto \left\{ \begin{array}{ll} i & \text{ if } i\neq n, n+1, \\ i+(-1)^{\delta_{n+1, i}} & \text{ if } i=n, n+1. \end{array} \right.$ On the other hand,
if $n+1$ is even then $ i^*=i , \text{ for } i\in I.$
}\label{eq: involution D}
\end{eqnarray}

For any reduced expression $\ii_0 =i_1\ i_2\ \cdots i_{\N}$ of $w_0$ of type $D_{n+1}$,
\[ s_{i_1}s_{i_2}  \cdots s_{i_p}= w_0 s_{i_\N} s_{i_{\N-1}} \cdots s_{i_{p+1}}.\]
Since $s_{i_p}(\alpha_{i_p})= -\alpha_{i_p}$, we have $s_{i_1}s_{i_2}  \cdots s_{i_{p}}(\alpha_{i_p})=-\beta^{\ii_0}_p$  and
 \[ -\beta^{\ii_0}_p=w_0 s_{i_\N} s_{i_{\N-1}} \cdots s_{i_{p+1}} (\alpha_p)= - ( s_{i_\N} s_{i_{\N-1}} \cdots s_{i_{p+1}} (\alpha_{i_p}))^*,\]
where $(\alpha+\beta)^*:= \alpha^* +\beta^*$ for $\alpha, \beta\in \Phi^+$ and $-w_0(\al_i)=(\alpha_i)^*= \alpha_{i^*}.$
Hence
\begin{equation} \label{Eqn:beta_rev}
 s_{i_1}s_{i_2}  \cdots s_{i_{p-1}}(\alpha_{i_p})= (s_{i_\N} s_{i_{\N-1}} \cdots s_{i_{p+1}} (\alpha_{i_p}))^*=
 s_{{i_\N}^*} s_{{i_{\N-1}}^*}\cdots s_{i_{{p+1}}^*}(\alpha_{{i_p^*}}) .
 \end{equation}

 \begin{proposition} \label{Prop:label come from 2GaQ}
Let $[\ii_0]$ be a twisted adapted class. For $\jj'_0= j_1 j_2 \cdots j_\N\in [\ii_0]$, suppose  $t_1$ and $t$  are the largest integer and  second largest integers such that $j_{t_1}, j_t\in\bar{n}$. Let $j_{t_1}=n$ $($resp. $j_{t_1}=n+1)$. Consider the Dynkin quiver $Q$ of $A_n$ satisfying
$\mathfrak{p}_{A_n}^{D_{n+1}}([\ii_0])= [Q]$
and $Q^{{\rm rev}}$, the Dynkin quiver which is obtained by reversing every arrow in $Q$.
Then there is a reduced expression $\kk_0\in [Q^{{\rm rev}}]$ of $A_n$ starting with $j_\N j_{\N-1} \cdots j_{t+1}$ $($resp. $(j_\N j_{\N-1} \cdots j_{t+1})^\vee$ $).$
\end{proposition}

\begin{proof}
Consider $[\ii'_0]$ such that $j_\N j_{\N-1} \cdots j_1 \in [\ii'_0]$. Note the following facts:
\begin{itemize}
\item If  $[i_1 i_2 \cdots i_n]\vee$ is a twisted Coxeter element then so is $[i_n i_{n-1} \cdots i_1]\vee$.
\item If $\mathfrak{p}_{A_n}^{D_{n+1}}([i_1 i_2 \cdots i_n]\vee)=\phi_Q$, then $\mathfrak{p}_{A_n}^{D_{n+1}}([i_n i_{n-1} \cdots i_1]\vee)=\phi_{Q^{{\rm rev}}}$.
\end{itemize}
Thus, by Theorem \ref{Thm:D_twisted adapted class}, we have $[\ii'_0]$ is twisted adapted and $\mathfrak{p}_{A_n}^{D_{n+1}}([\ii'_0])=[Q^{{\rm rev}}]$.
By the same argument as Proposition \ref{Prop:label come from 1GaQ}, we can prove the proposition.
\end{proof}

\begin{remark} \label{rem: jj_0 prime}
In the above proposition, we can assume that
$\jj'_0$ is a reduced expression in $[\ii_0]$ with the smallest $t$ without any loss of generality as in Remark \ref{rem: jj_0}.
By Theorem \ref{thm: OS14} and Theorem \ref{thm: labeling GammaQ}, such $\jj_0$ can be obtained by the following way:
\begin{eqnarray} &&
\parbox{80ex}{
(i) Read the residues of vertices except the residues of all vertices in the left part of length $n-1$ $N$-sectional path $\rho'$
inside ${}_2\Gamma_{Q^*}$ (ii) Read vertices in $\rho'$, (iii) Read the remained.
}\label{eq:jj2 reading}
\end{eqnarray}
(See the vertices $\{ \langle 1 ,-2 \rangle, \langle 1,-5 \rangle, \langle 2 ,-5\rangle,
\langle 3,-5\rangle, \langle 4, -5 \rangle \}$ in Example \ref{Ex:label 1+2}.)
We denote by $\Upsilon^L_{[\ii_0]}$ the full subquiver consisting of the vertices described in (i) and (ii) of \eqref{eq:jj2 reading}
\end{remark}

\begin{example}
For $\ii_0=\prod_{k=0}^4 (2\ 1\ 3\ 5)^{k\vee} \in [\ii_0]$ in Example \ref{Ex:label come from 1GaQ},
$\Upsilon_{[\ii_0']}$ for $\ii_0'=\prod_{k=0}^4 (5\ 3\ 1\ 2)^{k\vee}$ can be described as follows:
\[  \scalebox{0.8}{\xymatrix@C=2ex@R=1ex{
1& & \bullet\ar@{->}[dr] && \bullet\ar@{->}[dr]&& \bullet\ar@{->}[dr]&& \bullet\ar@{->}[dr]&& \bullet\\
2& \bullet \ar@{->}[dr]\ar@{->}[ur]&& \bullet\ar@{->}[dr]\ar@{->}[ur]&& \bullet\ar@{->}[dr]\ar@{->}[ur]&& \bullet\ar@{->}[dr]\ar@{->}[ur]&& \bullet\ar@{->}[dr]\ar@{->}[ur]\\
3& & \bullet\ar@{->}[ur]\ar@{->}[ddr] && \bullet\ar@{->}[ur]\ar@{->}[dr]&& \bullet\ar@{->}[ur]\ar@{->}[ddr]&& \bullet\ar@{->}[ur]\ar@{->}[dr]&& \bullet\ar@{->}[ddr]\\
4& &&&& \bullet\ar@{->}[ur] &&&& \bullet\ar@{->}[ur]\\
5& && \bullet\ar@{->}[uur] &&&& \bullet\ar@{->}[uur] &&&& \bullet
}}\]
where
\begin{align} \label{eq:jj_2}
Q^{{\rm rev}} = {\xymatrix@R=3ex{ \circ
\ar@{-}[r]_<{  1} &  \circ
\ar@{->}[l] \ar@{->}[r]_<{  2} &  \circ
\ar@{->}[r]_<{  3}
& \circ \ar@{-}[l]^<{\ \ \ \ 4} }}  \text{ and }
\Gamma_{Q^{{\rm rev}}} \scalebox{0.8}{\raisebox{3.6em}{\xymatrix@C=2ex@R=1ex{
(i \  \backslash \ z)&0& 1& 2& 3& 4& 5 \\
1  &  && [2,4]\ar@{->}[dr] && [1]\\
2 && \bullet \ar@{->}[dr]\ar@{->}[ur]  && [1,4]\ar@{->}[dr]\ar@{->}[ur] &&\\
3  &\bullet \ar@{->}[dr] \ar@{->}[ur]&& \bullet\ar@{->}[dr]\ar@{->}[ur] && [3,4]\ar@{->}[dr]&&\\
4 &&\bullet \ar@{->}[ur] && \bullet \ar@{->}[ur]   && [4] &&
}}}
\end{align}
\end{example}

\begin{corollary}\label{Cor:label come from 2GaQ}
Let $\ii_0$, $\jj_0$, $Q$ and $t$ be in {\rm Proposition \ref{Prop:label come from 2GaQ}} and $j_{t_1}=n^*$. The subquiver $\Upsilon^{L}_{[\ii_0]}$ of twisted AR-quiver $\Upsilon_{[\ii_0]}$
 is isomorphic to a subquiver $\Gamma^L_{Q^{{\rm rev}}}$  inside of $\Gamma_{Q^{{\rm rev}}}$.
 Moreover, labels in $\Upsilon^L_{[\ii_0]}$ directly follow from those in $\Gamma^L_{Q^{{\rm rev}}}$ in the same way of
{\rm  Corollary \ref{Cor:label come from 1GaQ}}.
\end{corollary}

\begin{proof}
It directly follows from Proposition \ref{Prop:label come from 2GaQ}.
\end{proof}

The following lemma is easy to check, we leave the proof for readers.

\begin{lemma} \label{Lem:A longest} Let $\ii^A_0=i_1\, i_2\, \cdots \, i_{\n-1} \, i_{\n}$ be a reduced expression of type $A_n$.
\begin{enumerate}
\item[{\rm (1)}] We have $\beta^{\ii^A_0}_p=s_{i^*_{\n}} s_{i^*_{\n-1}}\cdots s_{i^*_{p+1}} (\alpha_{i^*_p}).$
\item[{\rm (2)}] If $\ii^A_0$ is adapted to $Q$ with the Coxeter element  $\phi_{Q}=s_{i'_1}s_{i'_2}\cdots s_{i'_{n-1}} s_{i'_n}$
then  $(\ii^A_0)^{-1}:= i_{\n}\, i_{\n-1}\, \cdots\, i_2\, i_1$ is adapted to $Q^{{\rm rev}*}$ with the Coxeter element
$\phi_{Q^{{\rm rev}*}}=s_{i'^*_n}s_{i'^*_{n-1}}\cdots s_{i'^*_2} s_{i'^*_1}.$
\item[{\rm (3)}] If $\ii^A_0$ is adapted to the Dynkin quiver $Q$ then $(\ii^{A}_0)^{-1*}=i_{\n}^*\, i_{\n-1}^*\, \cdots\, i_2^* \, i_1^*$, for $*: I_n\to I_n$
such that $i \mapsto n+1-i,$  adapted to the Dynkin quiver $Q^{{\rm rev}}$.
\item[{\rm (4)}]  Let $\ii_0$, $\jj_0$, $Q$ and  $t$ be in {\rm Proposition \ref{Prop:label come from 2GaQ}}
and let $j_{t_1}=n$. There is an adapted reduced expression $\jj^A_0$ to $Q$ which ends with
$j_{t+1}^* \, j_{t+2}^*\, \cdots\, j_{l-1}^* \, j_l^*$, where $*: I_n\to I_n$, $i \mapsto n+1-i.$
\end{enumerate}
\end{lemma}

\begin{corollary} \label{Cor:rev Ga rev Q}
Let $Q$ be in {\rm Proposition \ref{Prop:label come from 2GaQ}}. We obtain $\Gamma_{Q^{{\rm rev}}}$ from $\Gamma_Q$ by reversing all the arrow in $\Gamma_Q$.
Hence $\Gamma^{{\rm rev}}_{Q^{{\rm rev}}}$ is isomorphic to $\Gamma_Q$ via the map $\alpha\mapsto\alpha$ for
$\alpha\in \Phi^+.$ Also, if $\alpha$ is in the $i$-th residue in $\Gamma_Q$ then $\alpha$ is in the $(n+1-i)$-th residue in $\Gamma_{Q^{{\rm rev}}}$.
\end{corollary}
\begin{proof}
It follows from (1) and (3) in Lemma \ref{Lem:A longest}.
\end{proof}

\begin{corollary}\label{Cor:label 2}
Let $t_1$ and $Q$ be in {\rm Proposition \ref{Prop:label come from 2GaQ}}. We obtain labels of $\Upsilon^L_{[\ii_0]}$ from the labels of $\Gamma_{Q}^L$.   Precisely,
correspondences between the labels of $\Gamma^L_{Q}$ and $\Upsilon^L_{[\ii_0]}$ are as follows: If $j_{t_1}=n^*$, we have
\[ \, [i,j] \mapsto \left\{ \begin{array}{ll}\left<i,-j-1\right> & \text{ if $i,j\neq n$ } \\
\left<i, -n-1\right> \ ( \text{resp.}\left<i, n+1\right>)  & \text{ if $j=n$.}  \end{array}\right.\]   If $j_{t_1}=(n+1)^*$ then we get the subquiver  $\Gamma^L_{[Q]}$ by substituting every $\alpha_{n+1}$ in  the subquiver $\Upsilon^L_{[\ii_0]}$ with $\alpha_{n}.$
\end{corollary}

\begin{proof}
It directly follows  from Lemma \ref{Lem:A longest}.
\end{proof}

\begin{remark}
In Corollary \ref{Cor:label 2}, the assumption $j_{t_1}=n^*$ can be replaced by {\it the twisted Coxeter element $\phi_{[\ii_0]}$ contains $n+1.$} On the other hand, the assumption  $j_{t_1}=(n+1)^*$ can be replaced by {\it the twisted Coxeter element $\phi_{[\ii_0]}$ contains $n.$}
\end{remark}

\begin{example}  \label{Ex:label 1+2}
Let $\ii_0= \prod_{k=0}^4 (2\ 1\ 3\ 5)^{k\vee}$. Then we can add labels to Example \ref{Ex:label come from 1GaQ} using Corollary
\ref{Cor:label come from 2GaQ} and Corollary \ref{Cor:label 2} as follows:
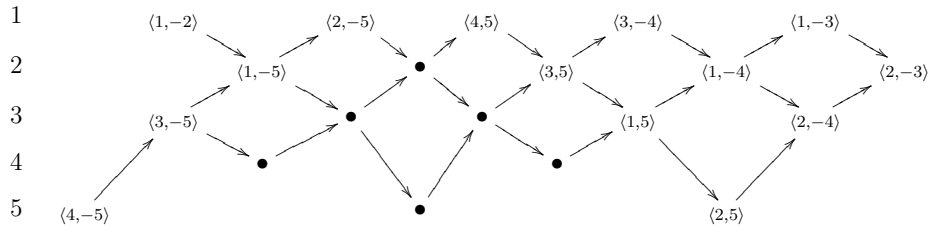
\[  \scalebox{0.8}{\xymatrix@C=2ex@R=1ex{
1&&\srt{1}{-2}\ar@{->}[dr]  && \srt{2}{-5}\ar@{->}[dr]  &&  \srt{4}{5}\ar@{->}[dr]  && \srt{3}{-4}\ar@{->}[dr] && \srt{1}{-3} \ar@{->}[dr]\\
2&&& \srt{1}{-5}\ar@{->}[dr]\ar@{->}[ur]  &&\bullet\ar@{->}[dr]\ar@{->}[ur]  && \srt{3}{5}\ar@{->}[dr]\ar@{->}[ur]  && \srt{1}{-4}\ar@{->}[dr]\ar@{->}[ur] && \srt{2}{-3} \\
3& &\srt{3}{-5}\ar@{->}[dr]\ar@{->}[ur]  &&\bullet\ar@{->}[ddr]\ar@{->}[ur]  &&\bullet\ar@{->}[dr]\ar@{->}[ur]  && \srt{1}{5}\ar@{->}[ddr]\ar@{->}[ur] && \srt{2}{-4}\ar@{->}[ur] \\
4&&& \bullet\ar@{->}[ur]  &&&& \bullet\ar@{->}[ur] \\
5&\srt{4}{-5}\ar@{->}[uur]  &&&& \bullet\ar@{->}[uur]  &&&& \srt{2}{5}\ar@{->}[uur]
}}\]
The left part of $\Upsilon_{[\ii_0]}$ can be obtained from $\Gamma_Q$ or $\Gamma_{Q^{{\rm rev}}}$ where
\[ Q = {\xymatrix@R=3ex{ \circ
\ar@{->}[r]_<{ \ 1} &  \circ
\ar@{<-}[r]_<{ \ 2}  &  \circ
\ar@{<-}[r]_<{ \ 3}
& \circ \ar@{-}[l]^<{\ \ \ \ \ \ 4} }}\]
and AR-quivers $\Gamma_Q$ and $\Gamma_{Q^{{\rm rev}}}$ are
 \[ \scalebox{0.8}
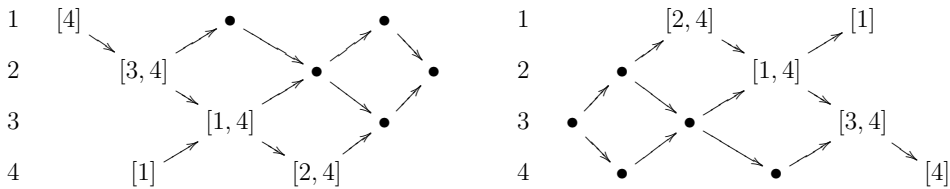
{\xymatrix@C=2ex@R=1ex{
1 &  [4]\ar@{->}[dr]  && \bullet \ar@{->}[dr] && \bullet\ar@{->}[dr]\\
2 && [3,4]\ar@{->}[dr]\ar@{->}[ur]  && \bullet\ar@{->}[dr]\ar@{->}[ur] &&\bullet \\
3 &  && [1,4]\ar@{->}[dr]\ar@{->}[ur] && \bullet\ar@{->}[ur] \\
4 && [1]  \ar@{->}[ur] && [2,4]\ar@{->}[ur]
}} \qquad   \scalebox{0.8}{\xymatrix@C=2ex@R=1ex{
1  &  && [2,4]\ar@{->}[dr] && [1]\\
2 && \bullet \ar@{->}[dr]\ar@{->}[ur]  && [1,4]\ar@{->}[dr]\ar@{->}[ur] &&\\
3  &\bullet \ar@{->}[dr] \ar@{->}[ur]&& \bullet\ar@{->}[dr]\ar@{->}[ur] && [3,4]\ar@{->}[dr]&&\\
4 &&\bullet \ar@{->}[ur] && \bullet \ar@{->}[ur]   && [4] &&
}}\]
 \end{example}

\begin{convention} \label{Conv:xi_0}
Now, we fix the height function $\xi$ in Algorithm \ref{Alg_twAR} such that
\begin{align} \label{eq: conv1}
\xi(\overline{1}) - 2 \times a^Q_1 + 1 =0.
\end{align}

Note that the $\rho$ (resp. $\rho'$) in \eqref{eq:jj1 reading} (resp. \eqref{eq:jj2 reading}) is the rightmost (resp. leftmost)
$S$-sectional (resp. $N$-sectional) path of length $n-1$ in $\Upsilon_{[\ii_0]}$. Furthermore, by \cite[Remark 1.14]{Oh14A}, \eqref{eq: conv1} implies that
\[ \text{
vertices in $\rho'$ are $(1,-1)$ and $\rho$ starts at $(1,1)$.}\]
Equivalently, we can say
\[ \text{ in $\widehat{\Upsilon}_{[\ii_0]}$,
vertices in $\rho'$ are $(\bar{r},-r)$ and vertices $\rho$ are $(\bar{r},r)$}.\]
\end{convention}

Now, with the above convention, we can summarize the previous results on the labeling $\Upsilon_{[\ii_0]}$ in this subsection as follows:

\begin{proposition} \label{Prop:cover by GQ}
For   a twisted adapted class $[\ii_0]$, we can label the vertices whose coordinates in $\widehat{\Upsilon}_{[\ii_0]}$
are $$\{ (\bar{i}, j),\, (\bar{i}, k)\, | \, 1\leq i\leq n, \, j\leq -i, \, k\geq i\}$$
by using the labels of $\Gamma_Q$.
In particular, the set of coordinates of
\begin{align} \label{eq: Phi n+1}
\Phi^+ |_{n+1\rangle}:=\{\left<i, n+1\right>,\ \left<i, -n-1\right>|\,  1\leq i\leq n\, \}
\end{align}
in $\widehat{\Upsilon}_{[\ii_0]}$ is
$\{ (\bar{i}, -i),\, (\bar{i}, i)\, | \, 1\leq i\leq n\}$.
 \end{proposition}

\begin{proof}
Note the following:
\begin{eqnarray} &&
\parbox{80ex}{
In $\Gamma_Q$, $\rho$ in Remark \ref{rem: jj_0} and $\rho'$ in Remark \ref{rem: jj_0 prime} correspond to $S$-sectional path of length $n-1$ whose second component is $n$, simultaneously.
}\label{eq: simul}
\end{eqnarray}
Then our assertion follows from Corollary \ref{Cor:label come from 1GaQ} and Corollary \ref{Cor:label come from 2GaQ}.
\end{proof}

 \begin{corollary} \label{Cor:4.25}
If the coordinate in $\widehat{\Upsilon}_{[\ii_0]}$ of $\left< j, n+1\right>$ is $(\bar{i}, i)$ then
the coordinate of  $\left< j, -n-1\right>$ is $(\overline{n+1-i}, -n-1+i).$
On the other hand, if the coordinate in $\widehat{\Upsilon}_{[\ii_0]}$ of $\left< j, n+1\right>$ is $(\bar{i}, -i)$
then the coordinate of  $\left< j, -n-1\right>$ is $(\overline{n+1-i}, n+1-i).$
 \end{corollary}

 \begin{proof}
 It follows from the fact that if the vertices with coordinates $(i,i)$, $1\leq i \leq n$, are determined by the $\Gamma_Q$ then
 the vertices with coordinates $(i,-i)$ $1\leq i \leq n$, are also obtained by the map (vertices in $i$-th residue)$\mapsto$
 (vertices in $n+1-i$-th residue) on $\Gamma_Q$.
 \end{proof}

\begin{definition} \label{def: folded multiplicity}
For any $\gamma \in \PR$,
{\it the folded multiplicity} of $\gamma$, denoted by $\overline{\mathsf{m}}(\gamma)$ is a positive integer defined as follows:
$$\overline{\mathsf{m}}(\gamma) = \max\left\{\,\left. \sum_{j \in \overline{i}} m_j\, \right| \, \overline{i} \in \overline{I} \text{ and } \sum_{i\in I} m_i\alpha_i=\ga\, \right\}.$$
If $\overline{\mathsf{m}}(\gamma)=1$, we say that $\gamma$ is {\it folded multiplicity free}.
\end{definition}

\begin{remark}
Note that every folded multiplicity free positive root is multiplicity free. However, $\ga=\langle a,n \rangle$ of $\Phi^+_{D_{n+1}}$ has
$\mathsf{m}(\gamma)=1$ and $\overline{\mathsf{m}}(\gamma)=2$.
\end{remark}

From now on, we shall label the vertices which are not covered by Proposition \ref{Prop:cover by GQ}. By the observation \eqref{eq: simul},
the number of such vertices is
\begin{align*}
\dfrac{n(n-1)}{2} &= n(n+1) -\dfrac{n(n+1)}{2}-n \\
& = |\Phi^+_{D_{n+1}}| - |\Phi^+_{A_{n}}| - \text{the number of vertices in $\rho$} \\
& = | \{ \ga \in \Phi^+_{D_{n+1}} \ | \ \overline{\mathsf{m}}(\gamma)=2  \} |.
\end{align*}
Furthermore, the vertices satisfy the following condition: Let us denote by $\Upsilon_{[\ii_0]}^C$ the full subquiver consisting of the uncovered vertices. Then
\begin{eqnarray} &&
\parbox{80ex}{
\begin{itemize}
\item each vertex in $\Upsilon_{[\ii_0]}^C$ is contained in two sectional paths of length $n-1$,
\item $\Upsilon_{[\ii_0]}^C$ is surrounded by $\rho$ in Remark \ref{rem: jj_0} and $\rho'$ in Remark \ref{rem: jj_0 prime}.
\end{itemize}
Thus, as the set of vertices, we have $$\Upsilon_{[\ii_0]} = \Upsilon^L_{[\ii_0]} \sqcup \Upsilon^C_{[\ii_0]} \sqcup \Upsilon^R_{[\ii_0]}\qquad\qquad\qquad.$$
}\label{eq: cond uncovered}
\end{eqnarray}

 \begin{proposition}  \label{Prop:label for Cvertices}
Let $\be_1$ and $\be_2$ be positive roots of $\Phi^+ |_{ n+1\rangle}$
such that whose twisted coordinates are $(i_1, -i_1)$ and $(i_2, i_2)$ for $1\leq i_1, i_2 \leq n-1$.
If $i_1+i_2\leq n$ then the label with the coordinate $(\overline{i _1+i_2}, -i_1+i_2)$ in $\widehat{\Upsilon}_{[\ii_0]}$ is $\be_1+\be_2.$
\end{proposition}

\begin{proof}
Suppose the vertex whose coordinate is $(\overline{i_1+i_2}, -i_1+i_2)$ in $\widehat{\Upsilon}_{[\ii_0]}$
has a residue less than $n$.
 In order to see the case when the third coordinate of $(\overline{i_1+i_2}, -i_1+i_2)$ is $1$,
let $V$ be the set of vertices which can be described as follows: $V=\cup_{k=0}^{i_1}V_t$, where
\begin{equation*}
\begin{aligned}
& V_0=  \{\, (i_2, i_2),(i_2-1, i_2-1), \cdots, (1,1)\, \}, \\
& V_t= \{ (i_2+t, i_2-t), (i_2+t-1, i_2-t-1), \cdots,  \, (t,-t) \}, \text{ for } t=1, \cdots, i_1.
\end{aligned}
\end{equation*}
Hence there is a compatible reading of $\Upsilon_{[\ii'_0]}$, equivalently the $Q$-adapted reduced word of $w_0$, which is the form
of $\ \jj \ J_1^{i_2} \ J_1^{i_2+1}\ J_1^{i_2+2} \ \cdots \ J_{i_1}^{i_1+i_2} \ \jj'\  $
where
\begin{itemize}
\item $\, J_i^j =j \, j-1\, j-2\, \cdots i+1\, i\ $,
\item $\jj$, $\jj'$ are some $Q$-adapted reduced words of Weyl group of type $D_{n+1}$,
\end{itemize}
by reading compatibly {\rm (i)} the vertices right to $V$ first, {\rm (ii)} the vertices in $V$ second and {\rm (iii)} the vertices left to $V$ last.
Denote \[ S_\jj:= s_{j_1}s_{j_2}\cdots s_{j_k} \text{ for } \jj=j_1\, j_2\, \cdots\, j_k \quad  \text{ and } \quad  S_{t_1}^{t_2}:= S_{J_{t_1}^{t_2}} \text{ for } 1\leq t_1<t_2 \leq n.\]
Then the label with coordinate $(i_2, i_2)$ is $S_{\jj}(\alpha_{i_2}).$ The label with coordinate $(i_1,-i_1)$ is
\[ S_{\jj} S_1^{i_2} S_1^{i_2+1} S_{2}^{i_2+2} \cdots S_{i_1-1}^{i_1+i_2-1} S_{i_1+1}^{i_1+i_2}(\alpha_{i_1})= S_{\jj}\big(\, \sum_{k=i_2+1}^{i_1+i_2}\alpha_k\, \big)
\in \Phi^+,\]
since one can easily compute that
$S_1^{i_2} S_1^{i_2+1} S_{2}^{i_2+2} \cdots S_{i_1-1}^{i_1+i_2-1} S_{i_1+1}^{i_1+i_2}(\alpha_{i_1})= \sum_{k=i_2+1}^{i_1+i_2}\alpha_k \in \Phi^+.$
Thus the label with coordinate $(i_1+i_2,-i_1+i_2)$ is
 \[ S_{\jj} S_1^{i_2} S_1^{i_2+1} S_{2}^{i_2+2} \cdots S_{i_1-1}^{i_1+i_2-1} (\alpha_{i_1+i_2})= S_{\jj}\big(\, \sum_{k=i_2}^{i_1+i_2}\alpha_k\, \big)\in \Phi^+.\]
 Since $ S_{\jj}(\alpha_{i_2})+   S_{\jj}\big(\, \sum_{k=i_2+1}^{i_1+i_2}\alpha_k\, \big) =S_{\jj}\big(\, \sum_{k=i_2}^{i_1+i_2}\alpha_k\, \big)$,
 we proved the assertion.

 Also, when the third coordinate of $(\overline{i_1+i_2}, -i_1+i_2)$ is $-1$, our assertion can be proved by the same argument.
 \end{proof}

By \eqref{eq: cond uncovered}, we have the following proposition.

\begin{proposition} For each vertex $v$ in $\Upsilon_{[\ii_0]}^C$, there exist unique $\al \in \rho$ and $\beta \in \rho'$ such that
\begin{itemize}
\item whose folded coordinates are $(\overline{i_1},i_1)$ and $(\overline{i_2},-i_2)$ respectively,
\item $(\overline{i _1+i_2}, -i_1+i_2)$ coincides with the folded coordinate of $v$.
\end{itemize}
\end{proposition}

\begin{proof}
By taking $\al$ as the intersection of $\rho$ and the $N$-sectional path containing $v$ and
$\be$ as the intersection of $\rho'$ and the $S$-sectional path containing $v$, our assertion follows.
\end{proof}

\begin{example}
Let $\ii_0= \prod_{k=0}^4 (2\ 1\ 3\ 5)^{k\vee}$. Then we can add labels to Example \ref{Ex:label 1+2} using Proposition \ref{Prop:label for Cvertices} as follows:
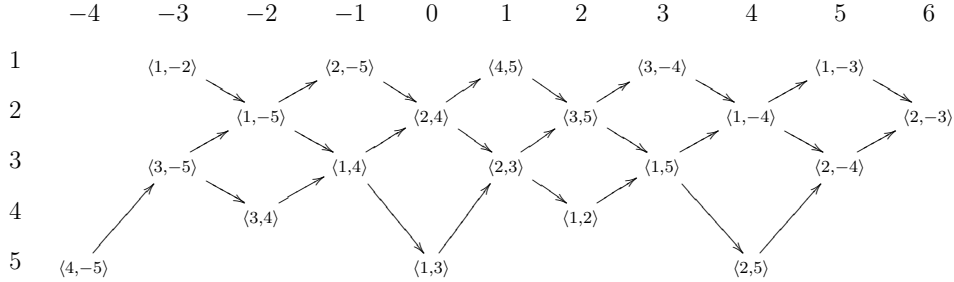
\[  \scalebox{0.8}{\xymatrix@C=2ex@R=1ex{
& -4 & -3& -2 & -1 & 0 & 1& 2& 3& 4& 5& 6& \\
1&&\srt{1}{-2}\ar@{->}[dr]  && \srt{2}{-5}\ar@{->}[dr]  &&  \srt{4}{5}\ar@{->}[dr]  && \srt{3}{-4}\ar@{->}[dr] && \srt{1}{-3} \ar@{->}[dr]\\
2&&& \srt{1}{-5}\ar@{->}[dr]\ar@{->}[ur]  &&\srt{2}{4}\ar@{->}[dr]\ar@{->}[ur]  && \srt{3}{5}\ar@{->}[dr]\ar@{->}[ur]  && \srt{1}{-4}\ar@{->}[dr]\ar@{->}[ur] && \srt{2}{-3} \\
3& &\srt{3}{-5}\ar@{->}[dr]\ar@{->}[ur]  &&\srt{1}{4}\ar@{->}[ddr]\ar@{->}[ur]  &&\srt{2}{3}\ar@{->}[dr]\ar@{->}[ur]  && \srt{1}{5}\ar@{->}[ddr]\ar@{->}[ur] && \srt{2}{-4}\ar@{->}[ur] \\
4&&& \srt{3}{4}\ar@{->}[ur]  &&&& \srt{1}{2}\ar@{->}[ur] \\
5&\srt{4}{-5}\ar@{->}[uur]  &&&& \srt{1}{3}\ar@{->}[uur]  &&&& \srt{2}{5}\ar@{->}[uur]
}}\]
We can see that $\left<2,3\right>=\left<2,-5\right>+\left<3,5\right>$, $\left<1,3\right>=\left<1,-5\right>+\left<3,5\right>$ and so on.
 \end{example}

Let us denote by
\[ \nu_{[\ii_0]}: {}_1\Gamma_Q\sqcup {}_2\Gamma_{Q^*}   \to  \Upsilon_{[\ii_0]},\]
the map between the sets of vertices.
Recall Remark \ref{rem: reflection w.r.t x-axis} for the isomorphism of quivers ${}_1\Gamma_Q\simeq{}_2\Gamma_{Q^*}\simeq\Gamma_Q$.

\begin{proposition}
For a twisted adapted class  $[\ii_0]$, we have
\[\nu_{[\ii'_0]}^{-1}(\left< i,j\right>)=\sum_{k=i}^{j-1} \alpha_k=[i,j-1]\]
for $1\leq i<j\leq n.$
\end{proposition}

\begin{proof}
Suppose the twisted Coxeter element of $[\ii_0]$ has $n$ and $\nu_{[\ii'_0]}^{-1}(\left< i, j\right>)\in {}_1\Gamma_Q.$ Note that
\begin{enumerate}
\item $\left< i,j\right>=\left<i, -n-1\right>+\left< j, n+1\right>$,
\item $\nu_{[\ii'_0]}^{-1}(\left< i, -n-1\right>)\in {}_1\Gamma_Q,$
\item  $\nu_{[\ii'_0]}^{-1}(\left< j, n+1\right>)\in {}_2\Gamma_{Q^*}.$
\end{enumerate}
If we denote the coordinate of $\left< i, -n-1\right>$ and $\left< j, n+1\right>$ in $\widehat{\Upsilon}_{[\ii_0]}$ by $( \overline{i_1}, i_1)$ and
$(\overline{i_2} ,-i_2)$ then the coordinate of $\left<i, j\right>$ and $\left<j, -n-1\right>$ are $(\overline{i_1+i_2}, i_1-i_2)$ and $(\overline{n+1-i_2}, n+1-i_2).$
Then, by taking proper height function of type $A_n$,  we have
\begin{enumerate}[(a)]
\item coordinates of $[i,n]$ and $[j,n]$ in $\Gamma_Q$ are $(i_1, i_1)$ and $(n+1-i_2, n+1-i_2)$,
\item the root $\alpha$ with the coordinate $(i_1+i_2, i_1-i_2)$ in $\Gamma_Q$ satisfies that $\alpha+[j,n]=[i,n],$ i.e., $\alpha=[i,j-1].$
\end{enumerate}
Here (b) follows by the additive property of AR-quiver. For the remaining cases. we can apply the similar arguments.
\end{proof}

By summing up results in this section, we get the following algorithm for labeling twisted and  folded AR-quiver of type $D_{n+1}$
by using the labels of $\Gamma_Q$ of type $A_n$.

\begin{algorithm} \label{Alg:label}
Let $Q$ be the Dynkin quiver of type $A_n$ such that $\mathfrak{p}^{D_{n+1}}_{A_n}([\ii_0])= [Q]$
\begin{enumerate}
\item[{\rm (i)}] Let us consider take ${}_1\Gamma_Q$ as $\Gamma_Q$ with the labeling $\Phi^+_{A_n}$ and
 take ${}_2\Gamma_{Q^*}$ by turning ${}_1\Gamma_Q$ upside down.
\item[{\rm (ii)}] Put ${}_2\Gamma_{Q^*}$ at the left of ${}_1\Gamma_Q$ in a canonical way.
\item[{\rm (iii)}] Take $\rho$  the unique $S$-sectional path of length $n-1$ in ${}_1\Gamma_Q$ and take $\rho'$  the unique $N$-sectional path of length $n-1$  in ${}_2\Gamma_{Q^*}$.
Then their labels are the same as $\{ [i,n] \ | \ 1 \le i \le n\}$
\item[{\rm (iv)}] Change all labels right to $\rho$ or left to $\rho'$ from $[a,b]$ to $\langle a,-b-1 \rangle$.
\item[{\rm (v)}] Change the duplicated labels $\{ [i,n] \ | \ 1 \le i \le n\}$ in $\rho \sqcup \rho'$ as follows:
\begin{align*}
  & {}_1\Gamma_Q  \ni [i,n] \longmapsto \begin{cases}
  \langle i, -n-1 \rangle   & \text{ if $n$ appears in $\phi_{[\ii_0]}$}, \\
  \langle i, n+1 \rangle    & \text{ if $n+1$ appears in $\phi_{[\ii_0]}$},
  \end{cases} \\
  & {}_2\Gamma_{Q^*} \ni [i,n] \longmapsto \begin{cases}
  \langle i, -n-1 \rangle   & \text{ if $n+1$ appears in $\phi_{[\ii_0]}$}, \\
  \langle i, n+1 \rangle    & \text{ if $n$ appears in $\phi_{[\ii_0]}$}.
  \end{cases} \\
\end{align*}
\item[{\rm (vi)}] Change all labels surrounded by $\rho$ and $\rho'$ from $[a,b]$ to $\langle a,b+1 \rangle$.
\end{enumerate}
The resulting quiver is the folded AR-quiver $\widehat{\Upsilon}_{[\ii_0]}.$
\end{algorithm}

\begin{remark} \label{rem:LR path}
From the above algorithm, one can notice that
\begin{itemize}
\item[{\rm (i)}] for any $\be \in \Upsilon^L_{[\ii_0]}$ and $\al \in \Upsilon^R_{[\ii_0]}$,
there exists a path from $\be$ to $\al$ inside $\Upsilon_{[\ii_0]}$,
\item[{\rm (ii)}] there exists a path from $\be$ to $\al$ in $\Gamma_Q$ if and only if there exists a path from $\be'$ to $\al'$ in  ${}_1\Gamma_Q$
(resp. ${}_2\Gamma_{Q^*}$), where $\be'$ and $\al'$ correspond to labels $\be$ and $\al$ obtained from the algorithm.
\end{itemize}
\end{remark}

From Algorithm \ref{Alg:label}, we have also an interesting phenomena as follows:
\begin{corollary} \label{lem: folded mul free iff}
Assume that we have $\ga \in \PR \setminus \Pi$. Then $\ga$ is folded multiplicity non-free if and only if $\ga$ is contained in $\Upsilon^C_{[\ii_0]}$.
\end{corollary}

\begin{corollary}  \label{Cor:label_dist}
Let  $(\bar{i},k)$ be the coordinate of $\left< a,b \right>$ in $\widehat{\Upsilon}_{[\ii_0]}$. Then the coordinate $(\bar{j}, k')$ of  $\left< a,-b \right>$ satisfies
\[ j=n+1-i \text{ and } |k-k'|=n+1.\]
\end{corollary}

\begin{proof}
Corollary \ref{Cor:4.25} shows it is true when $m=0$. Now by the fact that  $\widehat{\Upsilon}_{[\ii_0]}= {}_1 \Gamma_Q \sqcup  {}_2 \Gamma_{Q^*}$ and Corollary \ref{Cor:4.25},  we can see that the label of  $(\bar{i}, j)\in {}_1 \Gamma_Q $   is $\left<k_1, k_2\right>$ then  the label of  $(\overline{n+1-i}, n+1+j)\in {}_2 \Gamma_Q $   is $\left<k_1, -k_2\right>.$
\end{proof}

\begin{example} \label{Ex:twistedD_{n+1}}
Consider the twisted adapted class $[\ii_0]$ with the twisted Coxeter element $s_2 s_1 s_3 s_5$ of type $D_5.$ Then $\mathfrak{p}_{A_4}^{D_5}([\ii_0])=[Q]$, where
\[ Q = {\xymatrix@R=3ex{ \circ
\ar@{->}[r]_<{ \ 1} &  \circ
\ar@{<-}[r]_<{ \ 2}  &  \circ
\ar@{<-}[r]_<{ \ 3}
& \circ \ar@{-}[l]^<{\ \ \ \ \ \ 4} }}.\]

By {\rm (i)} in Algorithm \ref{Alg:label}, we can take ${}_2\Gamma_{Q^*}$ and ${}_1\Gamma_Q$ as follows:
$$
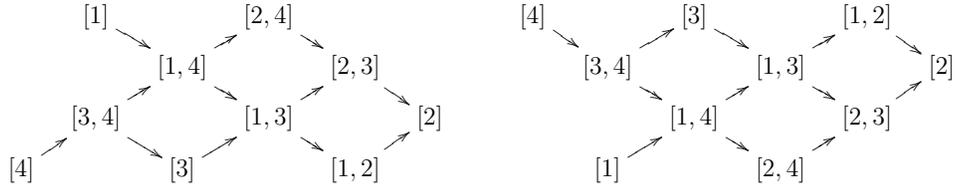
\scalebox{0.8}{\xymatrix@C=2ex@R=1ex{
& [1]\ar@{->}[dr]  && [2,4]\ar@{->}[dr]\\
&& [1,4]\ar@{->}[dr]\ar@{->}[ur]  && [2,3]\ar@{->}[dr]  \\
& [3,4]\ar@{->}[dr]\ar@{->}[ur] && [1,3]\ar@{->}[dr]\ar@{->}[ur] && [2] \\
[4] \ar@{->}[ur] && [3]\ar@{->}[ur] && [1,2] \ar@{->}[ur]
}}
\qquad
\scalebox{0.8}{\xymatrix@C=2ex@R=1ex{
[4]\ar@{->}[dr]  && [3]\ar@{->}[dr] && [1,2] \ar@{->}[dr]\\
& [3,4]\ar@{->}[dr]\ar@{->}[ur]  && [1,3]\ar@{->}[dr]\ar@{->}[ur] &&[2] \\
&& [1,4]\ar@{->}[dr]\ar@{->}[ur] && [2,3]\ar@{->}[ur] \\
& [1] \ar@{->}[ur] && [2,4]\ar@{->}[ur]
}}
$$
By {\rm (ii)}, {\rm (iii)} and {\rm (iv)} in Algorithm \ref{Alg:label}, we have
$$
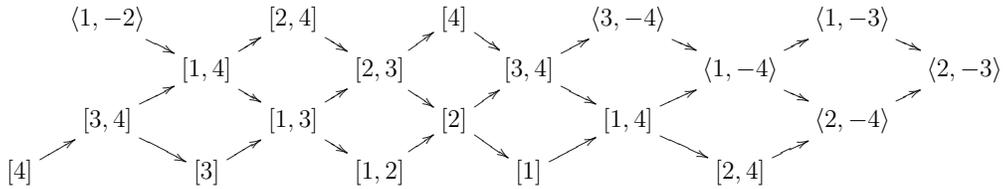
\scalebox{0.8}{\xymatrix@C=2ex@R=1ex{
& \langle 1,-2 \rangle\ar@{->}[dr]  && [2,4]\ar@{->}[dr]&& [4] \ar@{->}[dr]&& \langle 3,-4 \rangle \ar@{->}[dr]&& \langle 1,-3 \rangle \ar@{->}[dr]\\
&& [1,4]\ar@{->}[dr]\ar@{->}[ur]  && [2,3]\ar@{->}[dr]\ar@{->}[ur] && [3,4] \ar@{->}[ur]\ar@{->}[dr]&& \langle 1,-4 \rangle \ar@{->}[dr]\ar@{->}[ur]&& \langle 2,-3 \rangle \\
& [3,4]\ar@{->}[dr]\ar@{->}[ur] && [1,3]\ar@{->}[dr]\ar@{->}[ur] && [2] \ar@{->}[dr]\ar@{->}[ur] && [1,4] \ar@{->}[dr]\ar@{->}[ur] && \langle 2,-4 \rangle \ar@{->}[ur] \\
[4] \ar@{->}[ur] && [3]\ar@{->}[ur] && [1,2] \ar@{->}[ur] && [1] \ar@{->}[ur]  && [2,4] \ar@{->}[ur]
}}
$$

By {\rm (v)} in Algorithm \ref{Alg:label}, we have
$$
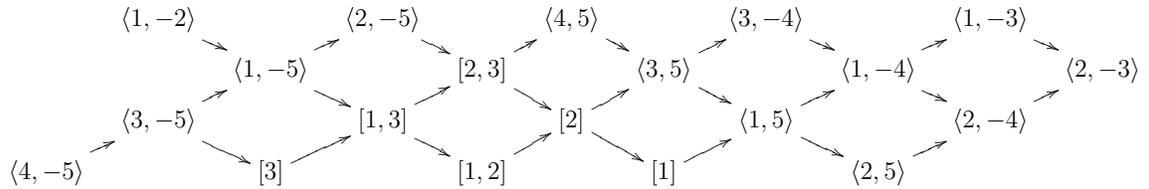
\scalebox{0.8}{\xymatrix@C=2ex@R=1ex{
& \langle 1,-2 \rangle\ar@{->}[dr]  && \langle 2,-5 \rangle \ar@{->}[dr]&& \langle 4,5 \rangle  \ar@{->}[dr]&& \langle 3,-4 \rangle \ar@{->}[dr]&& \langle 1,-3 \rangle \ar@{->}[dr]\\
&& \langle 1,-5 \rangle \ar@{->}[dr]\ar@{->}[ur]  && [2,3]\ar@{->}[dr]\ar@{->}[ur] && \langle 3,5 \rangle  \ar@{->}[ur]\ar@{->}[dr]&& \langle 1,-4 \rangle \ar@{->}[dr]\ar@{->}[ur]&& \langle 2,-3 \rangle \\
& \langle 3,-5 \rangle \ar@{->}[dr]\ar@{->}[ur] && [1,3]\ar@{->}[dr]\ar@{->}[ur] && [2] \ar@{->}[dr]\ar@{->}[ur] && \langle 1,5 \rangle  \ar@{->}[dr]\ar@{->}[ur] && \langle 2,-4 \rangle \ar@{->}[ur] \\
\langle 4,-5 \rangle  \ar@{->}[ur] && [3]\ar@{->}[ur] && [1,2] \ar@{->}[ur] && [1] \ar@{->}[ur]  && \langle 2,5 \rangle  \ar@{->}[ur]
}}
$$

By {\rm (vi)} in Algorithm \ref{Alg:label}, we have
$$
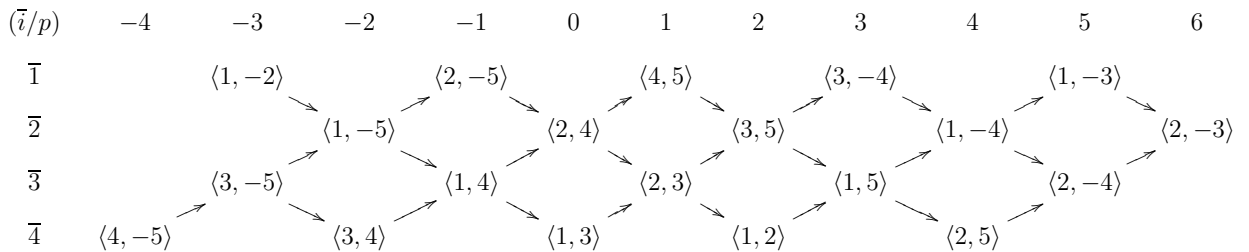
\scalebox{0.8}{\xymatrix@C=2ex@R=1ex{
(\overline{i}/p) & -4& -3& -2& -1& 0& 1& 2& 3& 4& 5& 6\\
\overline{1}&& \langle 1,-2 \rangle\ar@{->}[dr]  && \langle 2,-5 \rangle \ar@{->}[dr]&& \langle 4,5 \rangle  \ar@{->}[dr]&& \langle 3,-4 \rangle \ar@{->}[dr]&& \langle 1,-3 \rangle \ar@{->}[dr]\\
\overline{2}&&& \langle 1,-5 \rangle \ar@{->}[dr]\ar@{->}[ur]  && \langle 2,4 \rangle\ar@{->}[dr]\ar@{->}[ur] && \langle 3,5 \rangle  \ar@{->}[ur]\ar@{->}[dr]&& \langle 1,-4 \rangle \ar@{->}[dr]\ar@{->}[ur]&& \langle 2,-3 \rangle \\
\overline{3}&& \langle 3,-5 \rangle \ar@{->}[dr]\ar@{->}[ur] && \langle 1,4 \rangle\ar@{->}[dr]\ar@{->}[ur] && \langle 2,3 \rangle \ar@{->}[dr]\ar@{->}[ur] && \langle 1,5 \rangle  \ar@{->}[dr]\ar@{->}[ur] && \langle 2,-4 \rangle \ar@{->}[ur] \\
\overline{4}&\langle 4,-5 \rangle  \ar@{->}[ur] && \langle 3,4 \rangle\ar@{->}[ur] && \langle 1,3 \rangle \ar@{->}[ur] && \langle 1,2 \rangle \ar@{->}[ur]  && \langle 2,5 \rangle  \ar@{->}[ur]
}}
$$

Finally, we get the twisted AR-quiver:
$$
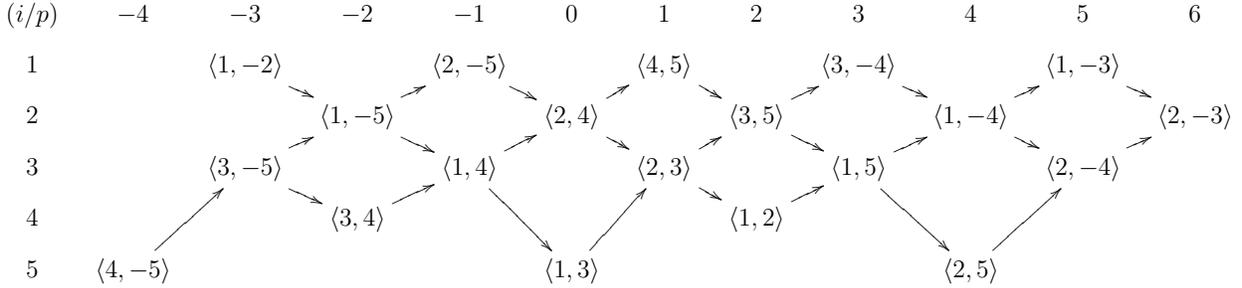
\scalebox{0.8}{\xymatrix@C=2ex@R=1ex{
(i/p) & -4& -3& -2& -1& 0& 1& 2& 3& 4& 5& 6\\
1&& \langle 1,-2 \rangle\ar@{->}[dr]  && \langle 2,-5 \rangle \ar@{->}[dr]&& \langle 4,5 \rangle  \ar@{->}[dr]&& \langle 3,-4 \rangle \ar@{->}[dr]&& \langle 1,-3 \rangle \ar@{->}[dr]\\
2&&& \langle 1,-5 \rangle \ar@{->}[dr]\ar@{->}[ur]  && \langle 2,4 \rangle\ar@{->}[dr]\ar@{->}[ur] && \langle 3,5 \rangle  \ar@{->}[ur]\ar@{->}[dr]&& \langle 1,-4 \rangle \ar@{->}[dr]\ar@{->}[ur]&& \langle 2,-3 \rangle \\
3&& \langle 3,-5 \rangle \ar@{->}[dr]\ar@{->}[ur] && \langle 1,4 \rangle\ar@{->}[ddr]\ar@{->}[ur] && \langle 2,3 \rangle \ar@{->}[dr]\ar@{->}[ur] && \langle 1,5 \rangle  \ar@{->}[ddr]\ar@{->}[ur] && \langle 2,-4 \rangle \ar@{->}[ur] \\
4& && \langle 3,4 \rangle\ar@{->}[ur] &&&& \langle 1,2 \rangle \ar@{->}[ur]  \\
5& \langle 4,-5 \rangle  \ar@{->}[uur] &&&& \langle 1,3 \rangle  \ar@{->}[uur] &&&& \langle 2,5 \rangle  \ar@{->}[uur]
}}
$$
\end{example}

\begin{remark}
We have another algorithm to get labels of twisted AR-quivers without using the labels of $\Gamma_Q$ of type $A_n$.
See Algorithm \ref{Alg:labels_comb}.
\end{remark}

\vskip 3mm

\section{Labeling of a twisted AR-quiver using only its shape}\label{Sec:twisted_AR_D_label_shape}
 As in the last section, we fix the height function of a folded AR-quiver
$\widehat{\Upsilon}_{[\ii_0]}$ by letting $\left< i, \pm n+1 \right>$  have the coordinate $(\bar{j}, \pm j).$ Note that we can naturally define the notion of the
sectional path for twisted AR-quivers and folded AR-quivers of type $D_n$. Now we shall extend the notion of the
sectional path for our purpose:

\begin{definition} \label{def: N_m S_m ext} Let us fix $m\in \Z$.
\begin{enumerate}
\item An {\it $N^{{\rm ext}}_m$-sectional quiver}  is the subquiver of $\widehat{\Upsilon}_{[\ii_0]}$ consisting of the following set of vertices
$ \{\, (\overline{i}, -i+2m+2(n+1)m') \, |\,  0 \le i\le n, \, m'\in \Z  \}\cap \widehat{\Upsilon}_{[\ii_0]}.$
\item An {\it $S^{{\rm ext}}_m$-sectional quiver}  is the subquiver of $\widehat{\Upsilon}_{[\ii_0]}$ consisting of the following set of vertices
$ \{\, (\overline{i}, i+2m+2(n+1)m') \, |\, 0 \le i\le n, \, m'\in \Z \}\cap \widehat{\Upsilon}_{[\ii_0]}.$
\item An {\it $N_m$-sectional path} is the subquiver of $\widehat{\Upsilon}_{[\ii_0]}$ consisting of the following set of vertices
$ \{\, (\overline{i}, -i+2m) \, |\, 0 \le i\le n  \}\cap \widehat{\Upsilon}_{[\ii_0]}.$
\item An {\it $S_m$-sectional path} is the subquiver of $\widehat{\Upsilon}_{[\ii_0]}$ consisting of the following set of vertices
$\{\, (\overline{i}, i+2m-(2n+2)) \, |\, 0 \le i\le n  \}\cap \widehat{\Upsilon}_{[\ii_0]}.$
\item An {\it $m$-swing} is the union of the $N_m$-sectional path and the $S_m$-sectional path.
\item An {\it $m$-snake} is the union of the $N_m^{{\rm ext}}$-sectional path and the $S_m^{{\rm ext}}$-sectional path associated to the same $m$.
\end{enumerate}
\end{definition}

\begin{example} In Example \ref{Ex:twistedD_{n+1}}, we can observe properties of the notions in Definition \ref{def: N_m S_m ext}:
\begin{enumerate}
\item[{\rm (a)}] $N^{{\rm ext}}_{-1}$-sectional quiver is union of $N_{-1}$-sectional path and $N_{4}$-sectional path which are disconnected.
Here, $N_{-1}$-sectional path consists of $\{ \langle 1, -2 \rangle \}$ and $N_{4}$-sectional path consists of
$\{ \langle 2,-3 \rangle, \langle 2, -4 \rangle, \langle 2,5 \rangle  \}$.
On the other hand, the $N^{{\rm ext}}_{0}$-sectional quiver is connected and coincides with $\rho'$ consisting of
$\{ \langle 2,-5 \rangle, \langle 1,-5 \rangle, \langle 3,-5 \rangle, \langle 4,-5 \rangle \}$.
\item[{\rm (b)}] $S^{{\rm ext}}_{2}$-sectional quiver is union of $S_{2}$-sectional path and $S_{-3}$-sectional path which are disconnected.
Here, $S_{2}$-sectional path consists of $\{ \langle 1, -3 \rangle,\langle 2, -3 \rangle \}$ and
$S_{-3}$-sectional path consists of $\{ \langle 3,-5 \rangle, \langle 3, 4 \rangle \}$.
On the other hand, the $S^{{\rm ext}}_{-1}$-sectional quiver is connected and coincides with $S_{-1}$-sectional path consisting of
$\{ \langle 2,-5 \rangle, \langle 2,4 \rangle, \langle 2,3 \rangle, \langle 1,2 \rangle \}$.
\end{enumerate}
\end{example}

The sectional paths have the following propositions.
\begin{proposition} \label{Prop: properties of ext} \hfill
\begin{enumerate}
\item[{\rm (1)}] Each $N^{{\rm ext}}_m$-sectional $($resp. $S^{{\rm ext}}_m$-sectional$)$ quiver consists of $n$ vertices with distinct residues
$\bar{1}, \bar{2}, \cdots, \bar{n}$.
\item[{\rm (2)}] If $m_1\equiv m_2$ $($mod $n+1)$ then $N_{m_1}^{{\rm ext}}$-sectional quiver
$($resp. $S_{m_1}^{{\rm ext}}$-sectional quiver$)$ coincides with $N_{m_2}^{{\rm ext}}$-sectional quiver $($resp. $S_{m_2}^{{\rm ext}}$-sectional quiver$)$.
\item[{\rm (3)}]  If $m_1\not \equiv m_2$ $($mod $n+1)$ the $N_{m_1}^{{\rm ext}}$-sectional quiver and the $S_{m_2}^{{\rm ext}}$-sectional quiver have a
unique intersection.
\item[{\rm (4)}] For each $m$, one of the $N^{{\rm ext}}_m$-sectional quiver and $S^{{\rm ext}}_m$-sectional quiver is connected.
\item[{\rm (5)}] If $\left< k_1, k_2\right>$ is in the $N_m$-sectional $($resp. $S_m$-sectional$)$ path
then $\left< k_1, -k_2\right>$ is in the $S_m^{{\rm ext}}$-sectional $($resp. $N_m^{{\rm ext}}$-sectional$)$ quiver.
\end{enumerate}
\end{proposition}

\begin{proof}
(1), (2) and (3) are not hard to prove so that we only give a proof of (4) and (5).
By the observations in the previous section, we know that $\left< 1, n+1 \right>$ and
$\left< 1, -n-1\right>$ have coordinates $(\bar{i},\pm i)$ and $(\bar{j}, \pm j)$ with $i+j=n+1$. Recall the facts in \cite[Corollary 1.15]{Oh14A}:
\begin{enumerate}[(a)]
\item In an AR-quiver $\Gamma_Q$ of type $A_n$, $[1, n]$ is located at the intersection of
the $S$-sectional path of length $n-1$ and the $N$-sectional path of length $n-1$.
\item In an AR-quiver $\Gamma_Q$ of type $A_n$, there are exactly two sectional path of length $n-1$.
\end{enumerate}

Since folded AR-quiver $\widehat{\Upsilon}_{[\ii_0]}$ of $D_{n+1}$ can be understood as the disjoint union of AR-quivers ${}_1\Gamma_Q$
and ${}_2\Gamma_{Q^*}$ of type $A_n$,  we conclude that every $N_m^{{\rm ext}}$-sectional path associated to $m=0,1, \cdots, i$ and
every $S_m^{{\rm ext}}$-sectional path associated to $m=0,1,\cdots, j$ is connected.
Then our fourth assertion follows from the fact that $i+j=n+1$.

 For (5), recall that  $(\bar{i}, j)$ and $(\overline{n+1-i}, n+1+j)$ are in the  $N^{{\rm ext}}_m$-sectional (resp. $S^{{\rm ext}}_m$-sectional) quiver and $S^{{\rm ext}}_m$-sectional (resp. $N^{{\rm ext}}_m$-sectional) quiver, respectively, for $m= (i+j)/2$ (resp. $m=(j-i)/2$). By Corollary \ref{Cor:label_dist}, we proved the corollary.
\end{proof}

\begin{proposition}\label{Prop:m-snake shares p} For each $m \in \Z$, the $m$-snake has a positive integer $p$ such that
every vertex in $m$-snake has a summand $p$ or $-p$.
 \end{proposition}

\begin{proof}
Let us take a connected $N_m^{{\rm ext}}$-sectional quiver associated $m\not \equiv 0 \ (\text{mod } n+1)$. Then the $N_m^{{\rm ext}}$-sectional quiver
has an intersection with the $S_0^{{\rm ext}}$-sectional.
Suppose $(\bar{i}, i)$ is the vertex with the label $\left< j, \pm(n+1)\right>.$ Recall the two following facts:
\begin{enumerate}[(a)]
\item vertices in an $N$-sectional path of $A_n$ share the first component (Theorem \ref{thm: labeling GammaQ}).
Thus $(\bar{i'},2i-i')$ in ${}_1\Gamma_Q\subset \widehat{\Upsilon}_{[\ii_0]}$ for $i'<i$ share $j$ as the first summand.
\item Proposition \ref{Prop:label for Cvertices} shows that $(\bar{i'},2i-i')$ for $i'>i$ has $j$ as its summand.
\end{enumerate}
Hence we proved for the $N$-sectional path $N^1$ contained in
$N_m^{{\rm ext}}$-sectional quiver and has an intersection with $S_0^{{\rm ext}}$-sectional quiver.
Thus if $N_m^{{\rm ext}}$-sectional quiver is connected, then all vertices in $N_m^{{\rm ext}}$ share $j$ as their first summand.

Now we assume that $N_m^{{\rm ext}}$-sectional quiver is not connected.
We claim that
\begin{eqnarray*} &&
\parbox{85ex}{
$N_m^{{\rm ext}}$-sectional quiver is not connected (equivalently the length of $N^1$ is $k$ less than $n-1$),
the connected component $N^1$ is contained in ${}_1\Gamma_Q$.
}
\end{eqnarray*}
If $N_1$ is not contained in
${}_1\Gamma_Q$, then it contains a vertex in $\Upsilon_{[\ii_0]}^C$ which yields a contradiction to the first observation in \eqref{eq: cond uncovered}.
By Theorem \ref{thm: labeling GammaQ} and Algorithm \ref{Alg:label}, the summand $j$ is the same as $n-k$, in this case. By (1) in Proposition \ref{Prop: properties of ext},
the other $N$-sectional path $N^2$ of $N_m^{{\rm ext}}$-sectional quiver is of length $n-2-k$. By Theorem \ref{thm: labeling GammaQ},
Remark \ref{rem: reflection w.r.t x-axis} and Algorithm \ref{Alg:label}, every vertex in $N^2$ has $-(n+1-k)$
as it second summand and $N^2$ is contained in ${}_2\Gamma_{Q^*}$. Hence we proved our assertion for $N_m^{{\rm ext}}$-sectional quiver.
As a summary,
\begin{enumerate}
\item[{\rm (i)}] If $N_m^{{\rm ext}}$-sectional quiver is connected, every vertex in it has $p$ as its summand,
\item[{\rm (ii)}] If $N_m^{{\rm ext}}$-sectional quiver is not connected, $N_m^{{\rm ext}}$-sectional quiver is decomposed into two connected $N$-sectional paths,
$N^1_m$ whose vertices share $p$ as their first summand and $N^2_m$ whose vertices share $-p$ as their second summand.
\end{enumerate}
The assertion restricted to $S_m^{{\rm ext}}$-sectional quiver follows from (5) in Proposition \ref{Prop: properties of ext}.
More precisely, for $m \ne 0$,
\begin{enumerate}
\item[{\rm (i)}] If $S_m^{{\rm ext}}$-sectional quiver is connected, then $N_m^{{\rm ext}}$-sectional quiver is not connected and
every vertex in it has $p$ as its summand,
\item[{\rm (ii)}] If $S_m^{{\rm ext}}$-sectional quiver is not connected, then $N_m^{{\rm ext}}$-sectional quiver is connected
and $S_m^{{\rm ext}}$-sectional quiver is decomposed into two connected $N$-sectional paths,
$S^1_m$ whose vertices share $-p$ as their second summand and $S^2_m$ whose vertices share $p$ as their first summand.
\end{enumerate}
\end{proof}

From the above proposition, we have the following corollaries:

\begin{corollary} \label{Cor: induced property ext}\hfill
\begin{enumerate}
\item[{\rm (1)}]
 If $N_m^{{\rm ext}}$-sectional quiver is not connected, it consists of an $N$-sectional subquiver $N^1_m$ contained in ${}_1\Gamma_Q$,
and an $N$-sectional subquiver $N^2_m$ contained in ${}_2\Gamma_{Q^*}$.   Furthermore,
 \begin{enumerate}
\item[{\rm (N-i)}] every vertex in $N^1_m$ has $n-k$ as its first summand and every vertex in $N^2_m$ has $-(n-k)$ as its second summand, when the length of $N^1_m$ is
of length $k < n-1$,
\item[{\rm (N-ii)}] every vertex in $N_m^{{\rm ext}}$ is folded multiplicity free.
\end{enumerate}
\item[{\rm (2)}]  If $S_m^{{\rm ext}}$-sectional quiver is not connected, it consists of an $S$-sectional subquiver $S^1_m$ contained in ${}_1\Gamma_Q$,
and an $S$-sectional subquiver $S^2_m$ contained in ${}_2\Gamma_{Q^*}$. Furthermore,
 \begin{enumerate}
\item[{\rm (S-i)}] every vertex in $S^1_m$ has $-(k+2)$ as its second summand and
every vertex in $S^2_m$ has $k+2$ as its first summand, when the length of $S^1_m$ is of length $k < n-1$.
\item[{\rm (S-ii)}] every vertex in $S_m^{{\rm ext}}$ is folded multiplicity free.
\end{enumerate}
\end{enumerate}
\end{corollary}

\begin{corollary} Recall that $a_1^Q$ is the number of arrows in $Q$ directed toward 1.
Only when $m=0$ or $m =\mathsf{a} \seteq a_1^Q+1$, both $N_m^{{\rm ext}}$-sectional quiver and $S_m^{{\rm ext}}$-sectional quiver are connected.
Moreover,
\begin{enumerate}
\item[{\rm (1)}] every vertex in $N_0^{{\rm ext}}$-sectional quiver or $S_0^{{\rm ext}}$-sectional quiver is folded multiplicity free and has $n+1$ or $-n-1$
as its second summand.
\item[{\rm (2)}] every vertex in $N_{\mathsf{a}}^{{\rm ext}}$-sectional quiver and $S_{\mathsf{a}}^{{\rm ext}}$-sectional quiver has $1$ as its first summand.
\item[{\rm (3)}] when $m \ne 0$ or $\mathsf{a}$, one of $N_m^{{\rm ext}}$-sectional quiver and $S_m^{{\rm ext}}$-sectional quiver is connected and the other is not.
\end{enumerate}

\end{corollary}

\begin{proof}
When $m=0$ (resp. $m=a^Q_i$), $N_m^{{\rm ext}}$-sectional quiver coincides with the $N$-sectional path of length $n-1$ in ${}_1\Gamma_Q$ (resp. ${}_2\Gamma_{Q^*}$)
and $S_m^{{\rm ext}}$-sectional quiver coincides with the $S$-sectional path of length $n-1$ in ${}_2\Gamma_{Q^*}$ (resp. ${}_1\Gamma_Q$).

(1) Since every vertex in the $N$-sectional path of length $n-1$ in $\Gamma_Q$ is of the form $[a,n]$, the assertion (1) follows from
Remark \ref{rem: reflection w.r.t x-axis} and {\rm (iii)} in Algorithm \ref{Alg:label}.

(2) Since every vertex in the $S$-sectional path of length $n-1$ in $\Gamma_Q$ is of the form $[1,b]$, the assertion (2) hold by the same reason of $(1)$.

(3) This assertion follows from Theorem \ref{thm: labeling GammaQ}.
\end{proof}

Now we can rename the notions in Definition \ref{def: N_m S_m ext} by considering the above results:

\begin{definition} \label{def: rename} \hfill
\begin{enumerate}
\item[{\rm (1)}] An $m$-snake is renamed as a $[p]$-snake if every vertex of the snake has $p$ or $-p$ as a summand.
\item[{\rm (2)}] An $N_m^{{\rm ext}}$-sectional (resp. $S_m^{{\rm ext}}$-sectional) quiver is renamed as the $N^{{\rm ext}}[\pm p]$-sectional
(resp. $S^{{\rm ext}}[\pm p]$-sectional) quiver if every vertex of it has $p$ or $-p$ as a summand.
\item[{\rm (3)}] An $X_m$-sectional path is renamed as the $X[p]$-sectional (resp. $X[-p]$-sectional) path
if every vertex of it has $p$ (resp. $-p$) as a summand. Here $X=N$ or $S$.
\item[{\rm (4)}] An $m$-swing is renamed as the $[p]$-swing (resp. $[-p]$-swing) if it
is the union of the $N[p]$-sectional path (resp. $N[-p]$-sectional path) and the $S[p]$-sectional path $($resp. $S[-p]$-sectional path$)$
\end{enumerate}
\end{definition}

We sometimes call the $N^{{\rm ext}}[\pm p]$-sectional quiver or $S^{{\rm ext}}[\pm p]$-sectional quiver by an extended $p$-sectional quiver.

\begin{theorem}\label{thm:p-snake well}
For the folded AR-quiver $\widehat{\Upsilon}_{[\ii_0]}$, a root $\alpha\in \Phi^+$ has $p$ or $-p$ as a summand if and only if $\alpha$ is in the $[p]$-snake.
\end{theorem}

\begin{proof}
Note that the number of roots with $p$ or $-p$ as a summand is the same as $2n$ which is also the number of vertices in the $m$-snake for any $m \in \Z$.
Thus our assertion follows from Proposition \ref{Prop:m-snake shares p}.
\end{proof}

\begin{corollary} \
\begin{enumerate}
\item[{\rm (1)}] For $p\in \{1, 2, \cdots, n+1\}$,
a root $\alpha\in \Phi^+$ has $p$ $($resp. $-p)$ as its summand if and only if $\alpha$ is in the $[p]$-swing $($resp. $[-p]$-swing$)$.
\item[{\rm (2)}] The root $\left< a, b\right>$ is the intersection of the $[a]$-swing and the $[b]$-swing.
\end{enumerate}
\end{corollary}

\begin{proof}
It is an immediate consequence from the fact that $[p]$-snake is the union of $[p]$-swing and $[-p]$-swing.
\end{proof}

\begin{corollary} \label{cor: nbar adjacent}
Let $\widehat{\Upsilon}_{[\ii_0]}$ be the folded AR-quiver and let $p\in I\setminus\{n+1\}$.
\begin{enumerate}
\item[{\rm (1)}] For any $(\bar{n}, k)$ and $(\bar{n}, k+2)$ in $\widehat{\Upsilon}_{[\ii_0]}$, these two vertices are in the same swing.
\item[{\rm (2)}] If a sectional path has a vertex in the $\bar{n}$-th residue then it is not the $N[-p]$ nor the $S[-p]$-sectional path,
i.e. it is either $N[p]$, $S[p]$, $N[\pm(n+1)] $ or $S[\pm(n+1)]$-sectional path.
\end{enumerate}
\end{corollary}

\begin{proof}
The first assertion follows from the computation that $(\bar{n}, k)$ and $(\bar{n}, k+2)$
are contained in same $m$-snake for some $m$. Note that all vertices with $\bar{n}$-th residue
are contained in $\Upsilon^C_{[\ii_0]}$, $\rho$ or $\rho'$. Since  $\rho$ (resp. $\rho'$) coincides with $S^{{\rm ext}}_{0}$-sectional quiver
(resp. $N^{{\rm ext}}_0$-sectional), our second assertion follows from Algorithm \ref{Alg:label}.
\end{proof}

\begin{example}
In the twisted AR-quiver $\widehat{\Upsilon}_{[\ii_0]}$ of $[\ii_0]$ in Example \ref{Ex:twistedD_{n+1}}, the $N[3]$-sectional path and $S[3]$-sectional path  are
  \[
\scalebox{0.8}{\xymatrix@C=0.9ex@R=1ex{
(\bar{i}\backslash p) & -2 & -1 & 0 & 1 \\
\bar{1} &&  && \srt{3}{-4}\\
\bar{2} & & &\srt{3}{5} \ar@{->}[ur]  \\
\bar{3} && \srt{2}{3} \ar@{->}[ur]\\
\bar{4} & \srt{1}{3}\ar@{->}[ur]\\
 }}
\ \ \ \ \text{ and } \ \ \ \  \scalebox{0.8}{\xymatrix@C=0.9ex@R=1ex{
(\bar{i}\backslash p) & -5 & -4 & \cdots & 3& 4 \\
\bar{1}&&&\cdots & \srt{1}{-3}\ar@{->}[dr]\\
\bar{2} & & & \cdots && \srt{2}{-3}\\
\bar{3} & \srt{3}{-5} \ar@{->}[dr]&& \cdots && \\
\bar{4} && \srt{3}{4}& \cdots & \\
 }}
 \]
We can see that the $S[3]$-sectional path is  $\left< 3, -5\right>\to \left< 3, 4\right>$,
the $S[-3]$-sectional path is $\left<1,-3\right>\to \left< 2, -3 \right>$, the $N[3]$-sectional path
is $\left< 1,3\right> \to \left< 2,3\right> \to \left< 3,5\right> \to \left< 3,-4\right> $ and
$N[-3]$-sectional path is $\emptyset$. The following picture is the $3$-snake and $3$-swing. Also, $[-3]$-swing is same as $S[-3]$-sectional path.
 \[
\scalebox{0.7}{\xymatrix@C=0.5ex@R=1ex{
(\bar{i}\backslash p) & -5 & -4 & -3 & -2 & -1 & 0 & 1& 2& 3& 4 \\
 & & & &  && & &  \cdot \ar@{-}[dr]& &  \\
\bar{1}&  && && && \srt{3}{-4}\ar@{-}[ur] && \srt{1}{-3}\ar@{->}[dr]\\
\bar{2} & & && & &\srt{3}{5}\ar@{->}[ur]  &&  && \srt{2}{-3}\\
\bar{3} & \srt{3}{-5}\ar@{->}[dr]&& && \srt{2}{3} \ar@{->}[ur]\ar@{->}[ur]&&&&  \\
\bar{4} && \srt{3}{4}\ar@{-}[dr]&& \srt{1}{3}\ar@{->}[ur]& & &&&& \\
 & & & \cdot \ar@{-}[ur]& \\
 }}
 \ \ \ \ \scalebox{0.7}{\xymatrix@C=0.5ex@R=1ex{
(\bar{i}\backslash p) & -5 & -4 & -3 & -2 & -1 & 0 & 1 \\
 & & & &  && & &   &  \\
\bar{1}&  && && && \srt{3}{-4} \\
\bar{2} & & && & &\srt{3}{5}\ar@{->}[ur]\\
\bar{3} & \srt{3}{-5}\ar@{->}[dr]&& && \srt{2}{3} \ar@{->}[ur]\ar@{->}[ur]&\\
\bar{4} && \srt{3}{4}\ar@{-}[dr]&& \srt{1}{3}\ar@{->}[ur]& & &\\
 & & & \cdot \ar@{-}[ur]& \\
 }}
 \]
\end{example}

Now we can characterize the $[p]$-snake for $p = 1,2,\ldots,n+1$ as follows:
\begin{eqnarray} &&
\parbox{85ex}{
Let $\{X,Y\}=\{N,S\}$.
\begin{itemize}
\item {\rm ($p \ne n+1$ and $1$)} $[p]$-snake consists of one $X$-sectional path of length $n-1$ and two sectional $Y$-paths $Y^1$ of length $k$ contained in ${}_1\Gamma_Q$
and $Y^2$ of length $l$ contained in ${}_2\Gamma_{Q^*}$ such that $k+l=n-2$. If $Y=S$ then $k= p-2$, every vertex in $S^1$ has $-p$ as its second
summand and the remained vertices in $[p]$-snake have $p$ as their common summand. If $Y=N$, then $k=p-2$, every vertex in $N^2$ has $-p$ as its second
summand and the remained vertices in $[p]$-snake have $p$ as their common summand.
\item {\rm ($p = n+1$)} $[n+1]$-snake consists of an $N$-sectional path ($= N_0$-sectional path) and an $S$-sectional path ($= S_{n+1}$-sectional path)
of length $n-1$ such that the $N$-sectional path is located at the right of the $S$-sectional path.
\item {\rm ($p = 1$)} $[n+1]$-snake consists of an $N$-sectional path and an $S$-sectional path of length $n-1$ such that the
$S$-sectional path is located at the right of the $N$-sectional path.
\end{itemize}
}\label{eq: [p]-cahracter}
\end{eqnarray}

Using the properties of sectional quivers (paths),
we can find labels of the twisted AR-quiver $\Upsilon_{[\ii_0]}$ by using only its shape.

\begin{algorithm} \label{Alg:labels_comb}  \hfill

\noindent
{\rm (Step 1)} By using the height function and \eqref{eq: Foldable Coxeter D}, find the shape of $\widehat{\Upsilon}_{[\ii_0]}$  \\
{\rm (Step 2)} Using the characterizations in \eqref{eq: [p]-cahracter},
find all $[p]$-snake for $p=1,2,\ldots,n$ and put their summand. In this step, we can complete the labels except those in $[n+1]$-snake. \\
{\rm (Step 3)} If $n$ $($resp. $n+1)$ appears in $\phi_{[\ii_0]}$, we complete labels for vertices in $S^{{\rm ext}}_0$-sectional path by putting $-n-1$ $($resp. $n+1)$  and
labels for vertices in $N^{{\rm ext}}_0$-sectional path by putting $n+1$ $($resp. $-n-1)$.
\end{algorithm}

\begin{example}
Suppose $[\ii_0]$ has the twisted Coxeter element $(s_2 s_5 s_1 s_3)\vee.$  Here we denote by $\{a,b\}$ the root $\alpha$ with summands $a$ and $b$.\\
By {\rm (Step 1)} and {\rm (Step 2)} for $p \ne n+1$
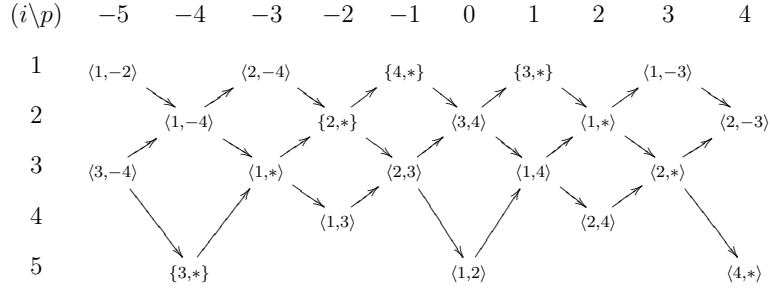
\[\scalebox{0.8}{\xymatrix@C=0.9ex@R=1ex{
( i\backslash p) & -5 & -4 & -3 & -2 & -1 & 0 & 1& 2& 3& 4 \\
1&_{ \langle 1, -2 \rangle }  \ar@{->}[dr]  && _{  \langle 2,-4 \rangle  }  \ar@{->}[dr]  && _{\{ 4,* \} } \ar@{->}[dr]  && _{\{  3,*\} }
 \ar@{->}[dr] &&  _{ \langle 1, -3 \rangle }  \ar@{->}[dr]\\
2&& _{ \langle 1, -4 \rangle } \ar@{->}[dr]\ar@{->}[ur]  && _{\{ 2,*\}} \ar@{->}[dr]\ar@{->}[ur]  &&  _{\langle 3,4 \rangle }  \ar@{->}[dr]\ar@{->}[ur]  && _{\langle 1, * \rangle}
\ar@{->}[dr]\ar@{->}[ur] && _{ \langle 2,-3 \rangle }  \\
3& _{\langle 3,-4 \rangle } \ar@{->}[ddr]\ar@{->}[ur]  && _{\langle 1, * \rangle } \ar@{->}[dr]\ar@{->}[ur]  && _{ \langle 2,3 \rangle }  \ar@{->}[ddr]\ar@{->}[ur]
&& _{ \langle 1, 4\rangle  }  \ar@{->}[dr]\ar@{->}[ur] && _{\langle 2, * \rangle}
 \ar@{->}[ddr]\ar@{->}[ur] \\
4&&&&  _{\langle 1, 3\rangle }  \ar@{->}[ur]  &&&&  _{\langle 2, 4 \rangle }  \ar@{->}[ur] \\
5&& _{\{ 3,*\}} \ar@{->}[uur]  &&&& _{\langle 1, 2 \rangle }  \ar@{->}[uur]  &&&& _{\langle 4, *\rangle}
}} \]
 By {\rm (Step 3)}, we can complete the label.
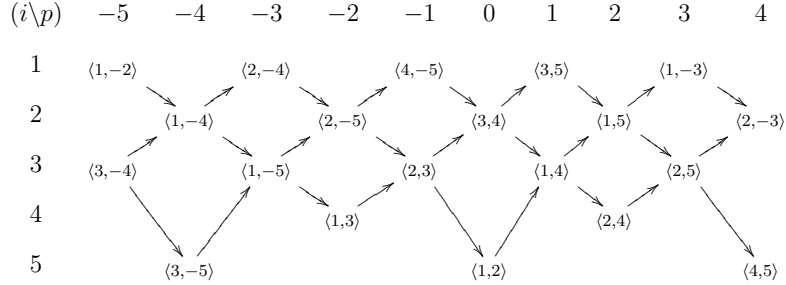
\[\scalebox{0.8}{\xymatrix@C=0.9ex@R=1ex{
( i\backslash p) & -5 & -4 & -3 & -2 & -1 & 0 & 1& 2& 3& 4 \\
1&_{ \langle 1, -2 \rangle }  \ar@{->}[dr]  && _{  \langle 2,-4 \rangle }  \ar@{->}[dr]  && _{\langle 4,-5 \rangle } \ar@{->}[dr]  && _{\langle  3,5\rangle }
 \ar@{->}[dr] &&  _{ \langle 1, -3 \rangle }  \ar@{->}[dr]\\
2&& _{ \langle 1, -4 \rangle } \ar@{->}[dr]\ar@{->}[ur]  && _{\langle 2,-5 \rangle} \ar@{->}[dr]\ar@{->}[ur]  &&  _{\langle 3,4 \rangle }  \ar@{->}[dr]\ar@{->}[ur]  && _{\langle 1,5\rangle}
\ar@{->}[dr]\ar@{->}[ur] && _{ \langle 2,-3 \rangle }  \\
3& _{\langle 3,-4 \rangle } \ar@{->}[ddr]\ar@{->}[ur]  && _{\langle 1, -5 \rangle } \ar@{->}[dr]\ar@{->}[ur]  && _{ \langle 2,3 \rangle }  \ar@{->}[ddr]\ar@{->}[ur]
&& _{ \langle 1, 4\rangle  }  \ar@{->}[dr]\ar@{->}[ur] && _{\langle 2,5\rangle}
 \ar@{->}[ddr]\ar@{->}[ur] \\
4&&&&  _{\langle 1, 3\rangle }  \ar@{->}[ur]  &&&&  _{\langle 2, 4 \rangle }  \ar@{->}[ur] \\
5&& _{\langle 3,-5\rangle} \ar@{->}[uur]  &&&& _{\langle 1, 2 \rangle }  \ar@{->}[uur]  &&&& _{\langle 4,5\rangle}
}} \]
\end{example}

\section{Twisted additive property of $D_{n+1}$} \label{Section:additive}

\begin{proposition} \label{Prop:tw_add_1}
Let us take a square with length of edge 1 in the folded AR-quiver $\widehat{\Upsilon}_{[\ii_0]}$:
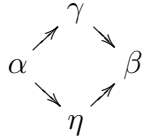
\[{\xymatrix@C=1.5ex@R=1ex{
 && \gamma \ar@{->}[dr]  && \\
 &\alpha\ar@{->}[dr]\ar@{->}[ur]  &&\beta \\
&& \eta \ar@{->}[ur]
 }}
 \]
 \begin{enumerate}[{\rm (i)}]
\item  There is a swing which has both root $\gamma$ $($resp. $\eta)$ and $\alpha$. Similarly, there is a swing which has both root
$\gamma$ $($resp. $\eta)$ and $\beta$.
\item  There are four swings which has $\alpha$ or $\beta$.
\item $\alpha+\beta=\eta+\gamma.$
\end{enumerate}
\end{proposition}

\begin{proof}
(i) is obvious. So here we give proofs for (ii) and (iii).
Let $\alpha$ be on the $N[\epsilon_1\cdot a_1]$ and $S[\epsilon_2\cdot a_2]$-sectional paths and
$\beta$ be on the $N[\varepsilon_1\cdot b_1]$ and $S[\varepsilon_2\cdot b_2]$-sectional paths, where $\epsilon_1, \epsilon_2, \varepsilon_1, \varepsilon_2$
are $1$ or $-1$. It is obvious that $a_1\neq b_1$ and $a_2\neq b_2$. Observe that the $N^{{\rm ext}}[\pm a_1]$-sectional (resp. $N^{{\rm ext}}[\pm b_1]$-sectional)
quiver and the $S^{{\rm ext}}[\pm b_2]$-sectional (resp. $S^{{\rm ext}}[\pm a_2]$-sectional) quiver have an intersection $\gamma$ (resp. $\eta$). By (3) in Proposition \ref{Prop: properties of ext},
we have $a_1\neq b_2$ and $b_1\neq a_2$. Hence $a_1$, $b_1$, $a_2$ and $b_2$ are distinct.
Now, by Proposition \ref{Prop:m-snake shares p} and Corollary \ref{Cor: induced property ext},  $\gamma$ has $\epsilon_1\cdot a_1$ and $\varepsilon_2\cdot b_2$
as its summands and $\eta$ has $\epsilon_2\cdot a_2$ and $\varepsilon_1\cdot b_1$ as its summands, which implies $\alpha+\beta=\gamma+\eta.$
\end{proof}

\begin{proposition}\label{Prop:tw_add_2}
Let us take three vertices in the first and second residues of twisted AR-quiver $\Upsilon_{[\ii_0]}$ such that :
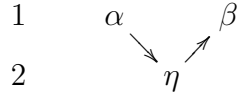
\[{\xymatrix@C=1.5ex@R=1ex{
 1 &&\alpha\ar@{->}[dr] &&\beta \\
2&& & \eta \ar@{->}[ur]
 }}
 \]
Then we have $\alpha+\beta= \eta.$
\end{proposition}

\begin{proof}
It follows from the additive property of adapted AR-quiver of type $A_n$
if three vertices are induced from  ${}_1\Gamma_Q$ (resp. ${}_2\Gamma_{Q^*}$) in  Algorithm \ref{Alg:label}.
Otherwise, $\alpha=\left< a_1, \pm (n+1)\right>$ and $\beta=\left< b_1,\pm( -n-1)\right>$.
Then we know $\eta=\alpha+\beta$ by Proposition \ref{Prop:label for Cvertices}.
\end{proof}

\begin{proposition}\label{Prop:tw_add_3}
Let us take six vertices in from $n-2$-th to $n+1$-th residues of twisted AR-quiver $\Upsilon_{[\ii_0]}$ such that :
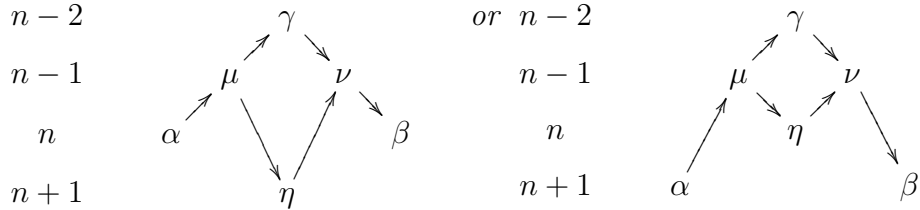
\[{\xymatrix@C=1.5ex@R=1ex{
n-2 && && \gamma \ar@{->}[dr]&& \\
n-1 &&& \mu\ar@{->}[ddr]\ar@{->}[ur] && \nu \ar@{->}[dr]&&\\
 n &&\alpha\ar@{->}[ur] && &&\beta \\
n+1 && && \eta \ar@{->}[uur]
 }}
 \text{ or }
 {\xymatrix@C=1.5ex@R=1ex{
 n-2 && && \gamma \ar@{->}[dr] && \\
n-1  &&& \mu\ar@{->}[dr] \ar@{->}[ur] && \nu\ar@{->}[ddr] &&\\
 n  && && \eta \ar@{->}[ur]    \\
n+1 &&\alpha\ar@{->}[uur] && &&\beta
 }}
  \]
Then we have $\alpha+\beta= \eta+\gamma=\mu+\nu.$
\end{proposition}

\begin{proof}
By Algorithm \ref{Alg:label}, we can see that summands of  $\alpha$ and $\beta$
lie in $\{1, 2, \cdots, n+1\}\cup \{-n-1\}$.
More precisely, if $\alpha$ (resp. $\beta$) has negative summand $(-n-1)$
then it is shared by every element in the $N[-n-1]$ (resp. $S[-n-1]$)-sectional path. Let $\alpha$
be on the $N[\epsilon\cdot a_1]$  and $S[a_2]$-sectional paths and $\beta$ be on the $N[b_1]$ and $S[\varepsilon\cdot b_2]$-sectional paths.
By Corollary \ref{cor: nbar adjacent}, we know that
\begin{itemize}
\item $\eta$ in on the $N[a_2]$-sectional and the $S[b_1]$-sectional paths,
\item $\gamma$ is on the $N^{{\rm ext}}[\pm a_1]$-sectional and $S^{{\rm ext}}[\pm b_2]$-sectional quivers.
\end{itemize}
Since $\alpha$ and $\gamma$ (resp. $\beta$ and $\gamma$) are in the same connected component of $N^{{\rm ext}}[\pm a_1]$-sectional
(resp. $S^{{\rm ext}}[\pm b_2]$-sectional) quiver,
$\gamma$ is on the $N[\epsilon\cdot a_1]$ and $S[\varepsilon\cdot b_2]$-sectional quivers.
Since $\gamma$ and $\eta$ have $\epsilon\cdot a_1$, $a_2$, $b_1$ and $\varepsilon\cdot b_2$ as their summands, we conclude that $$\alpha+\beta=\eta+\gamma.$$
Now, by Proposition \ref{Prop:tw_add_1}, our assertion follows.
\end{proof}

By Proposition \ref{Prop:tw_add_1}, Proposition \ref{Prop:tw_add_2} and Proposition \ref{Prop:tw_add_3},
we obtain {\it the twisted additive property} of $\Upsilon_{[\ii_0]}$, which is stated as follows:

\begin{theorem}[Twisted additive property of $D_{n+1}$] \label{twisted additive property}
Let $\alpha$ and $\beta$ be two roots in $\Phi^+$ whose folded coordinates are $(\bar{i}, p-2^{|\bar{i}|})$ and $(\bar{i}, p)$ in the folded AR-quiver $\widehat{\Upsilon}_{[\ii_0]}$. Here $|\bar{i}|$ denotes the number of indices in the orbit $\bar{i}$.  Then
\[ \alpha+\beta=\sum_{\gamma \in \mathcal{J}} \gamma,\]
where $\mathcal{J}\subset \Phi^+$  consists of $\gamma$ which are on paths from $\alpha$ to $\beta$ and have coordinates $(\bar{j},p-2^{|\bar{i}|-1})$ for $\bar{j}\in \bar{I}.$
\end{theorem}

\begin{remark} For adapted classes of type $ADE_n$ and twisted adapted classes of type $A_{2n+1}$, additive properties can be understand as analogous properties of Theorem \ref{twisted additive property}.
\begin{enumerate}
\item In order to see the cases for adapted classes, we can consider ordinary Coxeter elements are associated to the identity map on $I$. Then the order of every orbit $|\bar{i}|=1.$ (Note that the identity map on $I$ is not compatible with the involution $*$). Observe that, in
\eqref{eq: addtive property}, the coordinates of $\beta$  and $\phi_Q(\beta)$ are $(i,p)$ and $(i,p-2)=(i, p-2^{|\bar{i}|})$, respectively. Hence, for $\alpha= \phi_Q(\beta)$, the additive property
\eqref{eq: addtive property} can be rewritten as follows:
 \[ \alpha+\beta=\sum_{\gamma \in \mathcal{J}} \gamma,\]
where $\mathcal{J}\subset \Phi^+$  consists of $\gamma$ which are on paths from $\alpha$ to $\beta$ and have coordinates $(j,p-2^{|\bar{i}|-1})=(j, p-1)$ for $j\in I.$
\item In \cite[Proposition 7.7]{OS16}, we proved the analogous statement holds for twisted adapted classes of type $A_{2n+1}.$
\end{enumerate}
\end{remark}

\begin{example}
See the last quiver in Example \ref{Ex:twistedD_{n+1}}. Then the twisted additive property can be checked in this quiver. For example,
\[ \left<2, 4\right>+\left<3,5\right>=\left<4,5\right>+\left<2,3\right>, \quad \left<2,-5\right>+\left<4,5\right>=\left<2,4\right>, \quad
\left<1,3\right>+\left<2,5\right>=\left<3,5\right>+\left<1,2\right>.\]
\end{example}

\vskip 3mm

\section{ Twisted Dynkin quiver and Reflection functors} \label{sec: Twisted Dynkin quiver and Reflection} \label{Sec:twistedDynkin}

Recall the diagram \eqref{eq: 2n-1many 2}. For the twisted adapted classes, we introduced the twisted AR-quivers,
the twisted Coxeter elements and the  twisted adapted classes. Hence the only thing we want to generalize in  \eqref{eq: 2n-1many 2} is a Dynkin quiver $Q$.

\begin{definition}\label{Def:twistedQ}
A {\it twisted Dynkin quiver} of $D_{n+1}$ is obtained by giving an orientation to each edge in one of the following diagrams:
\begin{equation} \label{Eqn:D1}
 \xymatrix@R=4ex{
*{\circ}<3pt> \ar@{-}[r]_<{1 \ \ }  &*{\circ}<3pt>
\ar@{-}[r]_<{2 \ \ } & *{\circ}<3pt>
\ar@{-}[r]_<{3 \ \ } &*{\circ}<3pt>\ar@{.}[r]_<{4 \ \ }& *{\circ}<3pt>\ar@{.}[r]_<{n-1\ \ } &*{\odot}<3pt> \ar@{-}[l]^<{\  \ {n \choose n+1}  }}
\end{equation}
\begin{equation} \label{Eqn:D2}
  \xymatrix@R=4ex{
*{\circ}<3pt> \ar@{-}[r]_<{1 \ \ }  &*{\circ}<3pt>
\ar@{-}[r]_<{2 \ \ } & *{\circ}<3pt>
\ar@{-}[r]_<{3 \ \ } &*{\circ}<3pt>\ar@{.}[r]_<{4 \ \ }& *{\circ}<3pt>\ar@{.}[r]_<{n-1\ \ } &*{\otimes}<3pt> \ar@{-}[l]^<{\  \ {n+1 \choose n}  }}
\end{equation}
We denote by $\Qn$ a Dynkin quiver obtained from the diagram \eqref{Eqn:D1} and by $\Qnn$ a Dynkin quiver obtained from the diagram \eqref{Eqn:D2}.
\end{definition}

\begin{remark} \label{rem: C_n Dynkin}
Considering the number of indices to each vertex, the diagrams in the above definition can be understood as the Dynkin diagram of type $C_n$.
\end{remark}

\begin{example}
For $Q = {\xymatrix@R=3ex{ \circ
\ar@{->}[r]_<{ \ 1} &  \circ
\ar@{<-}[r]_<{ \ 2}  &  \circ
\ar@{<-}[r]_<{ \ 3}
& \circ \ar@{-}[l]^<{\ \ \ \ \ \ 4} }}$ of type $A_{4}$, $Q^{\gets 4}$ can be depicted as follows:
$$Q^{\gets 4} = {\xymatrix@R=3ex{ \circ
\ar@{->}[r]_<{ \ 1} &  \circ
\ar@{<-}[r]_<{ \ 2}  &  \circ
\ar@{<-}[r]_<{ \ 3}
& \odot \ar@{-}[l]^<{\ \ \ \ \ \ {4 \choose 5}} }}$$
\end{example}

\begin{remark}
Definition \ref{Def:twistedQ} is natural in the sense that \eqref{eq: C_n} is related to the folding operation letting $n$ and $n+1$ be overlapped.
\end{remark}

\begin{definition} \hfill
\begin{enumerate}
\item We say $i\neq n, n+1$ is a {\it sink} (resp. {\it source}) of $\Qn$ or $\Qnn$ if every arrow connecting $i$ and another vertex $j$ is entering into  (resp. exiting out of) $i$.
\item  We say $i=n$ is a {\it sink} (resp. {\it source}) of $\Qn$ if there is the arrow from $n-1$ to ${n \choose n+1}$ (resp.  from ${n \choose n+1}$ to $n-1$).  Similarly, $i=n+1$ is a {\it sink} (resp. {\it source}) of $\Qnn$ if there is the arrow from $n-1$ to ${n+1 \choose n}$ (resp.  from ${n+1 \choose n}$ to $n-1$).
\item A quiver  $\Qn$ cannot have $n+1$ as a sink or a source. Similarly, $\Qnn$ cannot have $n$ as a sink or a source.
\end{enumerate}
\end{definition}

\begin{definition}
For $i\in I_{n+1}$, the {\it reflection function} $r_i$ on $\Qn$ (resp. $\Qnn$) is defined as follows.
\begin{itemize}
\item If $i\neq n, n+1$ and $i$ is a sink or a source in $\Qn$ (resp. $\Qnn$) then $i \ \Qn$ (resp. $i \ \Qnn$)
is obtained by  reversing all the arrows entering in to or exiting out of $i$.
\item  If $i\neq n, n+1$ and $i$ is neither a sink nor a source then $i \ \Qn=\Qn$ (resp. $ i \ \Qnn=\Qnn$).
\item If $i=n$ (resp. $i=n+1$) then $i \ \Qn$ (resp. $ i \ \Qnn$) is obtained by (i) substituting $\odot$
(resp. $\otimes$) by $\otimes$ (resp. $\odot$) and (ii) reversing the arrow  entering into or exiting out of $i$.
\item If $i=n+1$ (resp. $i=n$) then  $ \ i \Qn=\Qn$ (resp. $ i \Qnn=\Qnn$).
\end{itemize}
\end{definition}

Now we can associate a reduced expression to a twisted Dynkin quiver.

\begin{definition}\label{def: adapted to twisted DQ}
Let $\ii=i_1\, i_2\, \cdots, i_{\ell(w)}$ be a reduced word of $w\in W^D_{n+1}.$ We say $\ii$ is {\it  adapted} to $\Qn$ (resp. $\Qnn$) if
\[ i_k \text{ is a sink of } i_{k-1} i_{k-2} \cdots i_2 i_1 \ \Qn \text{ (resp. $i_{k-1} i_{k-2}  \cdots i_2 i_1 \ \Qnn$)}\]
for $k=1, 2, \cdots, \ell(w).$
\end{definition}

The following proposition shows the interesting relations between twisted Dynkin quivers and reduced expressions.

\begin{proposition} \label{Prop:twisted_redEX-quiv}\
\begin{enumerate}
\item[{\rm (1)}] There is the canonical one-to-one correspondence between the set of twisted Coxeter elements and the set of twisted Dynkin quivers.
Thus we can define $\phi_{\Qn}$ and $\phi_{\Qnn}$ in a natural way.
\item[{\rm (2)}] Let us denote $ \phi_{\lf \Qg \rf}:=\{\, \phi_{\Qn}, \phi_{\Qnn}\, | \, \Qn \text{ and } \Qnn \text{ are twisted Dynkin quivers } \}$. Then
\[ \phi_{\lf \Qg\rf}= \{\,  \phi_{[\ii_0]}\, | \, [\ii_0] \text{ is a twisted adapted class }\}.\]
\item[{\rm (3)}] If $[\ii_0]$ is a twisted adapted class of reduced expressions with twisted Coxeter element $\phi_{\Qn}$ (resp. $\phi_{\Qnn}$)
then $\ii_0$ is adapted to $\Qn$ (resp. $\Qnn$).
\end{enumerate}
\end{proposition}

\begin{proof}
(1) follows by Proposition \ref{prop: 2 to 1 Dt to A}.
(2) follows from the correspondence between twisted Coxeter elements and twisted Dynkin quivers proved in (1) and
the correspondence between twisted Coxeter elements and classes of twisted adapted classes. Finally, (3) follows from the facts that
\begin{enumerate}[(i)]
\item if $i_1\, i_2$ is adapted to $\phi_{\Qn}$ (resp. $\phi_{\Qnn}$) and $[i_1\, i_2]= [i_2\, i_1]$ then
$i_2\, i_1$ is also adapted to $\phi_{\Qn}$ (resp. $\phi_{\Qnn}$),
\item if $(s_{i_1}\, s_{i_2}\, \cdots \, s_{i_n})\vee$ is the twisted Coxeter element associated to $\Qn$ (resp. $\phi_{\Qnn}$)
then $ i_n \cdots i_2 i_1 \ \Qn = \Qnn$ (resp. $ i_n \cdots i_2 i_1 \ \Qnn = \Qn$) and  hence
$\prod_{k=0}^{n} (s_{i_1}\, s_{i_2}\, \cdots \, s_{i_n})^{k\vee}$ is adapted to  $\Qn$ (resp. $\Qnn$).
 \end{enumerate}
\end{proof}

Using Proposition \ref{Prop:twisted_redEX-quiv}, let us introduce simpler notations:

\begin{definition}\label{Def:diamond}
We denote by $\lf \Qg\rf$ the twisted adapted cluster point of $D_{n+1}$ and denote by $[\Qn]$ (resp. $[\Qnn]$)
the class of  reduced expressions  adapted to  $\Qn$  (resp.  $\Qnn$).
Also, we denote by $\ii_0\in [\Qg]$ when $\ii_0$ is  twisted adapted  and denote by $\Qg$ the twisted Dynkin quiver.
\end{definition}

With the notion of sinks and sources of $\Qg$, we have an analogue of Lemma \ref{lem: Leclerc} for folded AR-quiver:

\begin{lemma} For any folded AR-quiver, positions of simple roots inside of $\widehat{\Upsilon}_{[\Qg]}$ are on the boundary of
$\widehat{\Upsilon}_{[\Qg]}$; that is, either
\begin{enumerate}
\item[{\rm (i)}] $\al_i$ are a sink or a source $\widehat{\Upsilon}_{[\Qg]}$, or
\item[{\rm (ii)}] $\al_i$ has $\overline{1}$ or $\overline{n}$ as its folded residue.
\end{enumerate}
Furthermore,
\begin{enumerate}
\item[{\rm (a)}] all sinks and sources of $\widehat{\Upsilon}_{[\Qg]}$ have their labels as simple roots,
\item[{\rm (b)}] $i$ is a sink $($resp. source$)$ of $\Qg$ if and only if $\al_i$ is a sink $($resp. source$)$ of $\widehat{\Upsilon}_{[\Qg]}$
\end{enumerate}
\end{lemma}

\begin{proof}
It is an immediate consequence of Algorithm \ref{Alg:label} and Lemma \ref{lem: Leclerc} for $Q$ of type $A_n$.
\end{proof}

\begin{theorem}
Using {\rm Proposition \ref{Prop:twisted_redEX-quiv}} and {\rm Definition \ref{Def:diamond}}, we obtain the analogue of \eqref{eq: 2n-1many 2}:
\begin{align}\label{eq: Corres_twisted}
\raisebox{3.6em}{\xymatrix@C=8ex@R=4ex{
& \{ \prec_{\Qg} \} \ar@{<->}[dl]_{1-1} \ar@{<->}[d]_{1-1} \ar@{<->}[dr]^{1-1}  \\
\{ [\Qg] \} \ar@{<->}[r]^{1-1} & \{ \phi_{\Qg} \} \ar@{<->}[r]^{1-1} & \{ \Qg \} \\
& \{ \Upsilon_{[\Qg]} \} \ar@{<->}[ul]^{1-1}\ar@{<->}[ur]_{1-1} \ar@{<->}[u]^{1-1}}} \text{for $[\Qg] \in \lf \Qg \rf$}.
\end{align}
\end{theorem}

We give a proposition about relations between $\Qg$ and $\Upsilon_{[\Qg]}$ without proof.

\begin{proposition}\
\begin{enumerate} [{\rm (1)}]
\item The subquiver consisting of $\{\, (i, \xi(\bar{i}))\, |\, i=1, \cdots, n\, \}$ of  $\Upsilon_{[\Qg]}$ is isomorphic to $\Qg,$ where $\xi$ is the height function on $\bar{I}.$
\item Using $\Qg$ and $\phi_{\Qg}$, we can recover $\Upsilon_{[\Qg]}.$
\item If $i\in I_{n+1}$ is a sink of $\Qg$ then $i \ \Qg$ is associated twisted Dynkin quiver to $\Upsilon_{[\Qg]}\cdot r_i$;
i.e. $\Upsilon_{[i \Qg]}=\Upsilon_{[\Qg]}\cdot r_i$. Here $\Upsilon_{[\Qg]}\cdot r_i$ is the twisted AR-quiver associated to $r_i[\Qg]$.
\end{enumerate}
\end{proposition}

Now we described the reflection functor $r_i$ on $[\Qg]$ (equivalently, the map obtaining $\Upsilon_{[i \Qg]}$ from $\Upsilon_{[\Qg]}$)
in a combinatorial way:
\begin{algorithm} \label{alg: tRef Q}
Let $\mathsf{h}^\vee=2n+2$ be a dual Coxeter number of $D_{n+1}$, an index  $i$ be a sink of $\Qg \in \lf \Qg \rf$ and $*$ be the involution of $D_{n+1}$ such that $w_0(\alpha_i)=-\alpha_{i^*}$.
\begin{enumerate}
\item[{\rm (D1)}] Remove the vertex $(i,\xi(\bar{i}))$ and the arrows entering into the vertex.
\item[{\rm (D2)}] Add the vertex $(i^*,\xi(\bar{i})-\mathsf{h}^\vee)$ and the arrows to all vertices
whose coordinates are $(\, j\, ,\, \xi(\bar{i})-\mathsf{h}^\vee+1\, ) \in \Upsilon_{[\ii_0]}$ where $j$ is an index adjacent to $i^*$ in the Dynkin diagram
$\Delta^D_{n+1}$.
\item[{\rm (D3)}] Label the vertex $(i^*,\xi(\bar{i})-\mathsf{h}^\vee)$ with $\al_i$ and change the labels $\be$ to $s_i(\be)$ for all $\be \in \Upsilon_{[\Qg]} \setminus \{\al_i\}$.
\end{enumerate}
\end{algorithm}

\begin{example} \label{Ex:5.12} Let $[\ii_0]$ be a twisted adapted class of $D_5$ with its twisted Coxeter element $(s_1 s_2 s_5 s_3)\vee$.
Then the associated twisted Dynkin quiver is $Q^{\gets 5}$:
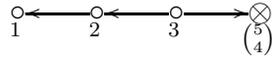
\begin{equation*}
  \xymatrix@R=4ex{
*{\circ}<3pt> \ar@{<-}[r]_<{1 \ \ }  &*{\circ}<3pt>
\ar@{<-}[r]_<{2 \ \ } & *{\circ}<3pt>
\ar@{->}[r]_<{3 \ \ }  &*{\otimes}<3pt> \ar@{-}[l]^<{\  \ {5 \choose 4}  }}
\end{equation*}
and  the associated twisted AR-quiver is
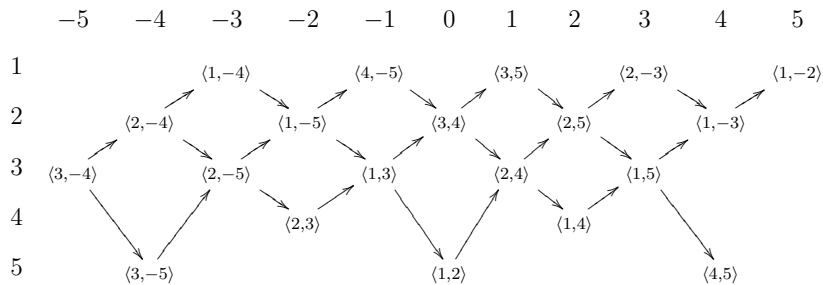
\[ \scalebox{0.8}{\xymatrix@C=0.9ex@R=1ex{
& -5 & -4 & -3& -2& -1& 0 & 1& 2& 3& 4& 5&   \\
1&&& \srt{1}{-4}\ar@{->}[dr]  && \srt{4}{-5}\ar@{->}[dr]  && \srt{3}{5}\ar@{->}[dr]  && \srt{2}{-3}\ar@{->}[dr] && \srt{1}{-2}\\
2&& \srt{2}{-4}\ar@{->}[dr]\ar@{->}[ur]  && \srt{1}{-5}\ar@{->}[dr]\ar@{->}[ur]  && \srt{3}{4}\ar@{->}[dr]\ar@{->}[ur]  && \srt{2}{5}\ar@{->}[dr]\ar@{->}[ur] && \srt{1}{-3}\ar@{->}[ur] \\
3&\srt{3}{-4}\ar@{->}[ddr]\ar@{->}[ur]  && \srt{2}{-5}\ar@{->}[dr]\ar@{->}[ur]  && \srt{1}{3}\ar@{->}[ddr]\ar@{->}[ur]  && \srt{2}{4}\ar@{->}[dr]\ar@{->}[ur] && \srt{1}{5}\ar@{->}[ur]\ar@{->}[ddr] \\
4&&&& \srt{2}{3}\ar@{->}[ur]  &&&& \srt{1}{4}\ar@{->}[ur] \\
5&&\srt{3}{-5}\ar@{->}[uur]  &&&& \srt{1}{2}\ar@{->}[uur]  &&&& \srt{4}{5}
}} \]

Consider $Q^{\gets 4}= r_5 Q^{\gets 5}.$ Then $Q^{\gets 4}$ is
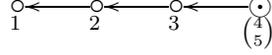
\begin{equation*}
  \xymatrix@R=4ex{
*{\circ}<3pt> \ar@{<-}[r]_<{1 \ \ }  &*{\circ}<3pt>
\ar@{<-}[r]_<{2 \ \ } & *{\circ}<3pt>
\ar@{<-}[r]_<{3 \ \ }  &*{\odot}<3pt> \ar@{-}[l]^<{\  \ {4 \choose 5}  }}
\end{equation*}
\end{example}
and $\Upsilon_{[Q^{\gets 4}]}= \Upsilon_{[Q^{\gets 5}]} \cdot r_5$ is
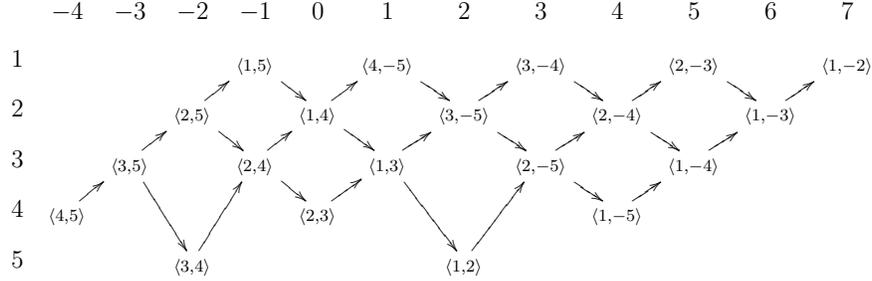
\[ \scalebox{0.8}{\xymatrix@C=0.9ex@R=1ex{
& -4 & -3 & -2 & -1 & 0 & 1& 2& 3& 4& 5& 6& 7 \\
1&&&& \srt{1}{5}\ar@{->}[dr]  && \srt{4}{-5}\ar@{->}[dr]  && \srt{3}{-4}\ar@{->}[dr]  && \srt{2}{-3}\ar@{->}[dr] && \srt{1}{-2} \\
2&&& \srt{2}{5}\ar@{->}[dr]\ar@{->}[ur]  && \srt{1}{4}\ar@{->}[dr]\ar@{->}[ur] && \srt{3}{-5}\ar@{->}[dr]\ar@{->}[ur] && \srt{2}{-4}\ar@{->}[dr]\ar@{->}[ur] && \srt{1}{-3}\ar@{->}[ur] \\
3&& \srt{3}{5}\ar@{->}[ddr]\ar@{->}[ur] && \srt{2}{4}\ar@{->}[dr]\ar@{->}[ur] && \srt{1}{3}\ar@{->}[ddr]\ar@{->}[ur] && \srt{2}{-5} \ar@{->}[dr]\ar@{->}[ur] && \srt{1}{-4}\ar@{->}[ur] \\
4&\srt{4}{5}\ar@{->}[ur] &&&& \srt{2}{3}\ar@{->}[ur] &&&& \srt{1}{-5}\ar@{->}[ur] \\
5&&& \srt{3}{4}\ar@{->}[uur] &&&& \srt{1}{2} \ar@{->}[uur]
}}\]

Let us introduce an interesting observation (see Algorithm \ref{alg: fRef Q} below).% in the sense that twisted AR-quiver of type $D_{n+1}$
%can be interpreted by using the notions on type $C_{n}$.

\begin{definition} \label{def: Cn}
  Note that $\overline{I}$ can be thought as an index set of $C_{n}$.
\begin{enumerate}
\item[{\rm (a)}] Let $\overline{\mathsf{D}}={\rm diag}(d_{\overline{i}} \ | \ \overline{i} \in \overline{I} )={\rm diag}(1,1,\cdots, 1, 2)$ be the diagonal matrix
which diagonalizes the Cartan matrix $\cmA=(a_{\overline{i}\overline{j}})$ of type $C_{n}$.
\item[{\rm (b)}] We denote by $\overline{\mathsf{d}}={\rm lcm}(d_{\overline{i}} \ | \ \overline{i} \in \overline{I})=2$.
\end{enumerate}
\end{definition}
Recall that the involution $^*$ induced by the longest element
$w_0$ of type $C_{n}$ is an identity map $\overline{i} \mapsto \overline{i}$.
Suppose $i$ is a sink of $\Qg$. We shall describe the algorithm which shows a way of obtaining $\widehat{\Upsilon}_{[i \ \Qg]}$ from
$\widehat{\Upsilon}_{[\Qg]}$ by using the notations on $C_n$. (cf. Algorithm \ref{alg: tRef Q}.)

\begin{algorithm} \label{alg: fRef Q}
Let $\mathsf{h}_C^\vee=n+1$ be a dual Coxeter number of $C_{n}$ and $\al_i$ be a sink of $[\ii_0] \in \lf \Qg \rf$.
\begin{enumerate}
\item[{\rm (D1)}] Remove the vertex $(\overline{i},\xi(\bar{i}))$ and the arrows entering into $(\overline{i},\xi(\bar{i}))$ in
$\widehat{\Upsilon}_{[\Qg]}$.
\item[{\rm (D2)}] Add the vertex $(\overline{i},\xi(\bar{i})-\overline{\mathsf{d}} \times \mathsf{h}_C^\vee)$ and
the arrows to all vertices whose coordinates are $(\overline{j}, \xi(\bar{i})-\overline{\mathsf{d}} \times \mathsf{h}_C^\vee+1) \in \widehat{\Upsilon}_{[\Qg]}$,
where $\overline{j}$ is adjacent to $\overline{i^*}$ in Dynkin diagram of type $C_n$.
\item[{\rm (D3)}] Label the vertex $(\overline{i},\xi(\bar{i})-\overline{\mathsf{d}} \times \mathsf{h}_C^\vee)$ with $\al_i$ and change the labels $\be$ to $s_i(\be)$ for all $\be \in \widehat{\Upsilon}_{[\Qg]}
\setminus \{\al_i\}$.
\end{enumerate}
\end{algorithm}

\section{Distances and radius with respect to $\lf \Qg \rf$} \label{Sec:dist_rds}

\subsection{Notions} In this subsection, we briefly review the notions on sequences of positive roots
which were mainly introduced in \cite{Mc12,Oh15E}.

\begin{convention} \label{conv: sequence}
Let us choose any reduced expression $\jj_0=i_1i_2 \cdots i_{\N}$ of $w_0 \in W$ of any finite type and fix a convex total order $\le_{\jj_0}$
induced by $\jj_0$ and a labeling of $\Phi^+$ as follows:
$$\beta^{\jj_0}_k \seteq s_{i_1}\cdots s_{i_{k-1}}\alpha_{i_k} \in \Phi^+, \   \beta^{\jj_0}_k \le_{\jj_0} \beta^{\jj_0}_l \text{ if and only if } k \le l.$$

\begin{enumerate}
\item[{\rm (i)}] We identify a sequence $\um_{\jj_0}=(m_1,m_2,\ldots,m_{\N}) \in \Z_{\ge 0}^{\N}$ with
$$(m_1\beta^{\jj_0}_1 ,m_2\beta^{\jj_0}_2,\ldots,m_{\N} \beta^{\jj_0}_{\N}) \in (\rl^+)^{\N},$$
where $\rl^+$ is the positive root lattice.
\item[{\rm (ii)}] For a sequence $\um_{\jj_0}$ and another reduced expression ${\jj_0}'$ of $w_0$, the sequence $\um_{{\jj_0}'}\in\Z_{\ge 0}^{\N}$ is induced from $\um_{\jj_0}$ as follows : (a) Consider $\um_{\jj_0}$  as a sequence of positive roots, (b)
rearranging with respect to $<_{\jj_0'}$, (c)  applying the convention
{\rm (i)}.
\end{enumerate}
For simplicity of notations, we usually drop the script ${\jj_0}$ if there is no fear of confusion.
\end{convention}

\begin{definition} \cite[Definition 1.10]{Oh15E} \hfill
\begin{enumerate}
\item[{\rm (i)}] A sequence $\um$ is called a {\it pair} if $|\um|\seteq \sum_{i=1}^\N m_i=2$ and $m_i \le 1$ for $1\le i\le \N$.
\item[{\rm (ii)}] The {\it weight} $\wt(\um)$ of a sequence $\um$ is defined by $\sum_{i=1}^{\N} m_i\beta_i \in \rl^+$.
\end{enumerate}
\end{definition}

We mainly use the notation $\up$ for a pair.
We also write $\up$ as $(\al,\be) \in \PR$ or $(\up_{i_1},\up_{i_2})$ where $\be_{i_1}=\al$, $\be_{i_2}=\be$ and $i_1 < i_2$.

\begin{definition}\cite{Mc12,Oh15E} We define partial orders $<^\tb_{\jj_0}$ and $\prec^\tb_{[\jj_0]}$ on $\Z_{\ge 0}^{\N}$ as follows:
\begin{enumerate}
\item[{\rm (i)}] $<^\tb_{\jj_0}$ is the bi-lexicographical order induced by $<_{\jj_0}$.
\item[{\rm (ii)}] For sequences $\um$ and $\um'$, $\um \prec^\tb_{[\jj_0]} \um'$ if and only if $\wt(\um)=\wt(\um')$ and
$\um_{\jj_0'} <^\tb_{\jj_0'} \um'_{\jj_0'}$ for all $\jj_0' \in [\jj_0]$.
\end{enumerate}
\end{definition}

\begin{definition} \cite[Definition 1.13, Definition 1.14]{Oh15E} \hfill
\begin{enumerate}
\item A pair $\up$ is called $[\jj_0]$-simple if there exists no sequence $\um \in \Z^{\ell(w)}_{\ge 0}$ such that $\um \prec^\tb_{[\jj_0]} \up$
\item A sequence $\um=(m_1,m_2,\ldots,m_{\ell(w)}) \in \Z^{\ell(w)}_{\ge 0}$ is called $[\jj_0]$-simple if $|\um|=1$ or all pairs $(\up_{i_1},\up_{i_2})$
such that $m_{i_1},m_{i_2} >0$ are $[\jj_0]$-simple pairs.
\end{enumerate}
\end{definition}

\begin{definition} \cite[Definition 1.15]{Oh15E} For a given $[\jj_0]$-simple sequence $\us=(s_1,\ldots,s_{\N}) \in \Z^{\N}_{\ge 0}$,
we say that a sequence $\um \in \Z^{\N}_{\ge 0}$ is called a {\it $[\jj_0]$-minimal sequence of $\us$} if {\rm (i)} $\us \prec^\tb_{\jj_0} \um$ and
{\rm (ii)} there exists no sequence $\um'\in \Z^{\N}_{\ge 0}$ such that
$$  \us \prec^\tb_{[\jj_0]} \um' \prec^\tb_{[\jj_0]}  \um.$$
\end{definition}

\begin{definition} \cite[Definition 4.8]{Oh15E} \label{def: gdist}
\begin{enumerate}
\item We say that a sequence $\um$ has {\it generalized $[\jj_0]$-distance $k$}
$(k \in\Z_{\ge 0})$, denoted by $\gdist_{[\jj_0]}(\um)=k$, if $\um$ is {\it not} $[\jj_0]$-simple and satisfies the following properties:
\begin{enumerate}
\item[{\rm (i)}] there exists a set of non $[\jj_0]$-simple sequences $\{ \um^{(s)} \ | \ 1 \le s \le k \}$ such that
\begin{align}\label{eq: gdist seq}
\um^{(1)} \prec^\tb_{[\jj_0]} \cdots \prec^\tb_{[\jj_0]} \um^{(k)}=\um,
\end{align}
\item[{\rm (ii)}] the set of non $[\jj_0]$-simple sequences $\{ \um^{(s)} \}$ has the maximal cardinality among sets of sequences satisfying \eqref{eq: gdist seq}.
\end{enumerate}
\item If $\um$ is $[\jj_0]$-simple, we define $\gdist_{[\jj_0]}(\um)=0$.
\end{enumerate}
\end{definition}

\begin{remark}
For any pair $\{ \alpha,\beta\} \in (\PR)^2$, we also use the notation $\gdist_{[\jj_0]}(\al,\be)$ in a natural way; that is,
$$\gdist_{[\jj_0]}(\al,\be) =\begin{cases}
\gdist_{[\jj_0]}(\al,\be) & \text{ if } \al \prec^\tb_{[\jj_0]} \be,\\
\gdist_{[\jj_0]}(\be,\al) & \text{ if } \be \prec^\tb_{[\jj_0]} \al,\\
0 & \text{ if $\al$ and $\be$ is incomparable with respect to $\prec^\tb_{[\jj_0]}$}.
\end{cases}
$$
\end{remark}

\begin{definition}\cite[Definition 1.19]{Oh15E}
For a non-simple positive root $\gamma \in \PR \setminus \Pi$, the {\it $[\jj_0]$-radius} of $\gamma$, denoted by $\rds_{[\jj_0]}(\gamma)$,
is the integer defined as follows:
$$\rds_{[\jj_0]}(\gamma)=\max({\rm gdist}_{[\jj_0]}(\up) \ | \ \gamma  \prec_{[\jj_0]}^\tb \up).$$
\end{definition}

\begin{definition}\cite[Definition 1.21]{Oh15E} \label{def: jj_0-socle}
For a pair $\up$, the {\it $[\jj_0]$-socle} of $\up$, denoted by $\soc_{[\jj_0]}(\up)$, is a $[\jj_0]$-simple sequence $\us$ such that
$$ \us \preceq^\tb_{[\jj_0]} \up,$$
if such $\us$ exists uniquely.
\end{definition}

\subsection{Distances and radius for twisted adapted cluster points}
Before we investigate twisted adapted cluster points, we recall the following theorem which is about adapted cluster points.

\begin{theorem}\cite{Oh14D,Oh15E} \label{thm: known for Q} Let $Q$ be any Dynkin quiver of type $AD_m$.
\begin{enumerate}
\item[{\rm (1)}] For any pair $(\al,\be) \in (\PR)^2$,  we have $\gdist_{[Q]}(\al,\be) \le \max\{ \mathsf{m}(\gamma)\ | \ \gamma=\alpha,\beta \  \}$.
\item[{\rm (2)}] For any $\gamma \in \PR \setminus \Pi$, we have $$\rds_{[Q]}(\gamma)=\mathsf{m}(\gamma).$$
\item[{\rm (3)}] For any $\up$, $\soc_{[Q]}(\up)$ is well-defined.
\end{enumerate}
\end{theorem}

In this section we prove the following theorem:
\begin{theorem} \label{thm: dist upper bound Qd} Take any $[\Qg] \in \lf\Qg\rf$ of type $D_{n+1}$.
\begin{enumerate}
\item[{\rm (1)}] For any pair $(\al,\be) \in (\PR)^2$, we have $\gdist_{[\Qg]}(\al,\be) \le \max\{ \overline{\mathsf{m}}(\gamma) \ | \ \ga=\alpha,\beta\,   \}=2$.
\item[{\rm (2)}] For any $\gamma \in \PR \setminus \Pi$,
\begin{equation}
\rds_{[\Qg]}(\gamma)= \overline{\mathsf{m}}(\ga).
\end{equation}
\item[{\rm (3)}] For any $\up$, $\soc_{[\Qg]}(\up)$ is well-defined.
\end{enumerate}
\end{theorem}

\begin{proposition} \label{prop:rds_1}
Let $\gamma\in \Phi^+\setminus \Pi$  and $\gamma \in \Upsilon^L_{[\Qg]}\cup \Upsilon^R_{[\Qg]}$. Then $$ \rds_{[\Qg]}(\gamma)=\overline{\mathsf{m}}(\ga)=1.$$
\end{proposition}

\begin{proof}
Note that $\gamma \in \Upsilon^L_{[\Qg]}\cup \Upsilon^R_{[\Qg]}$ is (folded) multiplicity free by Lemma \ref{lem: folded mul free iff}.
Thus any pairs $\up=(\al,\be)$ and $\up'=(\al',\be')$ with $\wt(\up)=\wt(\up')=\ga$ consist of multiplicity free positive roots $\al$, $\al'$, $\be$ and $\be'$.
By replacing $\al_{n+1}$ to $\al_n$, we can consider $\ga$, $\al$, $\al'$, $\be$ and $\be'$
as positive roots of type $A_n$. Assume that $\ga \in \Upsilon^R_{[\Qg]}$. By the convexity of $\prec_{[\Qg]}$, $\al$ and $\al'$ are also contained in
$\Upsilon^R_{[\Qg]}$. Then our assertion follows from Algorithm \ref{Alg:label} and
the fact that $\rds_{[Q]}(\gamma)=1$ (\cite[Theorem 4.15]{Oh15E}) for $Q=\mathfrak{p}^{D_{n+1}}_{A_n}([\Qg])$. For the case when $\ga \in \Upsilon^L_{[\Qg]}$, we can apply the same argument.
\end{proof}

\begin{proposition} \label{prop:rds_2}
For $\gamma \in \Upsilon^C_{[\Qg]}$, we have $$ \rds_{[\Qg]}(\gamma)=\overline{\mathsf{m}}(\ga)=2.$$
\end{proposition}

\begin{proof}
Note that $\gamma \in \Upsilon^C_{[\Qg]}$ is not folded multiplicity free by Lemma \ref{lem: folded mul free iff}.
Thus $$\gamma=\left< a, b \right> \quad \text{ for some } 1\leq a< b\neq n+1.$$
We have
$\gamma= \left< a, n+1 \right>+\left< b, -n-1 \right>= \left< a, -n-1 \right>+\left< b, n+1 \right>.$
Vertices $\left< a, -n-1 \right>$ and  $\left< b, -n-1 \right>$ are connected by a sectional path which is induced from the
path between $[a, n]$ and $[b,n]$ in $\Gamma_Q$ for $Q=\mathfrak{p}^{D_{n+1}}_{A_n}([\Qg])$.
Also, we can apply the same argument for the pair consisting of $\left< a, n+1 \right>$ and  $\left< b, n+1 \right>$.
Hence  $\left< a, -n-1 \right>$ and  $\left< b, -n-1 \right>$
(resp.  $\left< a, n+1 \right>$ and  $\left< b, n+1 \right>$) are always comparable with respect to $\prec_{[\Qg]}$ and
$$ \left< a, -n-1 \right> \prec_{[\Qg]} (\text{resp.} \succ_{[\Qg]})\left< b, -n-1 \right> \Longleftrightarrow
\left< a, n+1 \right>\prec_{[\Qg]} (\text{resp.} \succ_{[\Qg]})\left< b, n+1 \right>. $$
Thus we proved  $(\left< a, n+1 \right>,\left< b, -n-1 \right>)$ and $(\left< a, -n-1 \right>,\left< b, n+1 \right>)$ are comparable with respect to $\prec_{[\Qg]}^\tb$
and hence $\rds_{[\Qg]}(\gamma)\geq 2$ by the convexity of $\prec_{[\Qg]}$.

Conversely, let us show that $\rds_{[\Qg]}(\gamma)\leq 2.$ Suppose not. Then there are $c$ and $d$ such that
\begin{enumerate}
\item[{\rm (1)}] $|c|\neq |d| >b$,
\item[{\rm (2)}] $\underline{p}_1=(\left< a, c \right>,\left< b, -c \right>)$ and
$\underline{p}_2=(\left< a, d \right>,\left< b, -d \right>)$  are comparable with respect to $\prec_{[\Qg]}$.
\end{enumerate}
Here we assume further that
$$\text{(i) $c \ne d>0$,  (ii) $\underline{p}_1\succ_{[\Qg]}^\tb \underline{p}_2$ and (iii) $\left< a, c \right> \succ_{[\Qg]}\left< b, -c \right>$},$$
since other cases can be proved similarly.
By (ii) and (iii), we have $$\left< a, c \right> \succ_{[\Qg]}\left< a, d \right>,\left< b, -d \right> \succ_{[\Qg]}  \left< b, -c \right>.$$
Hence the convexity of $\prec_{[\Qg]}$ tells that two  roots $\left< b, -d \right>$ and $\left< b, -c \right>$ should be in $\Upsilon_{[\Qg]}^R.$
Moreover, the assumption  $\left< a, c \right> \succ_{[\Qg]}\left< a, d \right>,\left< b, -d \right> \succ_{[\Qg]}  \left< b, -c \right>$
implies that we can find
one of the following rectangles (not necessarily squares) in  ${}_1\Gamma_Q$ with vertices $[a, d-1]$, $[a, c-1]$, $[b,d-1]$ and $ [b,c-1]$:
 \[ {\xymatrix@C=0.5ex@R=1ex{
 && [b, c-1] \ar@{.>}[dr]  && \\
 &[b, d-1]\ar@{.>}[dr]\ar@{.>}[ur]  &&[a,c-1] \\
&& [a,d-1] \ar@{.>}[ur]
 }},
 \qquad {\xymatrix@C=0.5ex@R=1ex{
 && [a, c-1] \ar@{.>}[dr]  && \\
 &[a, d-1]\ar@{.>}[dr]\ar@{.>}[ur]  &&[b,c-1] \\
&& [b,d-1] \ar@{.>}[ur]
 }}.
 \]
In both cases, Algorithm \ref{Alg:label} and {\rm (ii)} in Remark \ref{rem:LR path} tell that
\begin{itemize}
\item $\langle a,-c \rangle$ and $\langle a,-d \rangle$ are contained in $\Upsilon^{R}_{[\Qg]}$,
\item there exists a path $\langle a,d \rangle \to \langle a,c \rangle$ in ${}_2\Gamma_{Q^*} \subset \Upsilon_{[\Qg]}$.
\end{itemize}
Hence $\up_1$ and $\up_2$ are not comparable with respect to $\prec^\tb_{[\Qg]}$. Thus our assertion follows.
\end{proof}

\begin{proof} [{\bf The first step for  Theorem \ref{thm: dist upper bound Qd}}]
From the above three lemmas, the second assertion of Theorem \ref{thm: dist upper bound Qd} follows. Furthermore, the first and the third assertions
for $\al+\be \in \PR$ also hold.
\end{proof}

\begin{proposition} \label{Lem:gdist_1}
Let both $\alpha$ and $\beta$ be positive roots in $\Upsilon_{[\Qg]}^L$ $($resp. $\Upsilon_{[\Qg]}^R)$. Then
\[ \gdist_{[\Qg]}(\alpha,\beta)\leq \max\{ \overline{\mathsf{m}}(\al),\overline{\mathsf{m}}(\be) \}=1.\]
\end{proposition}

\begin{proof}
It follows from \cite[Theorem 4.15]{Oh15E} and Algorithm \ref{Alg:label}.
\end{proof}

\begin{proposition} \label{prop:gdist_2}
If two positive roots $\alpha$ and $\beta$ are in $\Upsilon_{[\Qg]}^C$ then
\[ \gdist_{[\Qg]}(\alpha,\beta)\leq  \max\{ \overline{\mathsf{m}}(\al),\overline{\mathsf{m}}(\be) \}=2.\]
\end{proposition}

\begin{proof}
Note that
\begin{itemize}
\item every vertex in $\Upsilon^C_{[\Qg]}$ has $\alpha_n$ and $\alpha_{n+1}$ as its support,
\item if $\alpha\prec_{[\Qg]}\gamma \prec_{[\Qg]} \beta$  then $\gamma\in \Upsilon_{[\Qg]}^C.$
\end{itemize}
Since we are only interested in sequences with the same weight as $\alpha+\beta$,
any sequence whose sum in $\al+\be$ should be a pair. Thus it is enough to consider increasing sequences of pairs with weight $\wt(\underline{p})$
\begin{equation} \label{Eqn:pair_seq}
 \underline{p}_1 \prec_{[\Qg]}^\tb  \underline{p}_2  \prec_{[\Qg]}^\tb \cdots \prec_{[\Qg]}^\tb  \underline{p_m}=\underline{p}= (\alpha, \beta).
 \end{equation}

Suppose $\alpha$ and $\beta$ are not in the same sectional. Let $\alpha=\left< a_1, a_2\right>$ and $\beta=\left< b_1, b_2\right>$ for distinct
$a_1, a_2, b_1, b_2\in \{ 1, 2,\cdots, n\}$. A pair $(\alpha', \beta')$ with same weight as $\underline{p}$ has $a_1, a_2, b_1$ and $b_2$ as
summands of $\alpha'$ or $\beta'.$ Hence it is enough to consider intersections of swings containing $\alpha$ or $\beta.$

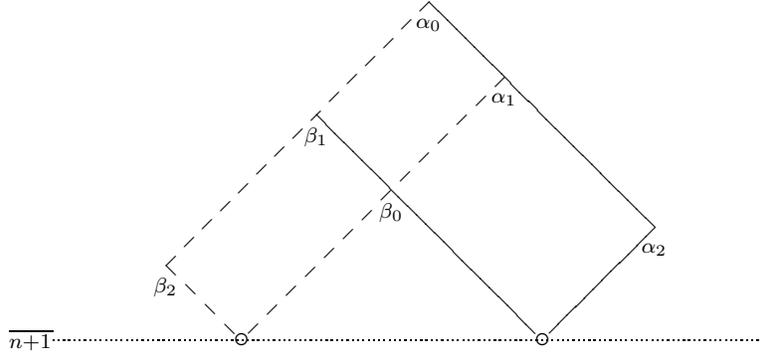
\begin{equation} \label{eq: intersection of swings}
{\xy  (0,0)*{}="T1"; (10,-10)*{}="R1"; (-15,-15)*{}="L1";
(-35,-35)*{}="A1";(-5,-25)*{}="P1";(30,-30)*{}="B1";
(-25,-45)*{\circ}="G1";(15,-45)*{\circ}="G2";
(-53, -45)*{_{\overline{n+1}}};
"A1";"T1" **\dir{--}; "A1";"G1" **\dir{--};  "G1";"R1" **\dir{--};
"T1";"B1" **\dir{-};"B1";"G2" **\dir{-};"L1";"G2" **\dir{-};
"T1"+(0,-3)*{_{\al_0}}; "L1"+(0,-3)*{_{\be_1}};"R1"+(0,-3)*{_{\al_1}};
"P1"+(0,-3)*{_{\be_0}};"A1"+(0,-3)*{_{\be_2}};"B1"+(0,-3)*{_{\al_2}};
(-50, -45);(45, -45)**\dir{.};
\endxy}
\end{equation}

Now, the pair $\underline{p}=(\alpha, \beta)$ should be one of $(\alpha_0, \beta_0)$, $(\alpha_1, \beta_1)$ and $(\alpha_2, \beta_2)$ and we have
\[ \gdist_{[\Qg]}(\al_i,\be_i)=i\]

If $\alpha$ and $\beta$ are in the same sectional path, there is no other pair whose weight is the same as $\alpha+\beta$ by Theorem \ref{thm:p-snake well}
and the definition of $\prec^\tb_{[\Qg]}$.
\end{proof}

\begin{lemma} \label{lem:gdist_3}
Let  $\alpha$ be in $\Upsilon_{[\Qg]}^C$ and $\beta$ be in $ \Upsilon_{[\Qg]}^R$ or  $ \Upsilon_{[\Qg]}^L$. If $\alpha+\beta\not \in \Phi^+$ then there is no sequence $\underline{m}$ such that
\begin{itemize}
\item the weight of $\underline{m}$ is same as $\alpha+\beta$,
\item $\underline{m}\prec_{[\Qg]}^{\tb} (\alpha, \beta)$,
\item $\underline{m}$ is not a pair.
\end{itemize}
\end{lemma}

\begin{proof}
 Consider $\beta\in \Upsilon_{[\Qg]}^R$ and denote $\alpha=\left<a_1, a_2\right>$ and $\beta=\left<b_1, b_2\right>.$
 Note that $b_2\in \{-2, -3, \cdots, -n-1\} \cup \{n+1\}.$   If $|b_2| < a_1$,  the proposition directly follows from the convexity of $\prec_{[\Qg]}$.
 Hence we only consider the case when $|b_2|\geq a_1.$ Also, we note that it is enough to show when  $\underline{m}$ is a triple.

Let us assume that $|b_2|> a_1$.  Suppose there is a triple $\underline{m}= (\gamma, \delta, \eta)$ where $(\alpha, \beta)\succ_{[\Qg]}^\tb \underline{m}$ and
the weight is same as that of $\alpha+\beta$. The set of summands of $\gamma, \delta$ and $\eta$ is $\{ a_1, a_2, b_1, b_2, c, -c\}$ for
$c\in \{2, \cdots, n+1\}\backslash\{a_1, a_2, b_1, b_2\}$. Without loss of generality, let $\gamma$ have $c$ as a summand and let $\delta$ have $-c$ as a summand.

(Case I) If the $N$-sectional path passing $\alpha$ and the $S$-sectional path passing $\beta$ have an intersection then the
$[c]$-snake in $\Upsilon_{[\Qg]}$ looks like one of $\text{$[c]$-snake}^{(1)}$ and $\text{$[c]$-snake}^{(2)}$ in the following picture.

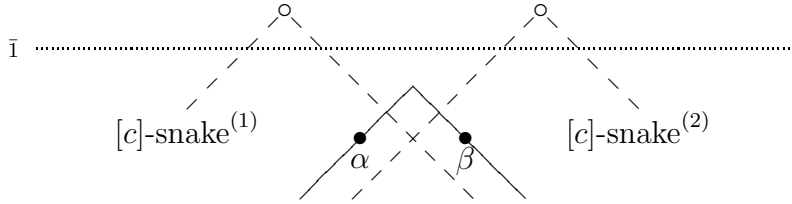
\[
{\xy (-17,0)*{\circ}="T1"; (17,0)*{\circ}="T2"; (-30,-13)*{}="Lc1"; (8,-25)*{}="Rc1"; (30,-13)*{}="Rc2"; (-8,-25)*{}="Lc2";
(-15,-25)*{}="Ns"; (15,-25)*{}="Ss";(-7,-17)*{\bullet}="al"; (7,-17)*{\bullet}="be"; (0, -10)*{}="inter";
"T1"; "Lc1"**\dir{--};"T1"; "Rc1"**\dir{--};"T2"; "Lc2"**\dir{--};"T2"; "Rc2"**\dir{--}; "Ns";"inter"**\dir{-}; "Ss";"inter"**\dir{-};
"al"+(0,-3)*{\alpha}; "be"+(0,-3)*{\beta};"Lc1"+(0,-3)*{\text{$[c]$-snake}^{(1)}};"Rc2"+(0,-3)*{\text{$[c]$-snake}^{(2)}};
(-50, -5)*{}="LE"; (50, -5)*{}="RE"; "LE";"RE"**\dir{.}; "LE"+(-3,0)*{{}_{\bar{1}}};
\endxy}
\]
If the $[c]$-snake is of the form of $\text{$[c]$-snake}^{(1)}$ then there is no root having $c$ as a summand between $\alpha$ and $\beta.$
On the other hand, if the $[c]$-snake is of the form of $\text{$[c]$-snake}^{(2)}$ then there is no root having $-c$ as a summand between $\alpha$
and $\beta.$ Hence, in any case, we cannot find both $\gamma$ and $\delta$.

(Case II) Suppose $N$-sectional path passing $\alpha$ and $S$-sectional path passing $\beta$ do not have an intersection.
This assumption implies $|a_2|\geq |b_2|$, so that there is no root having $a_2$ and $b_2$ as summands. In this case,
\begin{itemize}
\item $\alpha$ is the intersection between $N[a_2]$-sectional path and $S[a_1]$-sectional path.
\item $\beta$ is the intersection between $N[b_1]$-sectional path and $S[b_2]$-sectional path.
\end{itemize}
Also, from the fact that $\beta \prec_{[\Qg]} \gamma, \delta \prec_{[\Qg]} \alpha$, we can see that
\begin{itemize}
\item $N[c]$-sectional path and $S[-c]$-sectional path should be in between $N[a_2]$-sectional path and $S[b_2]$-sectional path.
\item $(\gamma, \delta)=( \left<a_1, c\right>, \left<b_1, -c\right> ) $ or $( \left<a_1, -c\right>, \left<b_1, c\right> )$.
\end{itemize}
However, since $|a_2|\geq |b_2|$, there is no root of the form $\left< a_2, b_2\right>.$ Hence $\alpha+\beta-\gamma-\delta=\eta \not \in \Phi^+,$
which is a contradiction.

For the case $|b_2|=a_1$, we omit the proof since it can be proved similarly. Moreover, the same argument works when $\beta\in \Upsilon_{[\Qg]}^L.$
\end{proof}

\begin{proposition} \label{prop:gdist_3}
Let  $\alpha$ be in $\Upsilon_{[\Qg]}^C$ and $\beta$ be in $ \Upsilon_{[\Qg]}^R$ or  $ \Upsilon_{[\Qg]}^L$. If $\alpha+\beta\not \in \Phi^+$ then
$$ \gdist_{[\Qg]}(\alpha,\beta)\leq  \max\{ \overline{\mathsf{m}}(\al),\overline{\mathsf{m}}(\be) \} =\overline{\mathsf{m}}(\al)=2.$$
\end{proposition}

\begin{proof}
By Lemma \ref{lem:gdist_3}, it is enough to consider pairs $\underline{p}$ whose weight are same as that of $\alpha+\beta.$
Now, we can apply the same argument as in the proof of Proposition \ref{prop:gdist_2} which was described in \eqref{eq: intersection of swings}.
\end{proof}

\begin{proposition} \label{prop:gdist_4}
Let $\alpha$ be in one of $\Upsilon_{[\Qg]}^L$ and $\Upsilon_{[\Qg]}^R$, and $\be$ be in the other. If $\alpha+\beta\not \in \Phi^+$ then
\[ \gdist_{[\Qg]}(\alpha,\beta)\leq \max\{ \overline{\mathsf{m}}(\al),\overline{\mathsf{m}}(\be) \}=1.\]
\end{proposition}

\begin{proof}

Note that $\al ,\be$ are (folded) multiplicity free and hence can be identified with corresponding positive roots in $\Phi^+_{A_n}$ in a natural way.
Without loss of generality, assume that $\al \in \Upsilon_{[\Qg]}^L$ and $\be \in \Upsilon_{[\Qg]}^R$.
Let us denote $\alpha=\left< a_1, -a_2\right>$ and $\beta=\left< b_1, -b_2\right>$. Then the followings are true.
\begin{itemize}
\item If  $b_1-1>a_2$ or $a_1-1>b_2$ then there is no other pair with the same weight and hence $\gdist_{[\Qg]}(\alpha,\beta)=0$ (\cite[(4.2)]{Oh15E}).
\item If $b_1=-a_2$ or $a_1=-b_2$ then $\alpha+\beta \in \Phi^+$.
\item If $|a_2|=|b_2|= n+1$ then $\alpha+\beta \not \in \Phi^+$ implies $a_1=b_1.$ In this case, one can check that $\gdist_{[\Qg]}(\alpha, \beta)=0.$
\end{itemize}

For the other cases, there exists another pair $\{ \ga,\eta \}$ such that $$\al+\be=\ga+\eta \not \in \PR.$$ Note that such a pair $\{ \ga,\eta \}$ is unique.
There is a square in the AR-quiver $\Gamma_Q$ for $Q=\mathfrak{p}_{A_n}^{D_{n+1}}([\Qg])$ which is  one of the followings
(Theorem \ref{thm: labeling GammaQ} and \cite[Proposition 4.12]{Oh14A}):
\[{\xymatrix@C=1.5ex@R=1ex{
 && \eta \ar@{->}[dr]  && \\
 &\alpha\ar@{->}[dr]\ar@{->}[ur]  &&\beta \\
&& \gamma \ar@{->}[ur]
 }}
 \text{ \ or \  }
 {\xymatrix@C=1.5ex@R=1ex{
 && \beta \ar@{->}[dr]  && \\
 &\gamma\ar@{->}[dr]\ar@{->}[ur]  &&\eta \\
&& \alpha \ar@{->}[ur]
 }}.
 \]
Here $\eta=\left<a_1, -b_2\right>$ and $\gamma=\left< b_1, -a_2\right>$ by Theorem \ref{thm: labeling GammaQ} and they are also (folded) multiplicity free.
Recall Remark \ref{rem:LR path} that {\rm (i)} there is a path from any vertex in $\Upsilon_{[\Qg]}^L$ to any vertex in $\Upsilon_{[\Qg]}^R$
and {\rm (ii)} $\eta$ (resp. $\gamma$) is in  $\Upsilon_{[\Qg]}^R$ (resp. $\Upsilon_{[\Qg]}^L$) .
For the first case, we get the path $\alpha\to \gamma\to \eta\to \beta$ by Algorithm \ref{Alg:label} and the assumption of our assertion.
For the second case, we get the path $\gamma\to \alpha\to \beta\to \eta$ by the same reason. Hence we get $\gdist_{[\Qg]}(\alpha,\beta)=1$ and
$\gdist_{[\Qg]}(\alpha,\beta)=0$, respectively. For a non-pair sequence $\um$ such that $\wt(\um)=\al+\be$, it cannot be $ \um \prec_{[\Qg]} (\al,\be)$
by the convexity of $\prec_{[\Qg]}$. The remained cases can be proved by applying the similar argument.
\end{proof}

\begin{proof} [{\bf The second step for  Theorem \ref{thm: dist upper bound Qd}}]
By Proposition \ref{Lem:gdist_1}, \ref{prop:gdist_2}, \ref{prop:gdist_3} and
\ref{prop:gdist_4}, we complete Theorem \ref{thm: dist upper bound Qd}.
\end{proof}

\begin{remark}
Note that, for a pair $(\al,\be)$ with $\gdist_{[\Qg]}(\al,\be)=2$, there exists a unique  sequence $\up'$, which is a pair, with $\wt(\up')=\al+\be$ such that
$$ \soc_{[\Qg]}(\al,\be) \prec^\tb_{[\Qg]} \up' \prec^\tb_{[\Qg]} (\al,\be).$$
 However, in an adapted class $[Q]$ of type $D$, every pair $(\al,\be)$ with $\gdist_{[Q]}(\al,\be)=2$, there exist sequences $\um^{(1)}$ and $\um^{(2)}$ such that
\begin{itemize}
\item $\wt(\um^{(1)})=\wt(\um^{(2)})=\al+\be$,
\item they are incomparable with respect to $ \prec^\tb_{[Q]}$,
\item $ \soc_{[Q]}(\al,\be) \prec^\tb_{[\Qg]} \begin{matrix} \um^{(1)} \\ \um^{(2)} \end{matrix} \prec^\tb_{[Q]} (\al,\be)$.
\end{itemize}
\end{remark}

\begin{theorem} \label{Thm:gdist2}
Let $\alpha$ and $\beta$ have coordinates $(\bar{i}, p)$ and $(\bar{j}, q)$ in the folded AR-quiver $\widehat{\Upsilon}_{[\Qg]}.$ Then the
$\gdist_{[\Qg]}(\alpha,\beta)$ depends only on  $\bar{i}$, $\bar{j}$ and $|p-q|.$
\end{theorem}

\begin{proof}
Let us assume that $\alpha\succ_{[\Qg]}\beta$. Then $q>p.$ Depending on whether there exists an intersection of the $S$-sectional path passing
$\alpha$ and the $N$-sectional path passing $\beta$ in $\widehat{\Upsilon}_{[\Qg]}$, we classify the type of the pair $(\alpha, \beta)$:
\begin{itemize}
\item  $(\alpha,\beta)$ is {\it Type I} if there exists no intersection between the $S$-sectional path passing $\alpha$ and the $N$-sectional path passing $\beta$.
\item $(\alpha,\beta)$ is {\it Type II} if there exists an intersection between the two sectional paths.
\end{itemize}

(Type I) the picture in \eqref{Pic:8.6} is induced from four swings containing $\alpha$ and $\beta.$ Note that $\mu$, $\nu$, $\gamma$, and $\eta$ are not necessarily in $\widehat{\Upsilon}_{[\Qg]}$ but they are on the rays containing corresponding sectional paths.
In order to give coordinates to $\mu$, $\nu$, $\gamma$, $\eta$ which are not in $\widehat{\Upsilon}_{[\Qg]}$, we extend the coordinate in a ``canonical" way. The new coordinate in $\{\, (k, r)\in \Z\times \Z \, |\,   k\leq n+1 \}$ has the following properties:
\begin{enumerate}[(i)]
\item For $k' \in \{1, 2, \cdots, n\}$, if a vertex has the coordinate $(\bar{k}',r')$ in $\widehat{\Upsilon}_{[\Qg]}$ then, in the new coordinate, it has the new coordinate $(k', r').$
\item For $k' \in \{n+1\}\cup \{0, -1, ,-2, \cdots \}$, the coordinate $(k', r')$ is the intersection of two rays containing the following sectional paths :
\begin{enumerate}[(a)]
\item $N$-sectional path consisting of $(\overline{k'+l}, r'+l) \cap \widehat{\Upsilon}_{[\Qg]}$ for $l\in \Z$,
\item $S$-sectional path consisting of $(\overline{k'-l}, r'+l) \cap \widehat{\Upsilon}_{[\Qg]}$ for $l\in \Z$.
\end{enumerate}
\end{enumerate}

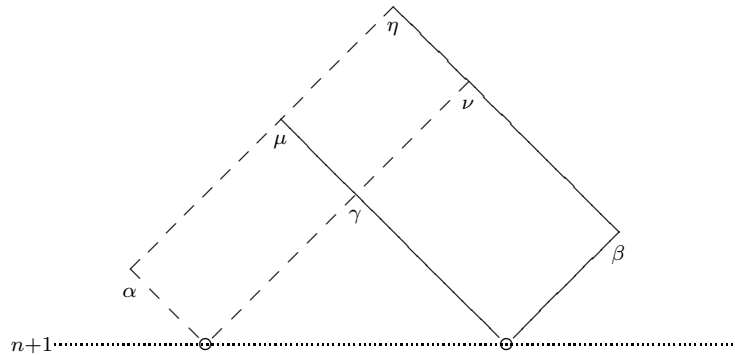
\begin{equation} \label{Pic:8.6}
{\xy  (0,0)*{}="T1"; (10,-10)*{}="R1"; (-15,-15)*{}="L1";
(-35,-35)*{}="A1";(-5,-25)*{}="P1";(30,-30)*{}="B1";
(-25,-45)*{\circ}="G1";(15,-45)*{\circ}="G2";
(-45,-45)*{}="GL1";(45,-45)*{}="GL2";"GL1";"GL2"**\dir{.};
(-48, -45)*{_{n+1}};
"A1";"T1" **\dir{--}; "A1";"G1" **\dir{--};  "G1";"R1" **\dir{--};
"T1";"B1" **\dir{-};"B1";"G2" **\dir{-};"L1";"G2" **\dir{-};
"T1"+(0,-3)*{_{\eta}}; "L1"+(0,-3)*{_{\mu}};"R1"+(0,-3)*{_{\nu}};
"P1"+(0,-3)*{_{\gamma}};"A1"+(0,-3)*{_{\alpha}};"B1"+(0,-3)*{_{\beta}};
\endxy}
\end{equation}
More precisely, we can compute the coordinates as follows.
\begin{equation}\label{corr_inter}
\begin{aligned}
& \eta=({\eta}_0, \eta_1)=( (i+j+p-q)/2, i-j+p+q),\\
& \mu=({\mu}_0, \mu_1)=(n+1+(i-j+p-q)/2, \eta_1+j-n-1), \\
& \nu=({\nu}_0, \nu_1)=(n+1+(j-i+p-q)/2, \eta_1+i-n-1), \\
 & \gamma=({\gamma}_0, \gamma_1)=(2n+2-(i+j-p+q)/2, \mu_1+\nu_1-\eta_1).
\end{aligned}
\end{equation}
Now we note that $\eta_0$, $\mu_0$, $\nu_0$, $\gamma_0$ depend only on
$i$, $j$ and $|p-q|.$ Moreover, by  Proposition \ref{prop:gdist_2}, Proposition \ref{prop:gdist_3} and Proposition \ref{prop:gdist_4}, we have
\begin{equation} \label{TypeA_gdist}
\begin{aligned}
& \gdist_{[\Qg]}(\alpha,\beta)=2 \text{ \ if \ } \eta_0\geq 0 \text{ and } \gamma_0<n+1; \\
& \gdist_{[\Qg]}(\alpha,\beta)=1 \text{ \ if \ } {\rm (i)}\  (\eta_0, \gamma_0) \not \in (\Z_{\geq 0}, \Z_{<n+1}), \  {\rm (ii)} \, (\, \mu_0, \nu_0\, ) \in \Z_{>0}\times \Z_{\geq 0} \cup \Z_{\geq 0} \times \Z_{>0}; \\
& \gdist_{[\Qg]}(\alpha, \beta)=0 \text{ \ otherwise}.
\end{aligned}
\end{equation}

\vskip 2mm

(Type II) In this case, we have the following picture. As in (Type I), $\eta$ is not necessarily in  $\widehat{\Upsilon}_{[\Qg]}.$ However, $\gamma'$ is in $\widehat{\Upsilon}_{[\Qg]}.$
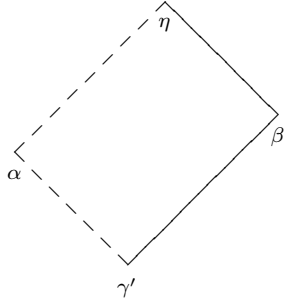
\begin{equation}\label{Pic:2}
{\xy (0,0)*{}="T2"; (15,-15)*{}="R2"; (-20,-20)*{}="L2" ;(-5,-35)*{}="B2";
"L2";"T2" **\dir{--}; "L2";"B2" **\dir{--};
"T2";"R2" **\dir{-};"B2";"R2" **\dir{-};
"T2"+(0,-3)*{_{\eta}}; "L2"+(0,-3)*{_{\alpha}};"R2"+(0,-3)*{_{\beta}};
"B2"+(0,-3)*{_{\gamma'}};
\endxy}
\end{equation}
Here, $\eta$ has the same coordinate as in \eqref{corr_inter} and
\begin{equation}
\gamma'=({\gamma'}_0, \gamma'_1)=((i+j-p+q)/2,p+q-\eta_1)
\end{equation}

\begin{equation} \label{TypeB_gdist}
\begin{aligned}
& \gdist_{[\Qg]}(\alpha, \beta)=1 \text{ \ if \ } \eta_0\geq 0; \\
& \gdist_{[\Qg]}(\alpha, \beta)=0 \text{ \ otherwise\ }.
\end{aligned}
\end{equation}

Since $\gamma_0$, $\eta_0$, $\mu_0$, $\nu_0$ and $\gamma'_0$ depend only on $i, j$ and $|p-q|$ and by  \eqref{TypeA_gdist} and \eqref{TypeB_gdist}. Moreover, the assumption $1\leq i,j\leq n$ implies $i=\bar{i}$ and $j=\bar{j}$. Hence we conclude that
$\gdist_{[\Qg]}(\alpha,\beta)$ only depends on $\bar{i},\bar{j}$ and $|p-q|.$
\end{proof}

\begin{remark} \label{rem: subquivers}
Let $\alpha,\beta$ be positive roots of $D_{n+1}$ and the  coordinates $(\bar{i},p)$ and $(\bar{j},q)$ of $\alpha$ and $\beta$ in $\widehat{\Upsilon}_{[\Qg]}$  satisfy $i\geq j$ and $q>p$.

If  $(\alpha,\beta)$ is of (Type I) in the proof of Theorem \ref{Thm:gdist2} then the pair $(\alpha,\beta)$ has the form of $\rm{(I-1)}\sim \rm{(I-8)}$ in
\eqref{Pic:gdist_D_1} or \eqref{Pic:gdist_D_1-2}. Note that the layer $n+1$ in \eqref{Pic:gdist_D_1} and \eqref{Pic:gdist_D_1-2} is induced from the new coordinate introduced in the proof of Theorem \ref{Thm:gdist2}.

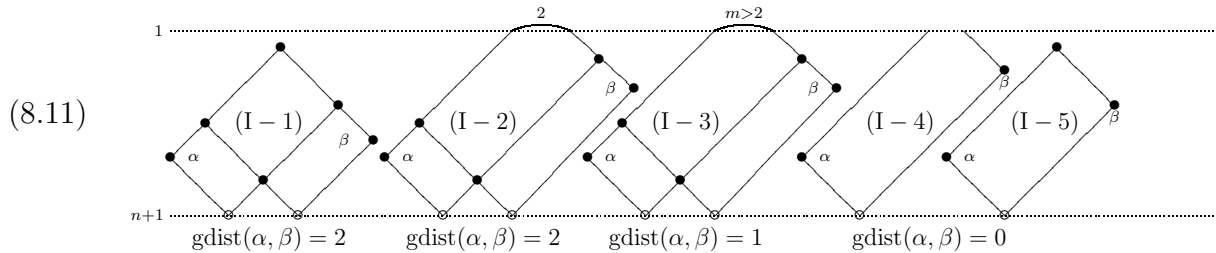
\begin{align} \label{Pic:gdist_D_1}
\scalebox{0.77}{{\xy
(-20,0)*{}="DL";(-10,-10)*{}="DD";(0,20)*{}="DT";(10,10)*{}="DR";
"DT"+(-30,-4)-(15,0); "DT"+(135,-4)**\dir{.};
"DD"+(-20,-6)-(15,0); "DD"+(145,-6) **\dir{.};
"DT"+(-32,-4)-(15,0)*{\scriptstyle {1}};
"DT"+(-34,-36)-(15,0)*{\scriptstyle {n+1}};
"DD"+(0,4)+(37,0); "DD"+(10,-6)+(37,0) **\dir{-};
"DD"+(16,0)+(37,0); "DD"+(10,-6)+(37,0) **\dir{-};
"DD"+(16,0)+(37,0); "DD"+(22,-6)+(37,0) **\dir{-};
"DD"+(43,16)+(37,0); "DD"+(22,-6)+(37,0) **\dir{-};
"DD"+(0,4)+(37,0); "DD"+(22,26)+(37,0) **\dir{-};
"DD"+(43,16)+(37,0); "DD"+(32,26)+(37,0) **\dir{-};
"DD"+(16,0)+(37,0); "DD"+(37,21)+(37,0) **\dir{-};
"DD"+(16,0)+(37,0); "DD"+(6,10)+(37,0) **\dir{-};
"DD"+(16,0)+(37,0)*{\bullet};
"DD"+(37,21)+(37,0)*{\bullet};
"DD"+(6,10)+(37,0)*{\bullet};
"DD"+(10,-6)+(37,0)*{\circ};
"DD"+(22,-6)+(37,0)*{\circ};
"DD"+(22,26)+(37,0); "DD"+(32,26)+(37,0) **\crv{"DD"+(27,28)+(37,0)};
"DD"+(27,29)+(37,0)*{\scriptstyle m>2};
"DD"+(0,4)+(37,0)*{\bullet};
"DL"+(27,0)+(37,0)*{{\rm(I-3)}};
"DD"+(4,4)+(37,0)*{\scriptstyle \al};
"DD"+(43,16)+(37,0)*{\bullet};
"DD"+(39,16)+(37,0)*{\scriptstyle \be};
"DL"+(27,0)+(37,-20)*{\gdist(\alpha,\beta)=1};
"DD"+(0,4)+(2,0); "DD"+(10,-6)+(2,0) **\dir{-};
"DD"+(16,0)+(2,0); "DD"+(10,-6)+(2,0) **\dir{-};
"DD"+(16,0)+(2,0); "DD"+(22,-6)+(2,0) **\dir{-};
"DD"+(43,16)+(2,0); "DD"+(22,-6)+(2,0) **\dir{-};
"DD"+(0,4)+(2,0); "DD"+(22,26)+(2,0) **\dir{-};
"DD"+(43,16)+(2,0); "DD"+(32,26)+(2,0) **\dir{-};
"DD"+(16,0)+(2,0); "DD"+(37,21)+(2,0) **\dir{-};
"DD"+(16,0)+(2,0); "DD"+(6,10)+(2,0) **\dir{-};
"DD"+(37,21)+(2,0)*{\bullet};
"DD"+(6,10)+(2,0)*{\bullet};
"DD"+(16,0)+(2,0)*{\bullet};
"DD"+(10,-6)+(2,0)*{\circ};
"DD"+(22,-6)+(2,0)*{\circ};
"DD"+(22,26)+(2,0); "DD"+(32,26)+(2,0) **\crv{"DD"+(27,28)+(2,0)};
"DD"+(27,29)+(2,0)*{\scriptstyle 2};
"DD"+(0,4)+(2,0)*{\bullet};
"DL"+(27,0)+(2,0)*{{\rm(I-2)}};
"DD"+(4,4)+(2,0)*{\scriptstyle \al};
"DD"+(43,16)+(2,0)*{\bullet};
"DD"+(39,16)+(2,0)*{\scriptstyle \be};
"DL"+(27,0)+(2,-20)*{\gdist(\alpha,\beta)=2};
"DD"+(35,4)-(70,0); "DD"+(45,-6)-(70,0) **\dir{-};
"DD"+(51,0)-(70,0); "DD"+(45,-6)-(70,0) **\dir{-};
"DD"+(51,0)-(70,0); "DD"+(57,-6)-(70,0) **\dir{-};
"DD"+(70,7)-(70,0); "DD"+(57,-6) -(70,0)**\dir{-};
"DD"+(35,4)-(70,0); "DD"+(54,23)-(70,0) **\dir{-};
"DD"+(70,7)-(70,0); "DD"+(54,23)-(70,0) **\dir{-};
"DD"+(51,0)-(70,0); "DD"+(64,13) -(70,0)**\dir{-};
"DD"+(51,0)-(70,0); "DD"+(41,10)-(70,0) **\dir{-};
"DD"+(41,10)-(70,0)*{\bullet};
"DD"+(64,13) -(70,0)*{\bullet};
"DD"+(35,4)-(70,0)*{\bullet};
"DD"+(54,23)-(70,0)*{\bullet};
"DD"+(45,-6)-(70,0)*{\circ};
"DD"+(57,-6)-(70,0)*{\circ};
"DD"+(51,0)-(70,0)*{\bullet};
"DL"+(62,0)-(70,0)*{{\rm(I-1)}};
"DD"+(39,4)-(70,0)*{\scriptstyle \al};
"DD"+(70,7)-(70,0)*{\bullet};
"DD"+(65,7)-(70,0)*{\scriptstyle \be};
"DL"+(62,0)-(70,20)*{\gdist(\alpha,\beta)=2};
"DD"+(74,4); "DD"+(84,-6) **\dir{-};
"DD"+(74,4); "DD"+(96,26) **\dir{-};
"DD"+(109,19); "DD"+(102,26) **\dir{-};
"DD"+(84,-6); "DD"+(109,19) **\dir{-};
"DD"+(74,4)*{\bullet};
"DL"+(101,0)*{{\rm(I-4)}};
"DD"+(78,4)*{\scriptstyle \al};
"DD"+(109,19)*{\bullet};
"DD"+(109,17)*{\scriptstyle \be};
"DD"+(84,-6)*{\circ};
"DD"+(99,4); "DD"+(109,-6) **\dir{-};
"DD"+(99,4); "DD"+(118,23) **\dir{-};
"DD"+(128,13); "DD"+(118,23) **\dir{-};
"DD"+(109,-6); "DD"+(128,13) **\dir{-};
"DD"+(99,4)*{\bullet};
"DD"+(118,23)*{\bullet};
"DL"+(126,0)*{{\rm(I-5)}};
"DD"+(103,4)*{\scriptstyle \al};
"DD"+(128,13)*{\bullet};
"DD"+(128,11)*{\scriptstyle \be};
"DD"+(109,-6)*{\circ};
"DD"+(96,-10)*{\gdist(\alpha,\beta)=0};
\endxy}}
\end{align}

\begin{align} \label{Pic:gdist_D_1-2}
\scalebox{0.77}
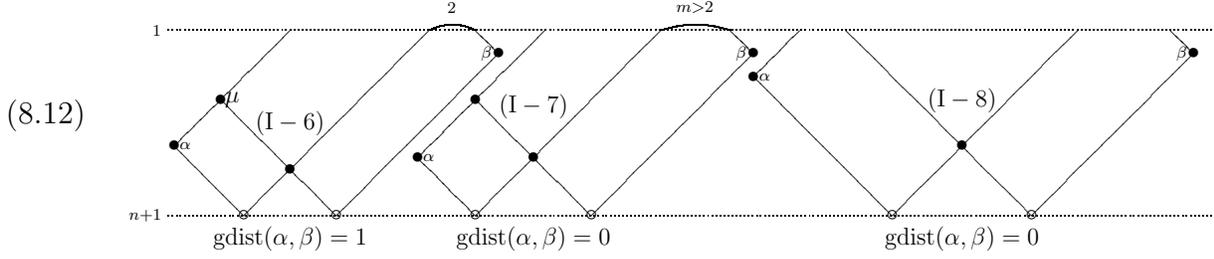
{{\xy
(-20,0)*{}="DL";(-10,-10)*{}="DD";(0,20)*{}="DT";(10,10)*{}="DR";
"DT"+(-30,-4)-(15,0); "DT"+(135,-4)**\dir{.};
"DD"+(-20,-6)-(15,0); "DD"+(145,-6) **\dir{.};
"DT"+(-32,-4)-(15,0)*{\scriptstyle {1}};
"DT"+(-34,-36)-(15,0)*{\scriptstyle {n+1}};
(-44,-4)*{\bullet}; (-42,-4)*{_{\alpha}}; (-32,-16)*{\circ}; (-16,-16)*{\circ};(12,12)*{\bullet};(-36,4)*{\bullet};(-34,4)*{\mu};(10,12)*{_{\beta}};
(-44,-4);(-36,4)**\dir{-};
(-44,-4);(-32, -16)**\dir{-};
(-36,4);(-16,-16)**\dir{-};
(-32,-16);(0,16)**\dir{-};
(8,16);(12,12)**\dir{-};
(-16,-16);(12,12)**\dir{-};
(0,16); (8,16)**\crv{(4,18)}; (4,20)*{_{2}};
(-24,16);(-36,4)**\dir{-};
(-24,-8)*{\bullet};
(-24,0)*{\rm{(I-6)}};
(-24,-20)*{\gdist(\alpha,\beta)=1};
(-2,-6)*{\bullet}; (0,-6)*{_{\alpha}}; (8,-16)*{\circ}; (28,-16)*{\circ};(56,12)*{\bullet};(8,4)*{\bullet};(54,12)*{_{\beta}};
(-2,-6);(8,4)**\dir{-};
(-2,-6);(8, -16)**\dir{-};
(8,4);(28,-16)**\dir{-};
(8,-16);(40,16)**\dir{-};
(52,16);(56,12)**\dir{-};
(28,-16);(56,12)**\dir{-};
(40,16); (52,16)**\crv{(46,18)}; (46,20)*{_{m>2}};
(20,16);(8,4)**\dir{-};
(18,-6)*{\bullet};
(18,3)*{\rm{(I-7)}};
(18,-20)*{\gdist(\alpha,\beta)=0};
(56,8)*{\bullet};(58,8)*{_{\alpha}}; (80, -16)*{\circ};(104,-16)*{\circ};(132,12)*{\bullet};(130,12)*{_{\beta}};
(56,8);(64,16)**\dir{-};
(56,8);(80,-16)**\dir{-};
(72,16);(104,-16)**\dir{-};
(80,-16);(112,16)**\dir{-};
(128,16);(132,12)**\dir{-};
(104,-16);(132,12)**\dir{-};
(92,-4)*{\bullet};
(92,4)*{\rm{(I-8)}};
(92,-20)*{\gdist(\alpha,\beta)=0};
\endxy}}
\end{align}

Now, let $(\alpha,\beta)$  be of (Type II) in the proof of Theorem \ref{Thm:gdist2}.
Then we have one of the following pictures. If we have (II-1) or (II-2) then $\gdist_{[\Qg]}(\alpha,\beta)=1$.
If we have (II-3) then $\gdist_{[\Qg]}(\alpha,\beta)=0$

\begin{align} \label{eq: complacted socle of D}
\scalebox{0.77}{{\xy
(-20,0)*{}="DL";(-10,-10)*{}="DD";(0,20)*{}="DT";(10,10)*{}="DR";
"DT"+(-40,-4); "DT"+(95,-4)**\dir{.};
"DD"+(-30,-10); "DD"+(105,-10) **\dir{.};
"DT"+(-42,-4)*{\scriptstyle {1}};
"DT"+(-44,-40)*{\scriptstyle {n+1}};
"DL"+(-10,-3); "DD"+(-10,-3) **\dir{-};
"DR"+(-14,-7); "DD"+(-10,-3) **\dir{-};
"DT"+(-14,-7); "DR"+(-14,-7) **\dir{-};
"DT"+(-14,-7); "DL"+(-10,-3) **\dir{-};
"DL"+(-6,-3)*{\scriptstyle \al};
"DR"+(-18,-7)*{\scriptstyle \be};
"DL"+(2,0)*{{\rm (II-1)}};
"DL"+(-10,-3)*{\bullet};"DR"+(-14,-7)*{\bullet};
"DT"+(-14,-7)*{\bullet}; "DD"+(-10,-3)*{\bullet};
"DT"+(-14,-5)*{\scriptstyle \eta};
"DD"+(-10,-5)*{\scriptstyle \gamma'};
"DD"+(-10,-13)*{\gdist(\alpha,\beta)=1};
"DD"+(69,4)-(50,0); "DD"+(79,-6)-(50,0) **\dir{-};
"DD"+(85,0)-(50,0); "DD"+(79,-6)-(50,0) **\dir{-};
"DD"+(69,4)-(50,0); "DD"+(91,26)-(50,0) **\dir{-};
"DD"+(104,19)-(50,0); "DD"+(97,26)-(50,0) **\dir{-};
"DD"+(85,0)-(50,0); "DD"+(104,19)-(50,0) **\dir{-};
"DD"+(69,4)-(50,0)*{\bullet};
"DL"+(96,0)-(52,0)*{{\rm (II-2)}};
"DD"+(73,4)-(51,0)*{\scriptstyle \al};
"DD"+(91,26)-(50,0); "DD"+(97,26)-(50,0) **\crv{"DD"+(94,28)-(50,0)};
"DD"+(104,19)-(50,0)*{\bullet};
"DD"+(104,17)-(50,0)*{\scriptstyle \be};
"DD"+(79,-6)-(50,0)*{\bullet};
"DD"+(79,-8)-(50,0)*{\scriptstyle \gamma'};
"DD"+(94,29)-(50,0)*{\scriptstyle 2};
"DD"+(69,4); "DD"+(79,-6) **\dir{-};
"DD"+(85,0); "DD"+(79,-6) **\dir{-};
"DD"+(79,-8)-(50,5)*{\gdist(\alpha,\beta)=1};
"DD"+(65,8); "DD"+(79,-6) **\dir{-};
"DD"+(65,8); "DD"+(83,26) **\dir{-};
"DD"+(104,19); "DD"+(97,26) **\dir{-};
"DD"+(85,0); "DD"+(104,19) **\dir{-};
"DD"+(65,8)*{\bullet};
%"DD"+(74,23)*{\bullet};
"DL"+(92,0)*{{\rm (II-3)}};
"DD"+(71,8)*{\scriptstyle \al};
"DD"+(83,26); "DD"+(97,26) **\crv{"DD"+(90,28)};
"DD"+(104,19)*{\bullet};
"DD"+(104,17)*{\scriptstyle \be};
"DD"+(79,-6)*{\bullet};
"DD"+(77,-8)*{\scriptstyle \gamma'};
"DD"+(94,29)*{\scriptstyle m>2};
"DD"+(77,-13)*{\gdist(\alpha,\beta)=0};
\endxy}}
\end{align}

Hence Theorem \ref{Thm:gdist2} shows
\begin{equation}
\gdist_{[\Qg]}(\alpha,\beta)=
\left\{\begin{array}{ll}
2 & \text{ if $(\alpha,\beta)$ is one of $\rm{(I-1)}$ or $\rm{(I-2)}$, }\\
1 & \text{ if $(\alpha,\beta)$ is one of $\rm{(I-3)}$, $\rm{(I-6)}$, $\rm{(II-1)}$ or $\rm{(II-2)},$ }\\
0 & \text{ otherwise. }
\end{array}
\right.
\end{equation}
\end{remark}

\begin{remark} \label{rem: minimal pair pic}
Observe $\rm{(I-6)}$ and $\rm{(II-2)}$ in Remark \ref{rem: subquivers}. In these cases, we have $\alpha+\beta= \mu\in \Phi^+$ and $\alpha+\beta= \gamma'\in \Phi^+$. More precisely,  Theorem \ref{Thm:gdist2} shows if $\alpha+\beta\in \Phi^+$ and $\gdist_{[\Qg]}(\alpha,\beta)=1$ for $\alpha\succ_{[\Qg]}\beta$ then $(\alpha,\beta)$ has one of the following pictures.
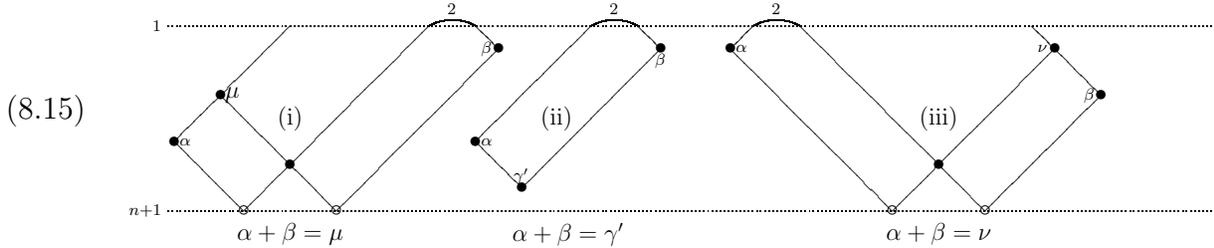
\begin{align} \label{Pic:gdist_D_al+be in PR}
\scalebox{0.77}{{\xy
(-20,0)*{}="DL";(-10,-10)*{}="DD";(0,20)*{}="DT";(10,10)*{}="DR";
"DT"+(-30,-4)-(15,0); "DT"+(135,-4)**\dir{.};
"DD"+(-20,-6)-(15,0); "DD"+(145,-6) **\dir{.};
"DT"+(-32,-4)-(15,0)*{\scriptstyle {1}};
"DT"+(-34,-36)-(15,0)*{\scriptstyle {n+1}};
(-44,-4)*{\bullet}; (-42,-4)*{_{\alpha}}; (-32,-16)*{\circ}; (-16,-16)*{\circ};(12,12)*{\bullet};(-36,4)*{\bullet};(-34,4)*{\mu};(10,12)*{_{\beta}};
(-44,-4);(-36,4)**\dir{-};
(-44,-4);(-32, -16)**\dir{-};
(-36,4);(-16,-16)**\dir{-};
(-32,-16);(0,16)**\dir{-};
(8,16);(12,12)**\dir{-};
(-16,-16);(12,12)**\dir{-};
(0,16); (8,16)**\crv{(4,18)}; (4,19)*{_{2}};
(-24,16);(-36,4)**\dir{-};
(-24,-8)*{\bullet};
(-24,-20)*{\alpha+\beta=\mu};
(-24,0)*{{\rm (i)}};
%(-24,0)*{\rm{(I-6)}};
%
(8,-4)*{\bullet};(16,-12)*{\bullet};(40,12)*{\bullet};(10,-4)*{_{\alpha}};(16,-10)*{_{\gamma'}};(40,10)*{_{\beta}};
(8,-4);(16,-12)**\dir{-};(8,-4);(28,16)**\dir{-};(36,16);(40,12)**\dir{-};(40,12);(16,-12)**\dir{-};
(28,16); (36,16)**\crv{(32,18)}; (32,19)*{_{2}};
(24,-20)*{\alpha+\beta=\gamma'};
(22,0)*{{\rm (ii)}};
(52,12)*{\bullet};(54,12)*{_{\alpha}}; (80,-16)*{\circ}; (88,-8)*{\bullet};(96,-16)*{\circ}; (108,12)*{\bullet};(106,12)*{_{\nu}};(116,4)*{\bullet};(114,4)*{_{\beta}};
(52,12);(56,16)**\dir{-};(64,16);(96,-16)**\dir{-};(52,12);(80,-16)**\dir{-};(80,-16);(108,12)**\dir{-};(108,12);(116,4)**\dir{-};(108,12);(104,16)**\dir{-};(116,4);(96,-16)**\dir{-};
(56,16); (64,16)**\crv{(60,18)}; (60,19)*{_{2}};
(88,-20)*{\alpha+\beta=\nu};
(88,0)*{{\rm (iii)}};
\endxy}}
\end{align}
The picture (i) and (iii) are of (Type I) in Theorem \ref{Thm:gdist2}. Using the notations in the proof of Theorem \ref{Thm:gdist2}, the picture (i) (resp. (iii)) implies $\eta_0<0$, $\mu_0>0$  (resp. $\nu_0>0$) and $\nu_0=0$ (resp. $\mu_0=0$). The picture (ii) is of (Type II) in Theorem \ref{Thm:gdist2} and it implies $\eta_0=0.$
\end{remark}

\begin{remark}
 Let  $\alpha+\beta \in \PR\setminus\Pi$. By Remark \ref{rem: subquivers} and Remark \ref{rem: minimal pair pic} we can find $\gdist_{[\Qg]}(\alpha,\beta)$ using
 \eqref{Pic:8.6} and \eqref{Pic:2} by the following facts.
\begin{enumerate}
\item $\gdist_{[\Qg]}(\alpha,\beta)=2$ if and only if $\gamma, \eta, \mu,\nu$ in  \eqref{Pic:8.6} are all in $\widehat{\Upsilon}_{[\Qg]}$.
\item  $\gdist_{[\Qg]}(\alpha,\beta)=1$ if and only if one of the followings hold:
\begin{enumerate}
\item not both of $\gamma$ and $\eta$ in \eqref{Pic:8.6}  are in $\widehat{\Upsilon}_{[\Qg]}$ but both of $\mu$ and $\nu$ are in $\widehat{\Upsilon}_{[\Qg]},$
\item both $\gamma'$ and $\eta$ in  \eqref{Pic:2}  are in $\widehat{\Upsilon}_{[\Qg]}.$
\end{enumerate}
\end{enumerate}
\end{remark}

\begin{corollary} \label{cor: do not depend}
Suppose $[{Q_1^{\gets}}]$ and $[{Q_2^{\gets}}]$ are both twisted adapted classes. Let $\alpha$ and $\beta\in\Phi^+$ have folded coordinates $(\bar{i}, p)$ and $(\bar{j}, q)$,
respectively, in $\widehat{\Upsilon}_{[{Q_1^{\gets}}]}$ and let $\alpha'$ and $\beta'\in\Phi^+$ have folded coordinates $(\bar{i}, p')$ and $(\bar{j}, q')$,
respectively, in $\widehat{\Upsilon}_{[{Q_2^{\gets}}]}$. If $p-q=p'-q'$ then $\gdist_{[{Q_1^{\gets}}]}(\alpha,\beta)= \gdist_{[{Q_2^{\gets}}]}(\alpha',\beta')$.
\end{corollary}

\begin{proof}
Since \eqref{TypeA_gdist} and \eqref{TypeB_gdist} do not depend on the classes of reduced expressions,
our assertion follows.
\end{proof}

\begin{corollary}
For $\al,\be,\ga \in \PR$ with $\widehat{\phi}_{[\Qg]}(\al)=(\overline{i},p)$, $\widehat{\phi}_{[\Qg]}(\be)=(\overline{j},q)$,
$\widehat{\phi}_{[\Qg]}(\ga)=(\overline{k},r)$ such that  $\al+\be =\ga$, the pair $(\al,\be)$ is a $[\Qg]$-minimal pair of $\ga$
if and only if one of the following conditions holds:
\begin{eqnarray}&&
\left\{\hspace{1ex}\parbox{75ex}{
$\ell \seteq \max(\ov{i},\ov{j},\ov{k}) \le n$, $s+m =\ell$ for $\{ s,m \} \seteq \{ \ov{i},\ov{j},\ov{k}\} \setminus \{ \ell \}$
and
$$ \left( q-r,p-r \right) =
\begin{cases}
\big( -\ov{i},\ov{j} \big), & \text{ if } \ell = \ov{k},\\
\big( \ov{i}-(2n+2),\ov{j} \big), & \text{ if } \ell = \ov{i},\\
\big( -\ov{i},2n+2-\ov{j}  \big), & \text{ if } \ell = \ov{j}.
\end{cases}
$$
}\right. \label{eq: Dorey folded coordinate Cn}
\end{eqnarray}
\end{corollary}

\begin{proof}
Our assertion follows from the consideration on folded coordinates of \eqref{Pic:gdist_D_al+be in PR} in Remark \ref{rem: minimal pair pic}.
\end{proof}

\begin{example}
Let us consider the reduced expression $\ii_0\in [\Qg]$ of $D_5$ with the twisted Coxeter element $(s_5 s_3 s_2 s_1)\vee.$ Then $\widehat{\Upsilon}_{[\Qg]}$ is as follows.
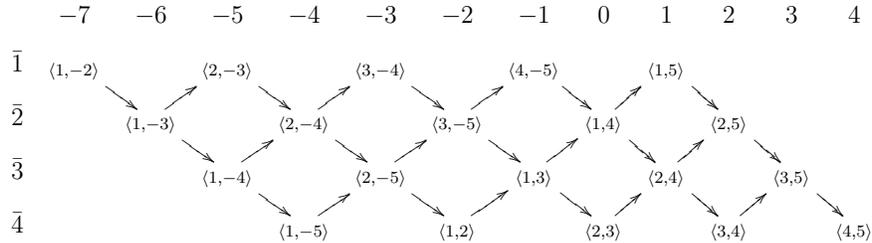
\[ \scalebox{0.8}{\xymatrix@C=0.9ex@R=1ex{
& -7 &   -6 & -5 & -4 & -3 & -2 & -1 & 0& 1& 2& 3& 4\\
\bar{1}& \srt{1}{-2}\ar@{->}[dr]  && \srt{2}{-3}\ar@{->}[dr]  && \srt{3}{-4}\ar@{->}[dr]  && \srt{4}{-5}\ar@{->}[dr] && \srt{1}{5}\ar@{->}[dr]\\
\bar{2}&&\srt{1}{-3}\ar@{->}[dr]\ar@{->}[ur]  && \srt{2}{-4}\ar@{->}[dr]\ar@{->}[ur]  && \srt{3}{-5}\ar@{->}[dr]\ar@{->}[ur]  && \srt{1}{4}\ar@{->}[dr]\ar@{->}[ur] && \srt{2}{5}\ar@{->}[dr]\\
\bar{3} & &&\srt{1}{-4}\ar@{->}[dr]\ar@{->}[ur]  && \srt{2}{-5}\ar@{->}[dr]\ar@{->}[ur]  && \srt{1}{3}\ar@{->}[dr]\ar@{->}[ur]  && \srt{2}{4}\ar@{->}[dr]\ar@{->}[ur] && \srt{3}{5}\ar@{->}[dr] \\
\bar{4} & &&&\srt{1}{-5}\ar@{->}[ur] && \srt{1}{2}\ar@{->}[ur]  &&\srt{2}{3}\ar@{->}[ur]&& \srt{3}{4}\ar@{->}[ur] && \srt{4}{5}
}} \]
Let us use notations in the proof of Theorem \ref{Thm:gdist2}.
\begin{enumerate}
\item Let $\alpha=\left< 1, -4 \right>$ and $\beta= \left< 2, 3 \right>$. Then $(\alpha,\beta)$ is Type $I$
and we have $\mu= \left< 2, -4 \right>$, $\nu=\left< 1, 3\right>$, $\delta=\left< 3, -4 \right>$, $\gamma=\left< 1, 2 \right>.$
Hence $\gdist_{[\Qg]}(\alpha,\beta)=2$ since $(\alpha,\beta)\succ^\tb_{[\Qg]}(\mu,\nu)\succ^\tb_{[\Qg]}(\delta,\gamma).$

\item  Let $\alpha=\left< 2, -5 \right>$ and $\beta= \left< 3, 5 \right>$. Then $(\alpha,\beta)$ is Type $I$ and we have
$\mu= \left< 3, -5 \right>$, $\nu=\left< 2, 5\right>$, and $\gamma=\left< 2,3 \right>.$ In addition, $\delta$ is not in $\widehat{\Upsilon}_{[\Qg]}$ but $\delta_0=0$.
Hence $\gdist_{[\Qg]}(\alpha,\beta)=2$  since $(\alpha,\beta)\succ^\tb_{[\Qg]}(\mu,\nu)\succ^\tb_{[\Qg]}\gamma.$

\item Let $\alpha=\left< 2, -4 \right>$ and $\beta= \left< 3, 5 \right>$. Then $(\alpha,\beta)$ is Type $I$ and we have
$\mu= \left< 3, -4 \right>$, $\nu=\left< 2, 5\right>$, $\gamma=\left< 2,3\right>$ and $\delta$ is not
in $\widehat{\Upsilon}_{[\Qg]}$. Hence $\gdist_{[\Qg]}(\alpha,\beta)=1$ since $(\alpha,\beta)\succ^\tb_{[\Qg]}(\mu,\nu).$

\item Let $\alpha=\left< 2, -4 \right>$ and $\beta= \left< 1,3 \right>$. Then $(\alpha,\beta)$ is Type $II$
and we have $\delta= \left< 3, -4 \right>$, $\gamma'=\left< 1,2\right>$. Hence $\gdist_{[\Qg]}(\alpha,\beta)=1$ since $(\alpha,\beta)\succ^\tb_{[\Qg]}(\delta,\gamma').$

\item Let $\alpha=\left< 2, -4 \right>$ and $\beta= \left< 1,4 \right>$. Then $(\alpha,\beta)$ is Type $II$ and we have $\gamma'=\left< 1,2\right>$ and $\delta_0=0$.
 Hence $\gdist_{[\Qg]}(\alpha,\beta)=1$ since $(\alpha,\beta)\succ^\tb_{[\Qg]}\gamma'.$
\end{enumerate}
\end{example}

\section{Distance polynomial and Folded distance polynomial} \label{Sec:distancePoly}
In this section, we briefly review the distance polynomials defined on the adapted cluster point $\lf Q \rf$ of type $ADE_m$, which was studied in \cite{Oh15E}.
Then we introduce and study the folded distance polynomials as in \cite{OS16},
which are well-defined on the twisted adapted cluster point $\lf \Qg \rf$ of type $D_{n+1}$.
The folded distance polynomials have an interesting relation with the quantum affine algebra of
type $C^{(1)}_{n}$. This section can be understood as a twisted analogue of \cite[Section 6]{Oh15E}. In this section, we refer to and follow
\cite[Section 6]{Oh15E} and \cite{Kas02} instead of introducing the notions for quantum affine algebras, their integrable representations and denominator formulas.

\bigskip

Let us take a base field $\ko$ the algebraic closure of $\C(q)$ in $\cup_{m >0} \C((q^{1/m}))$.

\subsection{Distance polynomial}

\begin{definition} \cite[Definition 6.11]{Oh15E}
For a Dynkin quiver $Q$, indices $k,l \in I$ and an integer $t \in \mathbb{N}$, we define the subset $\Phi_{Q}(k,l)[t]$ of $\PR \times \PR$ as follows:

A pair $(\alpha,\beta)$ is contained in $\Phi_{Q}(k,l)[t]$ if $\alpha \prec_Q \beta$ or $\be \prec_Q \al$ and
$$\{ \widetilde{\phi}_Q(\al),\widetilde{\phi}_Q(\be) \}=\{ (k,a), (l,b)\} \quad \text{ such that } \quad |a-b|=t.$$
\end{definition}

\begin{lemma} \cite[Lemma 6.12]{Oh15E} \label{lem: o well}
For any  $(\alpha^{(1)},\beta^{(1)})$ and $(\alpha^{(2)},\beta^{(2)})$ in $\Phi_{Q}(k,l)[t]$, we have
$$ o^{Q}_t(k,l) := \gdist_Q(\alpha^{(1)},\beta^{(1)})=\gdist_{Q}(\alpha^{(2)},\beta^{(2)}). $$
\end{lemma}

We denoted by $Q^{{\rm rev}}$ the quiver obtained
by reversing all arrows of $Q$ and $Q^*$ the quiver obtained from $Q$ by replacing vertices $i$ of $Q$ with  $i^*$
(see \eqref{eq: involution D} for $^*$ of type $D_{n+1}$).

\begin{definition}\cite[Definition 6.15]{Oh15E} \label{def: Dist poly Q}
For $k,l \in I$ and a Dynkin quiver $Q$, we define a polynomial $D^Q_{k,l}(z) \in \ko[z]$ as follows:
$$D^Q_{k,l}(z) \seteq  \prod_{ t \in \Z_{\ge 0} } (z-(-1)^t q^{t} )^{\mathtt{o}^{\overline{Q}}_t(k,l)},$$
where
\begin{align} \label{def: non-fold O}
\mathtt{o}^{\overline{Q}}_t(k,l) \seteq  \max( o^{Q}_t(k,l),o^{Q^\rev}_t(k,l) ).
\end{align}
\end{definition}

\begin{proposition} \cite[Proposition 6.16]{Oh15E} \label{prop: DQ DQ'}
For $k,l \in I$ and any Dynkin quivers $Q$ and $Q'$, we have
\begin{enumerate}
\item[{\rm (a)}] $D^Q_{k,l}(z)=D^Q_{l,k}(z)=D^Q_{k^*,l^*}(z)=D^Q_{l^*,k^*}(z)$.
\item[{\rm (b)}] $D^Q_{k,l}(z)=D^{Q'}_{k,l}(z)$.
\end{enumerate}
Hence $D_{k,l}(z)$ is well-defined for $\lf Q \rf$.
\end{proposition}

The denominator formulas $d^{\g}_{k,l}(z)=d^{\g}_{l,k}(z)$ $(1 \le k ,l \le n+1)$ for $U_q'(\g)$ ($\g=A^{(1)}_{n}$ or $D^{(1)}_{n+1}$) were computed in
\cite{AK,KKK13b}:

\begin{theorem} \label{thm: denominator of untwisted} \hfill
\begin{enumerate}
\item[{\rm (a)}] For $\g = A^{(1)}_{n}$ $(n \ge 2)$ and $1 \le k , l \le n$, we have
\begin{align} \label{eq: denominator A1}
d_{k,l}^{\;A^{(1)}_{n}}(z) = \prod_{s=1}^{\min(k,l,n+1-k,n+1-l)} (z-(-q)^{|k-l|+2s}).
\end{align}
\item[{\rm (b)}] For $\g = D^{(1)}_{n+1}$ $(n \ge 3)$ and $1 \le k , l \le n+1$, we have
\begin{equation} \label{eq: denominator D1}
\begin{aligned}
& d_{k,l}^{\; D^{(1)}_{n+1}}(z)= \begin{cases}
\displaystyle \prod_{s=1}^{\min (k,l)} (z-(-q)^{|k-l|+2s}) \prod_{s=1}^{\min (k,l)} (z-(-q)^{2n-k-l+2s}) & \text{ if } 1 \le k,l \le n-1, \allowdisplaybreaks \\
\ \displaystyle \prod_{s=1}^{k}(z-(-q)^{n-k+2s}) &  \hspace{-15ex}\text{ if }  1 \le k \le n-1 \text{ and } l \in \{ n, n+1\}, \allowdisplaybreaks \\
\ \displaystyle \prod_{s=1}^{\lfloor \frac{n}{2} \rfloor} (z-(-q)^{4s}) &  \text{ if }   k \neq l \in \{n,n+1\},  \allowdisplaybreaks  \\
\ \displaystyle \prod_{s=1}^{\lfloor \frac{n+1}{2} \rfloor} (z-(-q)^{4s-2}) &  \text{ if }  k=l \in \{ n,n+1\}.
 \end{cases}
\end{aligned}
\end{equation}
\end{enumerate}
\end{theorem}

\begin{theorem} \cite[Theorem 6.18]{Oh15E}\label{eq: dist denom 1A2n}
For any adapted class $[Q]$ of type $AD_m$, the denominator formulas for the quantum affine algebra $U'_q(\g)$ $(\g=A^{(1)}_{m}$ or $D^{(1)}_{m})$ can be read
from $\Gamma_Q$ and $\Gamma_{Q^\rev}$ as follows:
\begin{align*}
d^{\g}_{k,l}(z) & = D_{k,l}(z) \times (z-(-q)^{\mathtt{h}^\vee})^{\delta_{l,k^*}}
\end{align*}
where $\mathtt{h}^\vee$ is the dual Coxeter number corresponding to $Q$.
\end{theorem}

\subsection{Folded distance polynomial}
Now we fix the folded index set $$ \ov{I}=\{ 1,2,\ldots, n \}$$
which is induced from $\vee$ in \eqref{eq: C_n}.

\begin{definition}
For a folded AR-quiver $\wUp_{[\Qg]}$, indices $\ov{k},\ov{l} \in \ov{I}$ and an integer $t \in \mathbb{N}$,
we define the subset $\Phi_{[\Qg]}(\ov{k},\ov{l})[t]$ of $\PR \times \PR$ as follows:

A pair $(\alpha,\beta)$ is contained in $\Phi_{[\Qg]}(\ov{k},\ov{l})[t]$ if $\alpha \prec_{[\Qg]} \beta$ or $\be \prec_{[\Qg]} \al$ and
$$\{ \widehat{\phi}_{[\Qg]}(\al),\widehat{\phi}_{[\Qg]}(\be) \}=\{ (\ov{k},a), (\ov{l},b)\} \quad \text{ such that } \quad |a-b|=t.$$
\end{definition}

By Corollary \ref{cor: do not depend}, the following notion is well-defined:

\begin{definition}
For any  $(\alpha^{(1)},\beta^{(1)})\in \Phi_{[\Qg]}(\ov{k},\ov{l})[t]$, we define
$$ o^{[\Qg]}_t(\ov{k},\ov{l}) := \gdist_{[\Qg]}(\alpha^{(1)},\beta^{(1)}). $$
\end{definition}

Recall the notations on $C_{n}$ in Definition \ref{def: Cn}.

\begin{definition} \label{def: Dist poly Q twisted}
For $\ov{k},\ov{l} \in \ov{I}$ and a folded AR-quiver $\wUp_{[\Qg]}$, we define a polynomial $\widehat{D}^{[\Qg]}_{\ov{k},\ov{l}}(z) \in \ko[z]$ as follows:
$$\widehat{D}^{[\Qg]}_{\ov{k},\ov{l}}(z) \seteq \prod_{ t \in \Z_{\ge 0} } (z- (-q_s)^{t})^{\mathtt{o}^{[\Qg]}_t(\ov{k},\ov{l})} ,$$
where
$$ q_s^{\ov{\mathsf{d}}}=q_s^2=q \quad \text{ and } \quad  \mathtt{o}^{[\Qg]}_t(\ov{k},\ov{l}) \seteq
\left\lceil \dfrac{o^{[\Qg]}_t(\ov{k},\ov{l})}{ \ov{\mathsf{d}} } \right\rceil.$$
\end{definition}

\begin{proposition} \label{prop: DQ DQ' twisted}
For $\ov{k},\ov{l} \in \ov{I}$ and any twisted adapted classes $[{Q_1^{\gets}}]$ and $[{Q_2^{\gets}}]$ in $\lf \Qg \rf$, we have
$$\widehat{D}^{[{Q_1^{\gets}}]}_{\ov{k},\ov{l}}(z)=\widehat{D}^{[{Q_2^{\gets}}]}_{\ov{k},\ov{l}}(z).$$
\end{proposition}

\begin{proof}
It is an easy consequence which can be obtained from Corollary \ref{cor: do not depend} and the fact that $\wUp_{[\ii_0]}$ has $n+1$-vertices in each folded residue.
\end{proof}

From the above proposition, we can define {\it the folded distance polynomial} $\widehat{D}_{\ov{k},\ov{l}}(z)$ for $\lf \Qg \rf$ canonically.

\begin{theorem} \cite{AK} \hfill \label{thm: denom 1}
$$d^{C^{(1)}_{n}}_{k,l}(z)= \displaystyle \prod_{s=1}^{ \min(k,l,n-k,n-l)}
\big(z-(-q_s)^{|k-l|+2s}\big)\prod_{i=1}^{ \min(k,l)} \big(z-(-q_s)^{2n+2-k-l+2s}\big) \quad 1 \le k,l\le n$$
\end{theorem}

\begin{theorem} \label{eq: dist denom}
For any twisted adapted class $[\ii_0]$, the denominator formulas for $U'_q(C_n^{(1)})$ can be read
from $\wUp_{[\ii_0]}$ as follows:
\begin{align*}
d^{C^{(1)}_{n}}_{\ov{k},\ov{l}}(z) & =\widehat{D}_{\ov{k},\ov{l}}(z) \times (z-q^{\mathtt{h}^\vee})^{\delta_{\ov{l},\ov{k}}}
\end{align*}
where $\mathtt{h}^\vee=n+1$ is the dual Coxeter number of $C_{n}$.
\end{theorem}

\begin{proof}
Note that, for $1 \le k,l \le n$, one can observe that
\begin{eqnarray} &&
\parbox{81ex}{
\begin{itemize}
\item[{\rm (i)}] the first factor of $d^{C^{(1)}_{n}}_{k,l}(z)$ is the same as $d^{A^{(1)}_{n-1}}_{k,l}(z)$,
\item[{\rm (ii)}] the second factor of $d^{C^{(1)}_{n}}_{k,l}(z)$ is the same as the second factor of $d^{D^{(1)}_{n+2}}_{k,l}(z)$.
\end{itemize}
}\label{eq: CDCD}
\end{eqnarray}
Thus we can apply the same argument of \cite[Theorem 6.18]{Oh15E}. More precisely, {\rm (i)} is induced from (II-1) and (II-2) in \eqref{eq: complacted socle of D}
and {\rm (ii)} is induced from (I-1), (I-2), (I-3) and (I-6) in \eqref{Pic:gdist_D_1}.

Alternatively, it suffices to consider a particular folded AR-quiver $\wUp_{[\Qg]}$ by Proposition
\ref{prop: DQ DQ' twisted}. For the following $Q^{\gets n}$, using Algorithm \ref{Alg:label} and \cite[Remark 1.14]{Oh14A}, one can label $\wUp_{[Q^{\gets n}]}$ and check our assertion
(see Example \ref{ex: folded}):
$$  Q^{\gets n}= \xymatrix@R=4ex{
*{\circ}<3pt> \ar@{<-}[r]_<{1 \ \ }  &*{\circ}<3pt>
\ar@{<-}[r]_<{2 \ \ } & *{\circ}<3pt>
\ar@{<-}[r]_<{3 \ \ } &*{\circ}<3pt>\ar@{<.}[r]_<{4 \ \ }& *{\circ}<3pt>\ar@{<.}[r]_<{n-1\ \ } &*{\odot}<3pt> \ar@{->}[l]^<{\  \ {n \choose n+1}  }}
$$
\end{proof}

\section{Dorey's rule for $U_q'(C^{(1)}_{n})$} \label{Sec:Dorey}

In \cite[Section 9]{OS16}, the authors proved that Dorey's rule for $U_q'(B^{(1)}_{n+1})$ (\cite{CP96}) can be interpreted the
$[\ii_0]$-minimal pairs $(\al,\be)$ of $\ga \in \PR_{A_{2n+1}}$ when $[\ii_0]$ is a twisted adapted class of $A_{2n+1}$.
In this section, we shall do such an analogue for $U_q'(C^{(1)}_{n})$ by using a twisted adapted class $[\Qg]$ of type $D_{n+1}$
and its folded AR-quiver $\wUp_{[\Qg]}$.

\vskip 3mm

Recall the quantum affine algebra $U_q'(C^{(1)}_{n})$ is
the associative algebra with $1$ generated by $e_i$, $f_i$ (Chevalley generators), $K_i=q^{h_i}$ $(i \in \overline{I} \sqcup \{0\})$ subject to certain relations
(see \cite{Oh15E} for more detail).

\begin{proposition} \cite[Theorem 8.2]{CP96}
Let $(\ov{i},x)$, $(\ov{j},y)$, $(\ov{k},z) \in \ov{I} \times \ko^\times$. Then
$$ {\rm Hom}_{U_q'(C^{(1)}_{n})}\big( V(\varpi_{\ov{j}})_y \otimes V(\varpi_{\ov{i}})_x , V(\varpi_{\ov{k}})_z  \big) \ne 0 $$
if and only if one of the following conditions holds:
\begin{eqnarray}&&
\left\{\hspace{1ex}\parbox{75ex}{
$\ell \seteq \max(\ov{i},\ov{j},\ov{k}) \le n$, $s+m =\ell$ for $\{ s,m \} \seteq \{ \ov{i},\ov{j},\ov{k}\} \setminus \{ \ell \}$
and
$$ \left( y/z,x/z \right) =
\begin{cases}
\big( (-q_s)^{-\ov{i}},(-q_s)^{\ov{j}} \big), & \text{ if } \ell = \ov{k},\\
\big( (-q_s)^{\ov{i}-(2n+2)},(-q_s)^{\ov{j}} \big), & \text{ if } \ell = \ov{i},\\
\big( (-q_s)^{-\ov{i}},(-q_s)^{2n+2-\ov{j}}  \big), & \text{ if } \ell = \ov{j}.
\end{cases}
$$
}\right. \label{eq: Dorey C}
\end{eqnarray}
Here,
\begin{itemize}
\item $V(\varpi_i)$ is the unique simple $U_q'(C_{n}^{(1)})$-module which is finite dimensional integrable with its dominant weight $\varpi_i$ $(i \in \ov{I})$
$($see \cite{Kas02} for more detail$)$,
\item $V(\varpi_i)_x$ is $V(\varpi_i)$ as a vector space with the actions of $e_i$, $f_i$, $K_i$ $(i \in \overline{I} \sqcup \{0\})$ replaced
with $x^{\delta_{i0}}e_i$, $x^{-\delta_{i0}}f_i$, $K_i$, respectively.
\end{itemize}
\end{proposition}

\begin{definition} \label{def: [ii0] module category}
For any positive root $\beta$ contained in $\Phi^+_{D_{n+1}}$, we set the $U_q'(C^{(1)}_{n})$-module $V_{[\Qg]}(\beta)$ defined as follows:
\begin{align} \label{eq: Vii0(beta)}
V_{[\Qg]}(\beta) \seteq V(\varpi_{\ov{i}})_{(-q_s)^{p}} \quad \text{ where } \quad \widehat{\phi}(\beta)=({\ov{i}},p).
\end{align}
We define the smallest abelian full subcategory $\mathscr{C}_{[\Qg]}$ consisting of finite dimensional integrable $U_q'(C^{(1)}_{n})$-modules such that
\begin{itemize}
\item[{\rm (a)}] it is stable by taking subquotient, tensor product and extension,
\item[{\rm (b)}] it contains $V_{[\Qg]}(\beta)$ for all $\beta \in \Phi^+$.
\end{itemize}
\end{definition}

\begin{theorem}
Let $(\ov{i},x)$, $(\ov{j},y)$, $(\ov{k},z) \in \ov{I} \times \ko^\times$. Then
$$ {\rm Hom}_{U_q'(C^{(1)}_{n})}\big( V(\varpi_{\ov{j}})_y \otimes V(\varpi_{\ov{i}})_x , V(\varpi_{\ov{k}})_z  \big) \ne 0 $$
if and only if there exists a twisted adapted class $[\Qg]$ and $\al,\be,\ga \in \Phi^+$ such that
\begin{itemize}
\item[{\rm (i)}] $(\al,\be)$ is a $[\Qg]$-minimal pair of $\ga$,
\item[{\rm (ii)}] $V(\varpi_{\ov{j}})_y  = V_{[\Qg]}(\be)_t, \ V(\varpi_{\ov{i}})_x  = V_{[\Qg]}(\al)_t, \ V(\varpi_{\ov{k}})_z  = V_{[\Qg]}(\ga)_t$
for some $t \in \ko^\times$.
\end{itemize}
\end{theorem}

\begin{proof}
By comparing \eqref{eq: Dorey folded coordinate Cn} and \eqref{eq: Dorey C},
our assertion is an immediate consequence of \eqref{eq: Vii0(beta)} in Definition \ref{def: [ii0] module category}.
\end{proof}

\begin{corollary} The condition {\rm (b)} in {\rm Definition \ref{def: [ii0] module category}} can be restated as follows:
\begin{itemize}
\item[{\rm (b)}$'$] It contains $V_{[\Qg]}(\alpha_k)$ for all $\alpha_k \in \Pi$.
\end{itemize}
\end{corollary}

\appendix
\section{Triply twisted Dynkin quiver} \label{Appen_A}

In \cite[Appendix]{OS16}, the authors showed that there exist a unique $\vee$-foldable cluster $[\ii_0]$ and a unique $\vee^2$-foldable cluster $[\jj_0]$
for $\vee$ in \eqref{eq: G_2}:
$$ \ii_0 =\prod_{k=0}^{5} (2\ 1)^{k\vee} \quad \text{and} \quad \jj_0 =\prod_{k=0}^{5} (2\ 1)^{2k\vee}.$$

In this appendix, we will introduce a triply twisted Dynkin quivers which can be analogues of
twisted Dynkin quivers in Section \ref{sec: Twisted Dynkin quiver and Reflection}.

\begin{definition}\label{Def:triplytwistedQ}
A {\it $\vee$-triply twisted Dynkin quiver} $Q^{\gets}_{(i,j,k)}$ (resp. {\it $\vee^2$-triply twisted Dynkin quiver} $Q^{\gets}_{(i',j',k')}$) of $D_{4}$ is obtained by giving
an orientation to each edge to following diagrams:
\begin{align*}
 \xymatrix@R=4ex{
 *{\circledcirc}<3pt>\ar@{.}[r]_<{\ \scriptstyle{(i,j,k)} \ \ \ } &*{\circ}<3pt> \ar@{-}[l]^<{\  \ 2}}.  \quad \text{ and } \quad
  \xymatrix@R=4ex{
 *{\circledast}<3pt>\ar@{.}[r]_<{\ \scriptstyle{(i',j',k')} \ \ \  } &*{\circ}<3pt> \ar@{-}[l]^<{\  \ 2}}.
\end{align*}
where $\{i,i^\vee=j,i^{2\vee}=k\}=\{1,3,4\}=\{i',i'^{2\vee}=j',i'^{4\vee}=k'\}$.
\end{definition}

Then we can complete analogues of \eqref{eq: Corres_twisted} for $\vee$ and $\vee^2$ by applying the similar arguments
in Section \ref{sec: Twisted Dynkin quiver and Reflection}:
\begin{itemize}
\item For each triply twisted adapted class $[\ii'_0]$, there exists a unique $Q^{\gets}_{(i,j,k)}$ or
$Q^{\gets}_{(i',j',k')}$, the unique expression $\ii'_0$ in $[\ii'_0]$ is {\it adapted} to the triply twisted Dynkin quiver.
\item For each triply twisted adapted class $[\ii'_0]$, there exists a unique triply twisted Coxeter element
$a \ b \ \vee  = \phi_{Q^{\gets}_{(i,j,k)}} \vee$ or $a \ b \ \vee^2 =\phi_{Q^{\gets}_{(i',j',k')}} \vee^2$ such that
$$[\ii'_0] = \prod_{k=0}^{5} (a\ b)^{k\vee} \text{ or } [\ii'_0] = \prod_{k=0}^{5} (a\ b)^{2k\vee}.$$
\end{itemize}

\begin{align*}
& \raisebox{3.6em}{\xymatrix@C=8ex@R=4ex{
& \{ \prec_{[Q^{\gets}_{(i,j,k)}]} \} \ar@{<->}[dl]_{1-1} \ar@{<->}[d]_{1-1} \ar@{<->}[dr]^{1-1}  \\
\{ [Q^{\gets}_{(i,j,k)}] \} \ar@{<->}[r]^{1-1} & \{ \phi_{Q^{\gets}_{(i,j,k)}} \} \ar@{<->}[r]^{1-1} & \{ Q^{\gets}_{(i,j,k)} \} \\
& \{ \Upsilon_{[Q^{\gets}_{(i,j,k)}]} \} \ar@{<->}[ul]^{1-1}\ar@{<->}[ur]_{1-1} \ar@{<->}[u]^{1-1}}} \text{for $[Q^{\gets}_{(i,j,k)}] \in \lf \ii_0 \rf$} \\
&\raisebox{3.6em}{\xymatrix@C=8ex@R=4ex{
& \{ \prec_{[Q^{\gets}_{(i',j',k')}]} \} \ar@{<->}[dl]_{1-1} \ar@{<->}[d]_{1-1} \ar@{<->}[dr]^{1-1}  \\
\{ [Q^{\gets}_{(i',j',k')}] \} \ar@{<->}[r]^{1-1} & \{ \phi_{Q^{\gets}_{(i',j',k')}} \} \ar@{<->}[r]^{1-1} & \{ Q^{\gets}_{(i',j',k')} \} \\
& \{ \Upsilon_{[Q^{\gets}_{(i',j',k')}]} \} \ar@{<->}[ul]^{1-1}\ar@{<->}[ur]_{1-1} \ar@{<->}[u]^{1-1}}} \text{for $[Q^{\gets}_{(i',j',k')}] \in \lf \jj_0 \rf$}
\end{align*}


\begin{thebibliography}{99}

\bibitem{AK} T. Akasaka and M. Kashiwara,
{\it Finite-dimensional representations of quantum affine algebras},
Publ. Res. Inst. Math. Sci., {\bf33} (1997), 839-867.

\bibitem{ARS} M. Auslander, I. Reiten and S. Smalo, {\it Representation theory of Artin algebras}, Cambridge studies in advanced
mathematics {\bf 36}, Cambridge 1995.

\bibitem{ASS} I. Assem, D. Simson and A. Skowro\'{n}ski, {\it Elements of the representation theory of associative algebras. Vol.1}, London
Math. Soc. Student Texts {\bf 65}, Cambridge 2006.



\bibitem{B99}
R. Bedard, {\it On commutation classes of reduced words in Weyl
groups}, European J. Combin. {\bf 20} (1999), 483-505.

\bibitem{Ber91} Bernard, D., LeClair, {\it A quantum group symmetries and non-local currents in 2D QFT},
Commun. Math. Phys. {\bf 142} (1991), 99-138

\bibitem{Bour}
N.~Bourbaki.
\newblock {\em \'{E}l\'ements de math\'ematique. {F}asc. {XXXIV}. {G}roupes et
  alg\`ebres de {L}ie. {C}hapitres {IV}--{VI}}.
\newblock Actualit\'es Scientifiques et Industrielles, No. 1337. Hermann,

%\bibitem{BK09} J. Brundan and A. Kleshchev, Blocks of cyclotomic Hecke algebras and Khovanov-Lauda algebras, Invent.
%Math., 178 (2009) 451Ð484.

\bibitem{CP96} V. Chari and A. Pressley, {\it Yangians, integrable quantum systems and Dorey's rule}, Comm. Math. Phys. {\bf 181} (1996),
no. 2, 265-302.

\bibitem{Dor91} P.E. Dorey, Root systems and purely elastic S-matrices, Nucl. Phys. {\bf B358}, 654-676 (1991)

\bibitem{FH11} E. Frenkel and D. Hernandez, Langlands duality for finite-dimensional representations of quantum affine algebras, Lett. Math. Phys. {\bf 96} (2011) 217-261.

\bibitem{Gab80} P. Gabriel, {\it Auslander-Reiten sequences and Representation-finite algebras}, Lecture notes in
Math., vol. 831, Springer-Verlag, Berlin and New York, (1980), 1-71.

\bibitem{H10} D. Hernandez, Kirillov-Reshetikhin conjecture: the general case, Int. Math. Res. Not. {\bf 1} (2010) 149-193.

\bibitem{HL11} D. Hernandez and B. Leclerc, {\it Quantum Grothendieck rings and derived Hall algebras},
arXiv:1109.0862v2 [math.QA], to appear in J. Reine Angew. Math.

\bibitem{Kac} V. Kac,
\newblock{\it Infinite dimensional Lie algebras},
\newblock{3rd ed., Cambridge University Press, Cambridge, 1990}.

\bibitem{KKK13b}
S.-J. Kang, M. Kashiwara and M. Kim,  {\it Symmetric quiver Hecke algebras and R-matrices of
quantum affine algebras II}, Duke Math. J. {\bf 164}, (2015), 1549-1602.

\bibitem{Kas02}
M.~Kashiwara, {\it On level zero representations of quantum affine
algebras}, Duke. Math. J. {\bf112} (2002), 117-175.

\bibitem{KO16} M. Kashiwara and S-j. Oh, Categorical relations between Langlands dual quantum affine
algebras: Doubly laced types, in preparation.

\bibitem{KKK13} S.-J. Kang, M. Kashiwara and M. Kim, Symmetric quiver Hecke algebras and R-matrices of quantum
affine algebras, arXiv:1304.0323 [math.RT].

\bibitem{Mc12}
P. McNamara, {\it Finite dimensional representations of Khovanov-Lauda-Rouquier algebras
I: finite type},  arXiv:1207.5860 [math.RT], to appear in J. Reine Angew. Math.

\bibitem{Oh14}
S-j. Oh, {\it The Denominators of normalized R-matrices of
types $A_{2n-1}^{(2)}$, $A_{2n}^{(2)}$, $B_{n}^{(1)}$ and
$D_{n+1}^{(2)}$}, Publ. Res. Inst. Math. Sci. \textbf{51} (2015), 709-744.

\bibitem{Oh14A}
\bysame, {\it Auslander-Reiten quiver of type A and generalized
quantum affine Schur-Weyl duality},  arXiv:1405.3336v3 [math.RT], to appear in Trans. Amer. Math. Soc.

\bibitem{Oh14D}
\bysame, {\it Auslander-Reiten quiver of type D and generalized quantum affine Schur-Weyl duality}, arXiv: 1406.4555v3 [math.RT]

\bibitem{Oh15E}
\bysame, {\it Auslander-Reiten quiver and representation theories
related to KLR-type Schur-Weyl duality}, arXiv:1509.04949 [math.RT].

\bibitem{OS15}
S-j. Oh and  U. Suh, {\it Combinatorial Auslander-Reiten quivers and reduced expressions}, arXiv:1509.04820 [math.RT].

\bibitem{OS16}
S-j. Oh and  U. Suh, {\it Twisted Coxeter elements and folded AR-quivers via Dynkin diagram automorphisms: I}, in preparation.

\bibitem{Spr74}T. A. Springer, {\it Regular elements of finite reflection
groups},  Inv. Math. {\bf 25} (1974), 159-198

\bibitem{R96}
C.M. Ringel, {\it PBW-bases of quantum groups}, J. Reine Angew. Math.
{\bf 470} (1996), 51-88.

\bibitem{R80}
C.M. Ringel, {\it Tame algebras, Proceedings}, ICRA 2, Springer LNM 831, (1980), 137-287.





\end{thebibliography}
\end{document}